\documentclass{amsart}
\usepackage{amsfonts}
\usepackage{amsmath}
\usepackage{comment}
\usepackage{fullpage}
\usepackage[utf8]{inputenc}
\usepackage[english]{babel}
\usepackage{amsthm}
\usepackage{xcolor}
\usepackage{tikz}
\usepackage{cite}
\usepackage{multicol}
\usetikzlibrary{shapes, patterns}
\usepackage{appendix}
\usepackage[foot]{amsaddr}

\newcommand{\Q}{{\mathbb Q}}
\newcommand{\R}{{\mathbb R}}

\newcommand{\Z}{{\mathbb Z}}
\newcommand{\B}{{\mathcal B}}

\newcommand{\M}{{\mathcal M}_{I,d}}
\newcommand{\Mp}{{\mathcal M}_{I',d}}
\newcommand{\N}{{\mathbb N}}
\newcommand{\C}{{\mathcal C}_I}
\newcommand{\D}{{\mathcal{D}_I}}
\newcommand{\Cp}{{\mathcal{C}'_I}}
\newcommand{\Dp}{{\mathcal{D}'_I}}

\newtheorem{theorem}{Theorem}[section]
\newtheorem{lemma}[theorem]{Lemma}
\newtheorem{cor}[theorem]{Corollary}
\newtheorem*{defin}{Definition}
\newtheorem*{exam}{Examples}

\title{Cohomology of Polychromatic Configuration Spaces of $\R^d$}
\author{Nicholas Kosar}
\address{University of Illinois at Urbana-Champaign}
\email{kosar2@illinois.edu}
\date{}

\begin{document}
\maketitle

\begin{abstract}
Recently, the homology and cohomology of non-$k$-overlapping discs, or, equivalently, no $k$-equal subspaces of Euclidean space, were calculated by Dobrinskaya and Turchin. We calculate the homology and cohomology of two classes of more general spaces: decreasing polychromatic configuration spaces and bicolored configuration spaces. Instead of all points behaving similarly, we allow for varying behavior between points.
\end{abstract}

\section{Introduction}
For any topological space $X$, the $n^{th}$ no $k$-equal space of $X$ is the space of $n$ points on $X$ such that no $k$ of them are all equal. When $k=2$, these are configuration spaces. The study of no $k$-equal spaces started with work in complexity theory by Bj\"{o}rner, Lov\'{a}sz, and Yao \cite{BLY}. Further work by Bj\"{o}rner and Lov\'{a}sz begged the question: what are the Betti numbers of the no $k$-equal spaces of $\R$ \cite{BL}. Bj\"{o}rner and Welker answered this question \cite{BW}. However, their work gave no information on the cohomology rings. Further work by various authors determined the cohomology rings of the no $k$-equal spaces of $\R^2$ \cite{Yuz}, $\R$ \cite{YB2}, and, finally, $\R^d$ for all $d \geq 1$ \cite{DT}. Baryshnikov and Dobrinskaya-Turchin give explicit geometric representatives for homology.

We work with a generalization of no $k$-equal spaces: polychromatic configuration spaces. As with the no $k$-equal spaces, polychromatic configuration spaces of a topological space $X$ arise from points on $X$. Instead of removing all subspaces where $k$ points are equal, we remove subspaces corresponding to a set $I \subset \N^m$ for some number $m$. When $m=1$, these are the no $k$-equal spaces. Baryshnikov determined the generating function for the Euler characteristic of polychromatic configuration spaces whenever $X$ is a compact definable set in some o-minimal structure \cite{YB1}.

We compute the cohomology rings of two classes of polychromatic configuration spaces of $\R^d$: decreasing polychromatic configuration spaces and bicolored configuration spaces $(m=2)$. The rest of the paper is outlined as follows: in section \ref{prelim}, we give a rigorous definition of our spaces. We also define notation that will be used throughout the paper as well as a left action of the configuration spaces of $\R^d$ on the polychromatic configuration spaces of $\R^d$. In section \ref{config}, we review relevant background information regarding the homology of the configuration spaces of $n$ points in $\R^d$. In section \ref{decreasing}, we discuss decreasing polychromatic configuration spaces of $\R^d$. The homology and cohomology of these spaces are discussed in sections \ref{hom_1} and \ref{cohom_1}, respectively. Similar to Dobrinskaya and Turchin \cite{DT}, the cohomology ring is described as a space of forests. We use this description of the cohomology to prove that the generating set for homology found in section \ref{hom_1} is a basis. In section \ref{bicolored} we discuss bicolored configuration spaces of $\R^d$. The homology and cohomology of these spaces are discussed in sections \ref{hom_2} and \ref{cohom_2}, respectively. Again, the cohomology ring is described as a space of forests. In section \ref{tricolored}, we have a discussion of general polychromatic configuration spaces which are not decreasing. Finally, in section \ref{further}, we will consolidate all the unanswered questions discussed in this paper.

I would like to thank Yuliy Baryshnikov for extremely helpful conversations.

\section{Preliminaries} \label{prelim}

Throughout the paper, we include $0$ as a natural number.

\begin{defin}
Let $m \in \N^{>0}, I \subset \N^m$. $I$ is called \emph{ideal} if $(n_1, \ldots, n_m) \in I$ and $n'_i \leq n_i$ for all $i$, implies $(n'_1, \ldots, n'_m) \in I$
\end{defin}

Ideals will be used to describe the interaction between points of various types. For each ideal, $I$, let $\B_{I, d}(n_1, \ldots, n_m)$ denote the space of labeled discs, $n_i$ of color $i$, satisfying the following property: for all $(n_1, \ldots, n_m) \notin I$, any intersection containing $n_i$ discs of color $i$ for each $i$ is empty.

Let $\B_d$ denote the operad of little $d$-discs. There exists a left action of $\B_d$ on $\B_{I, d}$:
\begin{equation*}
\B_d(r) \times \B_{I, d}({\vec{n}_1}) \times \ldots \times \B_{I, d}(\vec{n}_r) \to \B_{I, d}(\vec{n}_1+\ldots + \vec{n}_r)
\end{equation*}

where the $i^{th}$ disc in $\B_d(r)$ is replaced by the configuration of discs from $\B_{I, d}(\vec{n}_i)$. One can also define a right action; however, it will not be necessary for this paper.

The K\"{u}nneth Theorem for homology gives a map from $H_* \B_d(r) \times H_* \B_{I, d}(\vec{n}_1) \times \ldots \times H_* \B_{I, d}(\vec{n}_r)$ to $H_* (\B_d(r) \times \B_{I, d} (\vec{n}_1) \times \ldots \times \B_{I, d}(\vec{n}_r))$. Combining this with the induced map on homology from the above action gives an action on homology groups
\begin{equation*}
H_* \B_d(r) \times H_* \B_{I, d}(\vec{n}_1) \times \ldots \times H_* \B_{I, d}(\vec{n}_r) \to H_* \B_{I, d}(\vec{n}_1+\ldots + \vec{n}_r)
\end{equation*}

Suppose $f: X \to \B_d (r)$ is a simplicial map representing $[\alpha] \in H_* \B_d(r)$ and $f_i : X_i \to B_{I, d} (\vec{n}_i)$ are simplicial maps representing $[\alpha_i] \in H_* \B_{I, d} (\vec{n}_i)$. Then $g : X \times X_1 \times \ldots \times X_r \to \B_{I, d} (\vec{n}_1 + \ldots + \vec{n}_r)$ given by $g(x, x_1, \ldots, x_r) = f(x) \cdot (f_1(x_1), \ldots, f_r(x_r))$ is a map representing $[\alpha] \cdot ([\alpha_1], \ldots, [\alpha_r])$ where $\cdot$ denotes the two actions described above.

The space $\B_{I, d}(n_1, \ldots n_m)$ is homotopy equivalent to a similar space of points.

\begin{defin}
Let $I \subset \N^m$ be an ideal. Let $\vec{n} = (n_1, \ldots, n_m)$. The $\vec{n}$ \emph{polychromatic configuration space} of $\R^d$ corresponding to $I$ is the space of labeled points, $n_i$ of color $i$ for all $i$, such that for all $(\ell_1, \ldots, \ell_m) \notin I$, any intersection containing $\ell_i$ points of color $i$ for all $i$ is empty. We denote this space by $\M(\vec{n})$.
\end{defin}

This space is the complement in $\R^{(n_1 + \ldots + n_m)d}$ to a linear subspace arrangement. We will denote the $i^{th}$ point of color $j$ by $x_i^j$.

\begin{lemma}
For all ideals $I \subset N^m$ and all $\vec{n} \in \N^m$, $\M(\vec{n})$ is homotopy equivalent to $\B_{I, d}(\vec{n})$.
\end{lemma}

A homotopy equivalence is given by taking the centers of the discs in the arrangement from $\B_{I, d}(\vec{n})$. Because $\M(\vec{n})$ is homotopy equivalent to $\B_{I, d}(\vec{n})$, the action of $H_* \B_d$ on $H_* \B_{I, d}$ gives an action of $H_* \mathcal{M}_d$ on $H_* \M$ where $\mathcal{M}_d(n)$ is the $n^{th}$ configuration space of $\R^d$.

As mentioned in the introduction, we will at times restrict to decreasing polychromatic configuration spaces.

\begin{defin}
Let $I \subset \N^m$ be an ideal. We call $I$ \emph{decreasing} if for all $i \leq m$, if $(n_1, \ldots, n_i, 0, \ldots 0) \notin I$, $n_i > 0$, and $(n_1, \ldots, n_{i-1}, n_{i} - 1, 0, \ldots 0) \in I$, then we have $(n_1, \ldots, n_j - 1, \ldots, n_i, 0, \ldots, 0) \in I$ for all $j < i$ with $n_j > 0$.

If $X$ is the polychromatic configuration space of a decreasing ideal, we call it \emph{decreasing}.
\end{defin}

The term decreasing comes from the function $f_I: \N^{m-1} \to \N \cup \{\infty\}$ defined by $f_I (n_1, \ldots, n_{m-1}) = \inf \{i | (n_1, \ldots, n_{m-1}, i) \in I \}$. The condition of $I$ being decreasing is equivalent to this function being strictly decreasing in each coordinate.

\begin{exam}
\begin{enumerate}

\item When $m = 1$, any ideal is decreasing. Thus, the no $k$-equal spaces are decreasing polychromatic configuration spaces.

\item Consider points of $m$ colors such that each color $c$ has an associated weight $w_c > 0$. Furthermore, we may assume that $w_{j-1} \geq w_{j}$ for all $j \leq m$. Let $M \in \R$. Let $Y \subset \R^{(n_1 + \ldots + n_m)d}$ be the space of arrangements of colored points in $\R^d$ satisfying the following property:
\begin{center}
for all $x \in \R^d$, $\displaystyle \sum_{(i, j) \in N_x} w_j < M$ where $N_x = \{(i, j) | x_i^j = x\}$
\end{center}

Then $Y$ is a decreasing polychromatic configuration space. $Y$ is a weighted analogue of no $k$-equal spaces which are obtained when all weights are one and $M = k$.
\end{enumerate}
\end{exam}

For each ideal, particular features are important.

\begin{defin}
Let $I$ be any ideal in $\N^m$. Call an $m$-tuple $\vec{n} \notin I$ \emph{critical} if $\vec{n} = (n_1, \ldots, n_i, 0, \ldots 0)$, $n_i > 0$ and for all $j \leq i$ with $n_j > 0$, $(n_1, \ldots, n_j - 1, \ldots, n_{i}, 0, \ldots 0) \in I$. Denote the set of critical $m$-tuples by $\C$. For each critical $m$-tuple, let its \emph{weight} be $\sum n_i$, denoted by $w_{\vec{n}}$.
\end{defin}

Before we state our main theorem regarding decreasing polychromatic configuration spaces, we have one more notational convention to define.

\begin{defin}
Let $\vec{e}_i$ denote the $m$-tuple with a one in the $i^{th}$ coordinate and zeros everywhere else. Let $E = \{\vec{e}_i, \ldots, \vec{e}_m$\}.
\end{defin}

Our theorem regarding decreasing polychromatic configuration spaces is as follows:

\begin{theorem}
\label{main}
The left module $H_* \M (\cdot)$ is generated by $H_0 \M (\vec{n})$ for $\vec{n} \in E$ and $H_{(w_{\vec{n}_c} - 1)d - 1} \M (\vec{n})$ for $\vec{n} \in \C$.
\end{theorem}

In the case where $m = 1$, this is exactly the theorem of Dobrinskaya and Turchin \cite{DT}. Moreover, it is in the same vein as said result in that each of the spaces $\M (\vec{n})$ for $\vec{n} \in \C$ are homotopy equivalent to spheres. In more general ideals, the above theorem does not hold. To show this fact, we will discuss the case where $m = 2$. To describe to homology for more general ideals, we need to denote more features of the ideal.

\begin{defin}
For each ideal $I \subset \N^2$, let $\D$ denote the set $\{ (n, m) \in I : (n+1, m), (n, m+1) \notin I \cup \C\}$. For each $(n_1, n_2) \in \D$, let its \emph{weight} be $n_1 + n_2$, denoted by $w_{(n_1, n_2)}$
\end{defin}

Before stating our theorem, we define a situation we avoid.

\begin{defin}
Let $I \subset \N^2$ be an ideal. We call $I$ \emph{rectangular} if there exists $m_1, m_2 \in \N \cup \{\infty\}$ such that $I = \{ (n_1, n_2) | 0 \leq n_1 \leq m_1, 0 \leq n_2 \leq m_2 \}$.
\end{defin}

Bicolored configuration spaces arising from rectangular ideals are simply products of two no $k$-equal spaces. Thus, their homology and cohomology can be computed using results on no $k$-equal spaces \cite{YB2, DT}.

Our theorem regarding bicolored configuration spaces is as follows:

\begin{theorem}
\label{main_2}
Let $I \subset \N^2$ be an ideal that is not rectangular. The left module $H_* \B_{I, d} (\cdot, \cdot)$ is generated by $H_0 \B_{I, d} (1, 0)$, $H_0 \B_{I, d} (0,1)$, $H_{(w_{(\ell_1, \ell_2)}-1)d-1} \B_{I, d}(\ell_1, \ell_2)$ for $(\ell_1, \ell_2) \in \C$, and $H_{(w_{(\ell_1, \ell_2)}+1)d-2} \B_{I, d}(\ell_1+1, \ell_2+1)$ for $(\ell_1, \ell_2) \in \D$.
\end{theorem}

Just as the general $m = 2$ case is fundamentally different from the $m = 1$ case, we will conclude by highlighting some differences between the $m = 2$ and $m = 3$ cases.


\section{Homology of $\mathcal{M}_d$} \label{config}

In this section, we will give a brief overview of the homology of configuration spaces of $\R^d$, $\mathcal{M}_d$. For a more extensive look at $H_* \mathcal{M}_d$, I direct the reader to an expository paper written by Sinha \cite{Sinha}.

\begin{defin}[May \cite{May}]
Let $\mathcal{S}$ be a symmetric monoidal category with multiplication $\otimes$ and unit $\kappa$. An \emph{operad}, $\mathcal{C}$, over $\mathcal{S}$ consists of objects indexed by natural numbers: $\mathcal{C}(j)$, a unit map $\eta : \kappa \to \mathcal{C}(1)$, a right action by the symmetric group $S_j$ on $\mathcal{C}(j)$ for all $j$, and product maps:
\begin{equation*}
\mathcal{C}(k) \otimes \mathcal{C}(j_1) \otimes \ldots \otimes \mathcal{C}(j_k) \to \mathcal{C}(j_1 + \ldots + j_k)
\end{equation*}
These maps are required to satisfy associative, unital, and equivarience conditions.
\end{defin}

Intuitively, one thinks of $\mathcal{C}(n)$ as being the set of $n$-ary operations for some algebra. The product maps encode how to compose these operations. In order to define two operads that are of interest to us, we first must introduce algebras over operads.

\begin{defin}
Let $\mathcal{C}$ be an operad. An algebra over $\mathcal{C}$ is an object, $A$, together with maps
\begin{equation*}
\mathcal{C}(j) \otimes A^{j} \to A
\end{equation*}
that satisfy associative, unital, and equivarience conditions.
\end{defin}

Intuitively, $A$ is an algebra whose operations are encoded by $\mathcal{C}$.

\begin{defin}
The associative operad, $\mathrm{Assoc}$, is the operad whose algebras over it are monoids. The degree $d$ Poisson operad, $\mathrm{Pois_d}$, is the operad whose algebras over it are graded unital Poisson algebras with bracket degree $d$.
\end{defin}

\begin{theorem}[Cohen \cite{Cohen}]
For $d=1$, $H_* \mathcal{M}_d$ is $\mathrm{Assoc}$. For $d > 1$, $H_* \mathcal{M}_d$ is $\mathrm{Pois_{d-1}}$.
\end{theorem}

In the case $d=1$, $\mathcal{M}_d (n)$ is homeomorphic to a disjoint union of $n!$ cells of dimension $n-1$. Thus, its only non-zero homology is in dimension zero. The contractible connected components of $\mathcal{M}_d(n)$ are indexed by elements of $S_n$. For $\sigma \in S_n$, a corresponding generator is any point in $\R^d$ such that for all $i, j \in \{1, \ldots, n\}$, if $\sigma(i) < \sigma(j)$, then $x_i < x_j$. Similarly, elements of Assoc are indexed by elements of $S_n$ thought of as describing in which order $n$ elements from the algebra are multiplied. It is not hard check that compositions are compatible. We write the element indexed by $\sigma \in S_n$ as $x_{\sigma(1)} \cdot \ldots \cdot x_{\sigma(n)}$.

Recall that the degree $d$ Poisson operad is generated by three elements: a nullary operation, $1$, and two binary operations, $[x_1, x_2]$ and $x_1 \cdot x_2$. For $d > 1$, $\mathcal{M}_d (d)$ is homotopy equivalent to $S^{d-1}$. Thus, we have non-zero homology in dimensions zero and $d-1$. These correspond to $x_1 \cdot x_2$ and $[x_1, x_2]$, respectively. The preferred generator of $\mathcal{M}_0 (0)$ corresponds to $1$. More concretely, a cycle representing $[x_1, x_2]$ is the $S^{d-1} \subset \mathcal{M}_d$ where $x_1, x_2$ are on the unit $(d-1)$-sphere and $x_1 = -x_2$. Recall, the elements of $\mathrm{Pois_{d-1}}$ satisfy Leibniz, Jacobi, and anti-symmetry relations. The Leibniz rule allows any element to be written such that all Lie multiplication occurs first. The Jacobi and anti-symmetry relations will be used later in this section to find a basis for $H_* \M$.

\section{Decreasing Polychromatic Configuration Spaces} \label{decreasing}

Throughout this section, we will assume that for all $\vec{n} = (n_1, \ldots, n_m)$ with $\sum_{i=1}^m n_i \leq 2$, we have $\vec{n} \in I$. This assumption is added only to avoid unnecessary complications. The same proofs only require small adjustments to go through if this assumption is not satisfied.

\subsection{Homology of $\M$} \label{hom_1}

Throughout this section, we will be concerned with homology with $\Z_2$ coefficients, ignoring the orientations of homology representatives. However, a generalization to $\Z$ coefficients is straightforward if one is careful with signs.

As is evident in the statement of Theorem \ref{main}, there is one class of non-trivial building blocks in $H_* \M$: elements from $H_{({w_{\vec{n}_c}} - 1)d - 1} \M (\vec{n})$ for $\vec{n} \in \C$.

Let $\vec{n} = (n_1, \ldots, n_m)$ be critical. Then $\M(\vec{n})$ is homotopy equivalent to $S^{(w_{\vec{n}}-1)d-1}$. This homotopy equivalence is given by retracting $\M(\vec{n})$ onto the sphere given by the equations:
\begin{align*}
\displaystyle \sum_{j=1}^{m} \sum_{i=1}^{n_j} x_i^j = 0 \\
\displaystyle\sum_{j=1}^{m} \sum_{i=1}^{n_j} |x_i^j|^2 = 1
\end{align*}
Thus, elements of $H_{(w_{\vec{n}_c} - 1)d - 1} \M (\vec{n})$ can be realized by spheres.

\begin{defin}
Denote the sphere described above by $\{x_1^1, \ldots, x_{n_1}^1, \ldots, x_{1}^{m}, \ldots, x_{n_m}^m \}$.
\end{defin}

To see that $\{x_1^1, \ldots, x_{n_1}^1, \ldots, x_{1}^{m}, \ldots, x_{n_m}^m \}$ is in fact non-trivial, consider the chain in $\M (\vec{n})$ given by the following equations:
\begin{align*}
x_1^1 &= x_i^j \text{ for all } j< m, i\leq n_j  \\
x_1^1 &= x_i^m \text{ for all } i <  n_m\\
(x_1^1)_1 &< (x^m_{n_m})_1 \\
(x_1)_{\ell} &= (x^m_{n_m})_{\ell}\text{ for all }\ell > 1
\end{align*}

Here $(z)_{\ell}$ denotes the $\ell^{th}$ coordinate of $z$. The boundary of this chain is in the complement to $\M(\vec{n})$ in $\R^{(n_1 + \ldots + n_m)d}$. Thus, it represents an element in $H^*(\M (\vec{n}), \Z_2)$. It is not hard to check that the intersection pairing between this element and $\{x_1^1, \ldots, x_{n_1}^1, \ldots, x_{1}^{m}, \ldots, x_{n_m}^m \}$ is non-zero.

\begin{defin}
Define \emph{local classes} to be classes of one of the following forms:
\begin{itemize}
\item $\{x_1^1, \ldots, x_{n_1}^1, \ldots, x_{1}^{m}, \ldots, x_{n_m}^m \} \in H_{({w_{\vec{n}_c}} - 1)d - 1} \M (\vec{n})$ for $\vec{n} \in \C$
\item $x_1^j \in H_0 \M(\vec{e}_j)$ for $\vec{e}_j \in E$
\end{itemize}
\end{defin}

The action of $H_* \mathcal{M}_d$ on $H_* \M$ is very similar to the action of $H_* \mathcal{M}_d$ on itself. That is, if $B_1$ and $B_2$ are two elements of $H_* \M$, a representative for $[B_1, B_2]$ is given by considering a representative for $[x_1, x_2]$ and replacing $x_i$ with sufficiently scaled representatives of $B_i$. We will show that all homology classes of $H_* \M$ can be built up using the left action of $\mathcal{M}_d$ on local classes.


Our proof will follow very similarly to that of Dobrinskaya and Turchin \cite{DT}. As with their proof, our proof will use a more general space. Consider the ideal $I' \subset \N^{m+1}$ consisting of the following $(m+1)$-tuples:
\begin{itemize}
\item $(n_1, \ldots, n_m, 0)$ for $(n_1, \ldots, n_m) \in I$ (We will denote such tuples by $(\vec{n}, 0)$)
\item $(0, \ldots, 0, 1)$
\end{itemize}

To emphasize the importance of points of color $m+1$, we will denote them by $z_i$ rather than $x_i^{m+1}$.

\begin{defin}
Define \emph{augmented local classes} to be classes of one of the following forms:
\begin{itemize}
\item $\{x_1^1, \ldots, x_{n_1}^1, \ldots, x_{1}^{m}, \ldots, x_{n_m}^m \} \in H_{({w_{\vec{n}_c}} - 1)d - 1} \Mp (\vec{n}, 0)$ for $\vec{n} \in \C$
\item $x_1^j \in H_0 \Mp(\vec{e}_j)$ for $\vec{e}_j \in E$
\item $z_1 \in H_0 \Mp(\vec{e}_{m+1})$
\end{itemize}
\end{defin}

We will prove the following:

\begin{theorem}
\label{hom_proof}
For all $m \geq 1$ and all decreasing ideals $I \subset \N^m$, the left module $H_* \Mp (\cdot, \ldots, \cdot)$ is generated by augmented local classes.
\end{theorem}

As a corollary of this theorem, we get Theorem \ref{main}. For convenience, we define the following:

\begin{defin}
Call a class \emph{organized} if it can be written as a sum of products of augmented local classes.
\end{defin}

Before proving Theorem \ref{hom_proof}, we define some notation that we will use.

\begin{defin}
For any $N \in H_* \Mp(\vec{n})$, let $N|_{a = A}$ be the class in $H_* \M(\vec{n}')$ given by substituting $A$ for $a$ where $a$ is some $z$ coordinate and $A$ is some element in $H_* \M$.
\end{defin}

For example, $[x_1^1, z_1] |_{z_1 = \{x_1^2, x_2^2, x_3^2 \}}$ is the class $[x^1_1, \{x_1^2, x_2^2, x_3^2\}]$.

We now prove Theorem \ref{hom_proof}

\begin{proof}
The proof will be by induction on $m$. The case $m = 1$ was done by Baryshnikov for $d = 1$ \cite{YB2} and Dobrinskaya and Turchin for $d > 1$ \cite{DT}.

Suppose $m > 1$ and that the claim holds for all $m' < m$. Let $I$ be a decreasing ideal in $\N^m$. We will show that for all $(n_1, \ldots, n_{m+1})$, organized classes span $H_* \Mp(n_1, \ldots, n_{m+1})$. This will be done by induction on $n_m$. First suppose $n_m = 0$. Then this space is homeomorphic to $\mathcal{M}_{J', d} (n_1, \ldots, n_{m-1}, n_{m+1})$ for the decreasing ideal $J \subset \N^{m-1}$ given by $(\ell_1, \ldots, \ell_{m-1}) \in J$ if and only if $(\ell_1, \ldots, \ell_{m-1}, 0) \in I$. All organized classes in $\mathcal{M}_{J', d} (n_1, \ldots, n_{m-1}, n_{m+1})$ are also organized in $H_* \Mp(n_1, \ldots, n_{m+1})$. Thus, the claim holds when $n_m = 0$.

Now suppose $n_m > 0$ and that the claim holds whenever $n'_m < n_m$.  Let $\gamma$ be a closed $s$-chain in $\Mp(n_1, n_2, n_3, \ldots, n_{m+1})$. Consider the homotopy of $\gamma$ affecting only the $x^{m}_{n_m}$ coordinate, $\gamma_t = \gamma + v \cdot t$ where $v$ is a vector that is non-zero only in the $x^m_{n_m}$ coordinate. For large enough $t$, say $t=M$, the $x^m_{n_m}$ coordinate is always far away from all other points. Call the $(s+1)$-chain given by this homotopy $\Gamma$. $\Gamma$ may not be a chain in $\Mp(n_1, n_2, \ldots, n_{m + 1})$. It may intersect forbidden subspaces of the forms:

\begin{align*}
x_{n_m}^m = x_i^j\text{ for all }j \leq m, i \in J_j\text{ where }|J_j| &= \ell_j\text{ for some }(\ell_1, \ldots, \ell_{m-1}, \ell_m + 1)\in \C \\
x_{n_m}^m &= z_j
\end{align*}

In the first case, remove a sufficiently small tubular neighborhood. The intersection of $\Gamma$ with the boundary of this neighborhood is $N|_{z_{(n_{m+1} + 1)} = \{ x_{i_{1, 1}}^1, \ldots, x_{i_{1, \ell_1}}^1, \ldots, x_{i_{m, 1}}^m, \ldots, x_{i_{m, \ell_{m}}}^{m}, x_{n_m}^m \}}$ where $i_{j, 1}, \ldots, i_{j, \ell_j}$ is an enumeration of $J_j$ and $N \in H_*\Mp(n_1 - \ell_1, \ldots, n_{m} - \ell_{m}, n_{m+1}+1)$.

In the second case, again remove a sufficiently small tubular neighborhood. The intersection of $\Gamma$ with the boundary of this neighborhood is $N|_{z_{j} = [ x_{n_m}^m, z_j]}$ where $N \in H_*\Mp(n_1, \ldots, n_{m-1}, n_{m}-1, n_{m+1})$.

For $t=M$, we have a class $N \cdot x_{n_m}^m$ where $N \in H_*\Mp(n_1, \ldots, n_{m-1}, n_{m} - 1, n_{m+1})$.

In each of these cases, the resultant classes are organized. Thus, $\Gamma$ with its intersection with these tubular neighborhoods removed gives a relation which allows $[\gamma]$ to be written as a sum of organized classes. Thus, for all $(n_1, \ldots, n_{m+1})$, organized classes span $H_* \Mp(n_1, \ldots, n_{m+1})$.

Thus, Theorem \ref{hom_proof} holds.
\end{proof}

Theorem \ref{hom_proof} produces a generating set for $H_*\M(\vec{n})$; we would like a basis. For this, relations between various elements in the generating set are needed. Some of the terms shown may be zero depending on $I$.

\begin{lemma}
\label{hom_relations}
Let $\vec{\ell} = (\ell_1, \ldots, \ell_k, 0, \ldots, 0) \in \N^m$, $\ell_k > 0$ be such that $(\ell_1, \ldots, \ell_k - 1, 0, \ldots, 0) \in \C$ Let $d > 1$. Let $J = \{i | \vec{\ell} - \vec{e}_i \in \C \}$. Then the elements of $H_* ( \M(n_1, n_2), \Z_2)$ satisfy the following relation:

\begin{align*}
\displaystyle \sum_{j \in J} \sum_{i=1}^{\ell_j} [\{x_1^1, \ldots, x_{\ell_1}^1, x_{1}^{j}, \ldots, \hat{x}_i^j, \ldots, x_{\ell_j}^j, x_{1}^k, \ldots, x_{\ell_k}^k\}, x_i^j] = 0
\end{align*}

\end{lemma}

\begin{proof}

\item Consider the sphere, $S$, given by the following equations:
\begin{align*}
\displaystyle \sum_{j=1}^k \sum_{i=1}^{\ell_k} x_i^j = 0 \\
\displaystyle \sum_{j=1}^k \sum_{i=1}^{\ell_k} |x_i^j|^2 = 1
\end{align*}
Remove from $S$ tubular neighborhoods of points on $S$ that are not in $\M$. This gives the above relation.
\end{proof}

Using these relations, along with the Jacobi and anti-symmetry relations from $H_* \mathcal{M}_d$, we can find a smaller generating set for $H_*\M(\vec{n})$.

\begin{theorem}
\label{hom_basis}
For all $d>1, \vec{n} \in \N^m$, let $S$ be the set of elements of $H_*\M(\vec{n})$ that can be written as a product where each factor is an $x_i^j$ or of the form:
\begin{equation}
[ \ldots [ [B_1, B_2], B_3] \ldots B_{\ell}],  \ell \geq 1
\end{equation}
where each $B_s$ is of the following form:
\begin{equation*}
[ \ldots [ [ \ldots [\{ x_{i_{1, 1}}^1, \ldots, x_{i_{1, \ell_1}}^1, \ldots, x_{i_{k, 1}}^k, \ldots, x_{i_{k, \ell_{k}}}^{k}\}, x_{r_{1,1}}] \ldots x_{r_{1, s_1}} ],\ldots x_{r_{k, 1}}] \ldots x_{r_{k, s_k}}]
\end{equation*}
where $(\ell_1, \ldots, \ell_k, 0, \ldots, 0) \in \C, i_{j,1} < \ldots < i_{j, \ell_j}, r_{j,1} < \ldots < r_{j, s_j}$. Furthermore, if $s_k > 0$, then $i_{k, \ell_k} > r_{k, s_k}$.

Additionally, we require that the smallest $x^1$ index in $B_1, \ldots, B_{\ell}$ be in $B_1$. Then $S$ is a generating set for $H_*\M(\vec{n})$.
\end{theorem}

\begin{proof}
Throughout this proof, \emph{items} will refer to either a set of curly brackets or a singleton coordinate not in any curly brackets. Recall, as mentioned at the beginning of this section, we may assume that all multiplication occurs outside of Lie brackets. 

If an element of $\alpha \in H_* \M(\vec{n})$ has no Lie brackets, then it is already in the desired form. Thus, we may assume it has Lie brackets. Consider one Lie bracket factor, $F$. The proof will follow by induction on the number of items in $F$. A small case analysis gives that if $F$ contains at most $3$ items, then it can be expressed in the desired form. Thus, suppose it contains $n$ items where $n>3$.  We may write $F = [F_1, F_2]$. There are a three cases.

\textbf{Case 1: $F_1$ and $F_2$ each have at least $2$ items}: In this case, inductively, $F_1$ and $F_2$ can be expressed in the desired form. Thus, $F = [ [ \ldots [ [B_1, B_2], B_3] \ldots B_{\ell}], [ \ldots [ [B'_1, B'_2], B'_3] \ldots B'_{\ell'}] ]$. Without loss of generality, we can assume the smallest $x^1$ index is in $B_1$. Using the Jacobi and anti-symmetry relations, we may write $F$ as
\begin{equation*}
[ [ F_1 , B'_{\ell'}], [ \ldots [ [B'_1, B'_2], B'_3] \ldots B'_{\ell'-1}] ] + [ [F_1, [ \ldots [ [B'_1, B'_2], B'_3], \ldots B'_{\ell'-1}] ],  B'_{\ell'}]
\end{equation*}
In the first summand, we reduced the number of $B_i$ blocks on the right side of the outer most Lie bracket. The second summand can be expressed as $[F', B'_{\ell'}]$ where $F'$ has fewer items than $F$. Thus, inductively $F'$ can be written in the desired form. Thus, continuing this procedure, we may write $F$ in the desired form.

\textbf{Case 2: $F_2$ is a curly bracket}: Inductively, $F_1$ can be expressed in the desired form. Thus, $F$ is written in the form $[ [ \ldots [ [B_1, B_2], B_3] \ldots B_{\ell}], B_{\ell+1} ]$ where $B_{\ell+1} = F_2$. It may happen that the least $x^1$ index is in $B_{\ell+1}$. If this is the case, $F$ may be expressed as:
\begin{equation*}
[ [ \ldots [ [B_1, B_2], B_3] \ldots B_{\ell-1}], [B_{\ell}, B_{\ell+1} ] ] + [ [ [ \ldots [ [B_1, B_2], B_3] \ldots B_{\ell-1}], B_{\ell+1}], B_{\ell} ]
\end{equation*}
The first summand can be treated as case 1. The second summand can either be treated as case 1 or as case 2 where the smallest $x^1$ index is not in $F_2$.

\textbf{Case 3: $F_2$ is a single $x_i^j$}: Inductively, we may write $F_1$ in the desired form. There are now two subcases: either $F_1$ contains a single $B$ block or it contains multiple. In the first case, we may write $F$ as:
\begin{equation*}
[ [F'_1, x_i^j], B_{\ell}] + [ [B_{\ell}, x_i^j], F'_1 ]
\end{equation*}
where $F'_1 = [ \ldots [ [B_1, B_2], B_3] \ldots B_{\ell-1}]$. Both of these summands can be treated by previous cases.

Thus, we may suppose $F_1$ contains only a single $B$ block. That is, $F$ is of the form:
\begin{equation*}
[ [ \ldots [F', x^1_{i_{1, 1}}] \ldots x^k_{i_{k,s_k}}], x_i^j]
\end{equation*}
where $F'$ is some curly bracket expression. If this is not in the desired form, $F$ may be expressed as:

\begin{equation*}
[ [ [ \ldots [F', x^1_{i_{1, 1}}] \ldots, x^k_{i_{k, s_k - 1}}], x_i^j], x^k_{i_{k,s_k}}] + [ [ \ldots [F', x^1_{i_{1, 1}}] \ldots, x^k_{i_{k, s_k - 1}}],  [x^k_{i_{k,s_k}}, x_i^j] ]
\end{equation*}

The second summand is zero. If we order all $x$ coordinates such that all $x^i$ come before $x^{i+1}$, all in their natural linear order, then the first summand has lesser last coordinate than the previous expression. Thus, repeating this case eventually terminates.

Thus, $F$ may be written in the desired form. Doing this for each factor of $\alpha$ completes the proof.
\end{proof}

In the case $d=1$, there is a similar relation to that from Lemma \ref{hom_relations}. The only difference is $[B_1, B_2]$ is replaced with $B_1 \cdot B_2 + B_2 \cdot B_1$. Using this relation, we get the $d=1$ analogue to Theorem \ref{hom_basis}.

\begin{theorem}
\label{hom_basis_1}
For $d=1$ and any $\vec{n} \in \N^m$, let $S$ be the set of elements of $H_* \M(\vec{n})$ that can be written in the form:
\begin{equation*}
A_{I_0} \cdot B_{J_1} \cdot A_{I_1} \cdot \ldots \cdot B_{J_\ell} \cdot A_{I_{\ell}}
\end{equation*}
where $I_0, J_1, \ldots, J_{\ell}, I_{\ell}$ is a partition of $\{ x_i^j : 1 \leq j \leq m, 1\leq i \leq n_j\}$. If $I_s = \{ x_i^j | 1\leq j \leq m, i \in M_{j} \subset [n_j] \}$, then $A_{I_s} = x^1_{i_{1,1}} \cdot \ldots \cdot x^1_{i_{1, \ell_1}} \cdot \ldots \cdot x^m_{i_{m, 1} } \cdot \ldots \cdot x^m_{i_{m, \ell_m}}$ where $i_{j, 1}, \ldots, i_{j, \ell_j}$ is an enumeration of $M_{j}$.

$B_{J_s}$ is of the form:

\begin{center}
$\{ x_{i_{1, 1}}^1, \ldots, x_{i_{1, \ell_1}}^1, \ldots, x_{i_{k, 1}}^k, \ldots, x_{i_{k, \ell_{k}}}^{k}\}$
\end{center}

where $J_s$ is the set of elements $\{ x_{i_{1, 1}}^1, \ldots, x_{i_{1, \ell_1}}^1, \ldots, x_{i_{k, 1}}^k, \ldots, x_{i_{k, \ell_{k}}}^{k}\}$ for some $(\ell_1, \ldots, \ell_k, 0, \ldots, 0) \in \C$.

Furthermore, if $k$ is the maximum color that appears in $J_s$, we require that $I_s$ has no color $\ell$ coordinates for all $\ell > k$ and that the greatest index of a color $k$ coordinate in $J_s$ is greater than any index of any color $k$ coordinate in $I_s$
\end{theorem}

The proof of this also follows by (a much simpler) induction. We leave the details for the reader to fill in. The main idea is to first use relations to ensure $B_{J_\ell}$ and $A_{I_\ell}$ satisfy the desired restrictions. Next, use relations to ensure $B_{J_\ell-1}$ and $A_{I_\ell-1}$ satisfy the desired restrictions. Doing this does not mess up the previous step. Continuing inductively we get each $B_{J_i}$ and $A_{I_i}$ satisfy the restrictions.

In the next section, we show that the generating sets given in Theorems \ref{hom_basis} and \ref{hom_basis_1} are actually bases.


\subsection{Cohomology} \label{cohom_1}
As in the no $k$-equal spaces studied by Dobrinskaya and Turchin \cite{DT}, the cohomology ring of $\M(\vec{n})$ can be described by a space of forests. We will be calculating cohomology with integer coefficients.

Let $N_j = \{ x_i^j : 1 \leq i \leq n_j \}$.

\begin{defin}
An \emph{admissible forest} is a forest satisfying the following:
it has two types of vertices: rectangles and circles. Each circle contains exactly one element of $\bigcup_{j= 1}^m N_j$. Each circle is connected to at most one rectangle and nothing else. Each rectangle is connected to at least one circle. For each rectangle, there exists $(\ell_1, \ldots, \ell_k, 0, \ldots, 0) \in \C, (\ell_k > 0),$ such that the rectangle contains $\ell_j$ elements from $N_j$ for all $j < k$ and $\ell_k - 1$ elements from $N_k$. All circles attached to this rectangle are from $\bigcup_{j = 1}^k N_j$.

An \emph{orientation} of an admissible forest is:
\begin{itemize}
\item an orientation of each edge
\item an ordering of elements within each rectangle
\item an ordering of the set of rectangles and edges
\end{itemize}
\end{defin}

\begin{figure}[t]
\begin{minipage}{0.3 \linewidth}
\def\scale{0.3}
\begin{center}
\begin{tikzpicture}
\foreach \x/\y in {0/0, 0/1, 0/2, 0/3, 0/4, 0/5, 1/0, 1/1, 1/2, 1/3, 1/4, 2/0, 2/1, 2/2, 2/3, 3/0, 3/1}
	{\node[draw,rectangle, minimum height = \scale cm, minimum width = \scale cm] at (\x*\scale, \y*\scale) {};}
\end{tikzpicture}
\end{center}
\end{minipage}
\begin{minipage}{0.6 \linewidth}
\begin{center}
\begin{tikzpicture}[every node/.style={scale=0.8}]
\node[draw, circle] (r0) at (0, 0.7) {$x^1_3$};
\node[draw, circle] (r1) at (1, 1.4) {$x^2_1$};
\node[draw, circle] (d0) at (2,1.4) {$x^2_2$};
\node[draw, circle] (d1) at (3, 1.4) {$x^2_3$};
\node[draw, rectangle] (s1) at (2,0) {$x^2_4$, $x^2_5$, $x^2_7$};
\node[draw, circle] (r2) at (4.5, 1.4) {$x^1_4$};
\node[draw, rectangle] (s2) at (4.5,0) {$x^1_1$, $x^1_2$, $x^1_5$, $x^2_6$, $x^2_8$};
\foreach \dest/\source in {r1/s1, d0/s1, d1/s1, r2/s2, s2/s1}
	\path[-] (\dest) edge (\source);
\end{tikzpicture}
\end{center}
\end{minipage}
\caption{An example of a decreasing ideal, $I$, and an unoriented admissible forest for $I$.}
\end{figure}
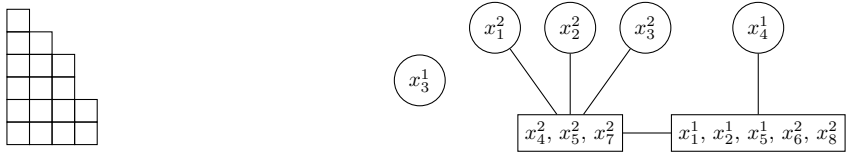

To each admissible forest, we associate a chain in $\R^{(n_1 + \ldots + n_m)d}$ whose boundary lies in the complement of $\M$. Thus, the forest will represent a cocycle in $H^* \M(\vec{n})$. We associate the chain as follows:
\begin{itemize}
\item for each rectangle $A$ and each $x, x' \in A$, $x = x'$
\item if there is an edge from $A$ to $B$ and $x \in A, x' \in B$, then $(x)_1 \leq (x')_1$ and $(x)_{\ell} = (x')_{\ell}$ for all $\ell > 1$ where $(x)_{\ell}$ is the $\ell^{th}$ coordinate of $x$
\end{itemize}

The rest of the orientation data is used to coorient the chain. We coorient the chain by giving an explicit basis for the normal bundle. Suppose there exists an edge from vertex $A$ to vertex $B$. Suppose $x$ is the first element in vertex $A$ and $x'$ is the first element in vertex $B$. Then this edge contributes:
\begin{align*}
\partial (x')_2 - \partial (x)_2, \ldots, \partial (x')_d - \partial (x)_d
\end{align*}
Suppose there exists a rectangle vertex with ordered elements $({x_1}, \ldots, {x_\ell})$. This rectangle vertex contributes:
\begin{align*}
\partial ({x_2})_1 - \partial ({x_1})_1, \ldots, ({x_2})_d - \partial ({x_1})_d,\ldots, \partial ({x_\ell})_d - \partial ({x_1})_d
\end{align*}

We now give relations between forests:

\begin{lemma}
\label{cohom_relations}
For $d>1$, the cohomology classes given by admissible forests have the following relations:
\begin{enumerate}
\item Orientation Relations:
\begin{enumerate}
\item Changing the order of the orientation set produces the Koszul sign of the permutation
\item A permutation $\sigma \in S_n$ of elements inside a rectangle vertex produces a sign $(-1)^{|\sigma| d}$.
\item Changing the orientation of an edge produces the sign $(-1)^d$.
\end{enumerate}
\item 3-term relation:
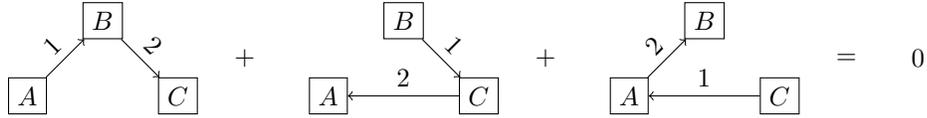
\begin{center}
\begin{tikzpicture}
\node[draw, rectangle] (1A) at (0,0) {$A$};
\node[draw, rectangle] (1B) at (1,1) {$B$};
\node[draw, rectangle] (1C) at (2,0) {$C$};
\path[->, style={sloped}] (1A) edge node[above]{$1$} (1B);
\path[->, style={sloped}] (1B) edge node[above]{$2$} (1C);
\node[text width=0.5 cm] at (3, 0.5) {$+$};
\node[draw, rectangle] (2A) at (4, 0) {$A$};
\node[draw, rectangle] (2B) at (5, 1) {$B$};
\node[draw, rectangle] (2C) at (6, 0) {$C$};
\path[->, style={sloped}] (2B) edge node[above]{$1$} (2C);
\path[->] (2C) edge node[above]{$2$} (2A);
\node[text width=0.5 cm] at (7, 0.5) {$+$};
\node[draw, rectangle] (3A) at (8, 0) {$A$};
\node[draw, rectangle] (3B) at (9, 1) {$B$};
\node[draw, rectangle] (3C) at (10, 0) {$C$};
\path[->] (3C) edge node[above]{$1$} (3A);
\path[->, style={sloped}] (3A) edge node[above]{$2$} (3B);
\node[text width=0.5 cm] at (11, 0.5) {$=$};
\node[text width=0.5 cm] at (12, 0.5) {$0$};
\end{tikzpicture}
\end{center}

\item Relation exchanging values in rectangles: Let $c(j)$ be the color of $z_j$. Let $c'$ be the maximum color of any $z_i$. Suppose $c' \geq k$ and that there exists $s \geq c'$ such that $(\ell_1, \ldots, \ell_k, 0 \ldots, 0) + \vec{e}_{c'} + \vec{e}_{s} \in \C$, then:
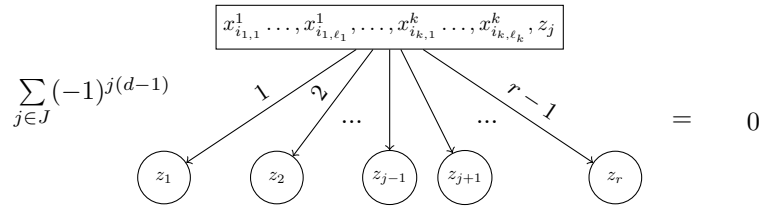
\begin{center}
\begin{tikzpicture}
\node[text width = 8cm] at (0,0) {$\sum\limits_{j\in J} (-1)^{j(d-1)}$};
\node[draw, rectangle, scale = .8] (r) at (1, 1) {$x^1_{i_{1,1}} \ldots, x^1_{i_{1, \ell_1}}, \ldots, x^k_{i_{k,1}} \ldots, x^k_{i_{k, \ell_k}}, z_j$};
\node[draw, circle, minimum size=1cm, scale = .7] (j1) at (-2, -1) {$z_1$};
\node[draw, circle, minimum size=1cm, scale = .7] (j2) at (-0.5, -1) {$z_2$};
\node[text width=.3 cm] at (0.5, -.25) {$...$};
\node[draw, circle, minimum size=1cm, scale = .7] (j3) at (1, -1) {$z_{j-1}$};
\node[draw, circle, minimum size=1cm, scale = .7] (j4) at (2, -1) {$z_{j+1}$};
\node[text width = .3 cm] at (2.3, -.25) {$...$};
\node[draw, circle, minimum size=1cm, scale = .7] (j5) at (4, -1) {$z_{r}$};
\foreach \label/\dest in {1/j1, 2/j2, {}/j3, {}/j4, {\hspace{.45 cm} r-1}/j5}
	\path[->, style ={sloped}] (r) edge node[above]{\small{$\label$}} (\dest);
\node[text width = .5cm] at (5, -.25) {$=$};
\node[text width = .5cm] at (6, -.25) {$0$};
\end{tikzpicture}
\end{center}
where $J = \{j | (\ell_1, \ldots, \ell_k, 0, \ldots, 0) + \vec{e}_{c(j)} \in I\}$.
\end{enumerate}

In relations 2 and 3, the rectangles may be attached to other rectangles.
\end{lemma}

\begin{proof}
Relations 1(a) and 1(b) come from changing the coorientation. For 1(c), the inequality changes from $(i)_1 \leq (j)_1$ to $(i)_1 \geq (j)_1$. To see these are homologous (up to a sign), consider the cell given by the the inequality $(i)_2 < (j)_2$. Its boundary is a sum of the two cells in question.

Relation 2 is equivalent to
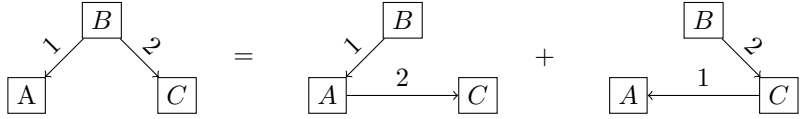
\begin{center}
\begin{tikzpicture}
\node[draw, rectangle] (1A) at (0,0) {A};
\node[draw, rectangle] (1B) at (1,1) {$B$};
\node[draw, rectangle] (1C) at (2,0) {$C$};
\path[->, style={sloped}] (1B) edge node[above]{$1$} (1A);
\path[->, style={sloped}] (1B) edge node[above]{$2$} (1C);
\node[text width=0.5 cm] at (3, 0.5) {$=$};
\node[draw, rectangle] (2A) at (4, 0) {$A$};
\node[draw, rectangle] (2B) at (5, 1) {$B$};
\node[draw, rectangle] (2C) at (6, 0) {$C$};
\path[->, style={sloped}] (2B) edge node[above]{$1$} (2A);
\path[->] (2A) edge node[above]{$2$} (2C);
\node[text width=0.5 cm] at (7, 0.5) {$+$};
\node[draw, rectangle] (3A) at (8, 0) {$A$};
\node[draw, rectangle] (3B) at (9, 1) {$B$};
\node[draw, rectangle] (3C) at (10, 0) {$C$};
\path[->] (3C) edge node[above]{$1$} (3A);
\path[->, style={sloped}] (3B) edge node[above]{$2$} (3C);
\end{tikzpicture}
\end{center}
The cell corresponding to the left hand side is the union of the cells corresponding to the trees on the right hand size.

Relation 3 comes from looking at the boundary of the cell corresponding to:
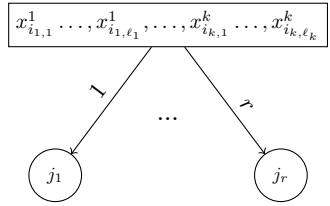
\begin{center}
\begin{tikzpicture}
\node[draw, rectangle, scale = .8] (r) at (1, 1) {$x^1_{i_{1,1}} \ldots, x^1_{i_{1, \ell_1}}, \ldots, x^k_{i_{k,1}} \ldots, x^k_{i_{k, \ell_k}}$};
\node[draw, circle, minimum size=1cm, scale = .7] (j1) at (-0.5, -1) {$j_1$};
\node[text width=.3 cm] at (1, -.25) {$...$};
\node[draw, circle, minimum size=1cm, scale = .7] (j5) at (2.5, -1) {$j_{r}$};
\foreach \label/\dest in {1/j1, {\hspace{.45 cm} r}/j5}
	\path[->, style ={sloped}] (r) edge node[above]{\small{$\label$}} (\dest);
\end{tikzpicture}
\end{center}

The boundary of the subspace corresponding to the above tree has a component for each circle where the coordinate in the circle is equal to all coordinates in the rectangle. For elements not in $J$, these subspaces are not in $\M$ and, thus, contribute zero. For $j \in J$, we need to show that the resultant tree is admissible.

Thus, suppose $c' \geq k$ and there exists $s \geq c'$ with $(\ell_1, \ldots, \ell_k, 0 \ldots, 0) + \vec{e}_{c'} + \vec{e}_{s} \in \C$. I claim that for all $c \leq c'$, if $(\ell_1, \ldots, \ell_k, 0, \ldots, 0) + \vec{e}_{c} \in I$, then $(\ell_1, \ldots, \ell_k, 0, \ldots, 0) + \vec{e}_{c} + \vec{e}_{c'} \in \C$.

By assumption $(\ell_1, \ldots, \ell_k, 0 \ldots, 0) + \vec{e}_{c'} + \vec{e}_{s} \in \C$. Thus, $(\ell_1, \ldots, \ell_k, 0 \ldots, 0) + \vec{e}_{c'} + \vec{e}_{s} \notin I$. Thus, we have $(\ell_1, \ldots, \ell_k, 0 \ldots, 0) + \vec{e}_{c'} + \vec{e}_{s} + \vec{e}_{c} \notin I$. Suppose $(\ell_1, \ldots, \ell_k, 0 \ldots, 0) + \vec{e}_{c'} + \vec{e}_{c} \in I$. Then because $I$ is decreasing,  $(\ell_1, \ldots, \ell_k, 0 \ldots, 0) + \vec{e}_{c'} + \vec{e}_{s} + \vec{e}_{c} \in \C$. However, this can't be since $(\ell_1, \ldots, \ell_k, 0 \ldots, 0) + \vec{e}_{c'} + \vec{e}_{s} \notin I$. Thus, $(\ell_1, \ldots, \ell_k, 0 \ldots, 0) + \vec{e}_{c'} + \vec{e}_{c} \notin I$.

In summary, $(\ell_1, \ldots, \ell_k, 0, \ldots, 0) + \vec{e}_{c} \in I$ and $(\ell_1, \ldots, \ell_k, 0 \ldots, 0) + \vec{e}_{c} + \vec{e}_{c'} \notin I$. Because $I$ is decreasing, $(\ell_1, \ldots, \ell_k, 0 \ldots, 0) + \vec{e}_{c} + \vec{e}_{c'} \in \C$. Thus, the term for each $j \in J$ is an admissible forest.
\end{proof}

The only relation that does not work when $d=1$ is relation 1(c). There is a substitute for 1(c) in the case that $d=1$:

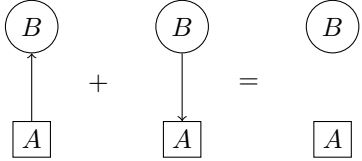
\begin{center}
\begin{tikzpicture}
\node[draw, rectangle] (1A) at (0,0) {$A$};
\node[draw, circle] (1B) at (0,1.5) {$B$};
\path[->] (1A) edge (1B);
\node[text width=0.5 cm] at (1, .75) {$+$};
\node[draw, rectangle] (2A) at (2, 0) {$A$};
\node[draw, circle] (2B) at (2, 1.5) {$B$};
\path[->] (2B) edge (2A);
\node[text width=0.5 cm] at (3, 0.75) {$=$};
\node[draw, rectangle] (3A) at (4, 0) {$A$};
\node[draw, circle] (3B) at (4, 1.5) {$B$};
\end{tikzpicture}
\end{center}
Here, the circle may be replaced with a rectangle.

Using these relations, we produce bases for cohomology.

\begin{defin}
Define a \emph{linear $I$-tree} to be an admissible tree of the following form:

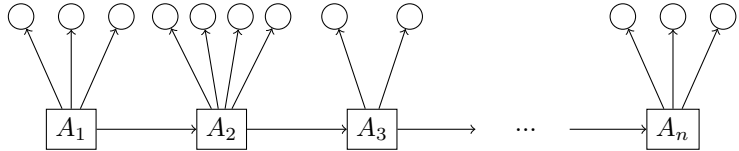
\begin{center}
\begin{tikzpicture}
\foreach \label/\x/\numc in {1/2/3,2/4/4,3/6/2, n/10/3}
	{
	\node[draw, rectangle] (r\label) at (\x,0){$A_{\label}$};
	\foreach \m in {1,...,\numc}
		{
		\node[draw, circle] (s\label\m) at ({\x-1+2*(\m-0.5)/(\numc)}, 1.5) {} ;
		\path[->] (r\label) edge (s\label\m);
		}
	}
\node[text width = 1cm](r4) at (8,0){$\hspace{.4cm}...$};
\path[->] (r1) edge (r2);
\path[->] (r2) edge (r3);
\path[->] (r3) edge (r4);
\path[->] (r4) edge (rn);
\end{tikzpicture}
\end{center}

such that
\begin{itemize}
\item The elements in $A_i$ appear in their natural order
\item The circles attached to each rectangle are ordered similarly
\item The minimal $N_1$ element in the tree is in $B_1$
\item For each $i$, suppose $c$ is the maximum color present in $B_i$. Then the maximum element from $N_c$ in $B_i$ is not in $A_i$.
\end{itemize}

where $B_i$ is the set of elements in $A_i$ and circles attached to $A_i$.
\end{defin}

Using the relations from Lemma \ref{cohom_relations}, any admissible forest can be written as a forest whose components are linear $I$-trees and singleton circles. We will show that this is a basis for $H^* \M(\vec{n})$. For $d>1$, this basis will be dual to the generating set for homology from Theorem \ref{hom_basis}.

\begin{defin}
Let $\mathcal{H}$ be the set of generators given in Theorem \ref{hom_basis}. Let $\mathcal{C}$ be the set of cohomology classes represented by products of linear $I$-trees and singleton circles. Define $f: \mathcal{H} \to \mathcal{C}$ as follows:

Let $A \in \mathcal{H}$. Let $f(A)$ be the forest satisfying the following:
\begin{itemize}
\item For each $x^j_i$ factor in $A$, $f(A)$ has a singleton circle containing $x^j_i$
\item Each other factor in $A$ has a corresponding linear $I$-tree as follows:
Suppose the factor is given by $[ \ldots [ [B_1, B_2], B_3] \ldots B_{\ell}]$, then each $B_i$ has a corresponding rectangle vertex, $A_i$. These rectangle vertices form a path from $A_1$ to $A_\ell$. Recall, each $B_i$, is of the following form:
\begin{align*}
B_i =[ \ldots [ [ \ldots [\{ x_{i_{1, 1}}^1, \ldots, x_{i_{1, \ell_1}}^1, \ldots, x_{i_{k, 1}}^k, \ldots, x_{i_{k, \ell_{k}}}^{k}\}, x_{r_{1,1}}] \ldots x_{r_{1, s_1}} ],\ldots x_{r_{k, 1}}] \ldots x_{r_{k, s_k}}]
\end{align*}
where $(\ell_1, \ldots, \ell_k, 0, \ldots, 0) \in \C, i_{j,1} < \ldots < i_{j, \ell_j}, r_{j,1} < \ldots < r_{j, s_j}$. The rectangle corresponding the $B_i$ satisfies the following:
\begin{itemize}
\item it contains $x_{i_{1, 1}}^1, \ldots, x_{i_{1, \ell_1}}^1, \ldots, x_{i_{k, 1}}^k, \ldots, x_{i_{k, \ell_{k}-1}}^{k}$
\item all other elements from $B_i$ are circles attached to it.

\end{itemize}
\end{itemize}
\end{defin}

It is a straight forward exercise to check $f$ is a bijection.

Order $\mathcal{H}$ such that if $F_1$ has more rectangles than $F_2$, then $F_1$ comes after $F_2$. Order $\mathcal{C}$ according to the ordering of corresponding elements in $\mathcal{H}$.

\begin{theorem}
\label{diag_proof}
With the above ordering, the intersection pairing matrix is diagonal such that each diagonal element is $1$ or $-1$.
\end{theorem}

\begin{cor}
$\mathcal{H}$ and $\mathcal{C}$ are bases for $H_* \M(\vec{n})$ and $H^* \M(\vec{n})$, respectively.
\end{cor}

We now prove Theorem \ref{diag_proof}.
\begin{proof}
We will first prove that the intersection pairing matrix contains $\pm 1$ along the diagonal, then that every entry off the diagonal is 0.

For $x^1_{1} \cdot \ldots \cdot x^1_{n_1} \cdot \ldots \cdot x^m_{1} \cdot \ldots \cdot x^m_{n_m} \in H_0 \M(\vec{n})$, the claim is obvious.

Next, consider a product of singleton coordinates with a single $\{ x_{i_{1, 1}}^1, \ldots, x_{i_{1, \ell_1}}^1, \ldots, x_{i_{k, 1}}^k, \ldots, x_{i_{k, \ell_{k}}}^{k}\}$ where $(\ell_1, \ldots, \ell_k, 0, \ldots, 0) \in \C$. Solving the system of equations from $\{ x_{i_{1, 1}}^1, \ldots, x_{i_{1, \ell_1}}^1, \ldots, x_{i_{k, 1}}^k, \ldots, x_{i_{k, \ell_{k}}}^{k}\}$ and $f(\{ x_{i_{1, 1}}^1, \ldots, x_{i_{1, \ell_1}}^1, \ldots, x_{i_{k, 1}}^k, \ldots, x_{i_{k, \ell_{k}}}^{k}\})$ gives one solution.

Next, consider a product of $[ \ldots [ [ \ldots [\{ x_{i_{1, 1}}^1, \ldots, x_{i_{1, \ell_1}}^1, \ldots, x_{i_{k, 1}}^k, \ldots, x_{i_{k, \ell_{k}}}^{k}\}, x_{r_{1,1}}] \ldots x_{r_{1, s_1}} ],\ldots x_{r_{k, 1}}] \ldots x_{r_{k, s_k}}]$ with singletons. Looking at the analogous system of equations as before again gives one solution. Below is a diagram showing geometrically what the point of intersection is in the case $[[[\{x^1_1, x^1_4, x^2_1\}, x^1_2], x^1_3], x^2_2]\cdot x^2_3\in H_*(\mathcal{M}_{I, 2} (4, 3))$. 
\begin{figure}[h]
\begin{center}
\begin{tikzpicture}
\node[draw, circle, inner sep = 1 pt] (c1) at (0,0) {};
\node[draw, circle, inner sep = 60 pt, dashed] (c2) at (0,0){};
\node[draw, circle, inner sep = 1 pt] (c3) at (-3, 0) {};
\node[draw, circle, inner sep = 30 pt, dashed] (c4) at (-3, 0) {};
\node[draw, circle, inner sep = 1pt, fill, label = {[shift = {(.25,-.25)}]$x^2_2$}] (c4) at (3,0){} ;
\node[draw, circle, inner sep = 1 pt, fill, label = {[shift = {(.25,-.25)}]$x^1_3$}] at (-1.5,0) {};
\node[draw, circle, inner sep = 15 pt, dashed] at (-4.5, 0){};
\node[draw, circle, inner sep = 1 pt] at (-4.5,0){};
\node[draw, circle, inner sep = 1pt, fill, label = {[shift = {(.25,-.25)}]$x^1_2$}] at (-3.75, 0) {};
\node[draw, circle, inner sep = 1pt, fill, label = {[yshift = -0.5 cm]$x^1_4$}] at (-4.75, 0){};
\node[draw, circle, inner sep = 1pt, fill, label = {[xshift = -0.35 cm]$x^1_1, x^2_1$}] at (-5.5, 0){};
\node[draw, circle, inner sep = 1pt] at (-5.25, 0){};
\node[draw, circle, inner sep = 1 pt, fill, label = {$x^2_3$}] at (6, 2){};
\end{tikzpicture}
\end{center}
\caption{Intersection of $[[[\{x^1_1, x^1_4, x^2_1\}, x^1_2], x^1_3], x^2_2]\cdot x^2_3$ and $f([[[\{x^1_1, x^1_4, x^2_1\}, x^1_2], x^1_3], x^2_2]\cdot x^2_3)$ when $d = 2$.}
\end{figure}
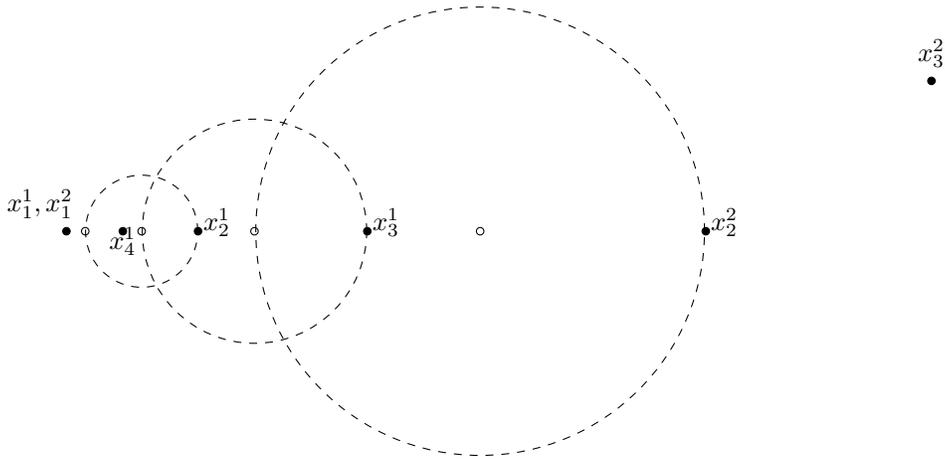

The case of a product of singletons and elements of the form $[ \ldots [ [ B_1, B_2], B_3], \ldots B_s]$ follows in a similar manner.

Arbitrary products of elements as in Theorem \ref{hom_basis} follows from noticing that the different factors correspond to different trees. Indices in different trees do not give any restrictions between the corresponding coordinates. Thus, there still exists a single point of intersection.

Finally, we need to show that all entries off the diagonal are zero. Let $[\alpha] \in \mathcal{H}$, $F \in \mathcal{C}$. Suppose there exist two coordinates in the same tree in $F$ but in different factors of $[\alpha]$. Then there exists a representative of $[\alpha]$ that does not intersect the chain represented by $F$. Thus, we reduce to the case where each tree in $F$ has indices from exactly one $[ \ldots [ [ B_1, B_2], B_3], \ldots B_s]$ factor.

Consider a factor of $\alpha$: $[ \ldots [ [ B_1, B_2], B_3], \ldots B_s]$. Let $T$ be the subforest of $F$ that is the trees that have coordinates from $[ \ldots [ [ B_1, B_2], B_3], \ldots B_s]$. First notice that if $T$ has a rectangle with coordinates that are not a subset of a single curly bracket of some $B_i$, then there exists a representative of $[\alpha]$ such that the corresponding chain does not intersect $\alpha$.

If $T$ has fewer than $s$ rectangles, then the preceding fact plus degree considerations tells us that the intersection pairing between $T$ and $[\alpha]$ is zero. Similarly, if $T$ has more than $s$ rectangle vertices, then the fact that each rectangle vertex of $T$ must correspond to a single curly bracket combined with the restriction of which coordinates can be in a rectangle in the definition of a linear $I$-tree gives that the intersection pairing is again zero. Thus, $T$ must contain exactly $s$ rectangles. Degree considerations also tell us $T$ must be a tree.

Suppose $T'$ is the tree produced from $[ \ldots [ [ B_1, B_2], B_3], \ldots B_s]$ as in the definition of $f$. Suppose $T$ is a linear $I$-tree containing the same coordinates set as $T'$ and has the same number of rectangle vertices as $T'$. As mentioned previously, each rectangle vertex of $T$ corresponds to a single curly bracket. Suppose the ordering of rectangle vertices in $T'$ and $T$ are different. Then there exists a representative of $[\alpha]$ that does not intersect the chain corresponding to $T$. Similarly, if the circles attached to a rectangle vertex in $T$ do not match those in $T'$, then there exists a change of orientation of edges of $T$ such that $\alpha$ does not intersect the the chain corresponding to this new tree. Which of the elements from the curly bracket is not in the rectangle vertex is determined uniquely based on the conditions for a tree being a linear $I$-tree. Thus, if the intersection product of $\alpha$ and the chain corresponding to $T$ is nonzero, then $T = T'$.
\end{proof}

For $d=1$, there is a slightly difference basis. The reason we need a different basis is due to the fact that changing the direction of edges is not as easy in the $d>1$ case. Additionally, we require more connectedness of our forests. Define a linear $I_1$-tree to be a linear $I$-tree with the condition requiring the minimal $N$ element to either be in $A_1$ or a circle attached to $A_1$ removed. Then a very similar proof to Theorem \ref{diag_proof} shows that products of singleton circles and a single linear $I_1$-tree forms a basis for $H^* \mathcal{M}_{I, 1} (\vec{n})$. This basis is not dual to the basis of $H_* \mathcal{M}_{I, 1} (\vec{n})$ given in Theorem \ref{hom_basis_1}, but with a suitable ordering, the intersection pairing matrix is upper triangular.


\subsubsection{Multiplicative Structure}

\begin{defin}
Let $T_1, T_2 \in H^* \M(n_1, n_2)$ be two admissible forests. Suppose for all rectangles $A$ in $T_1$ and $B$ in $T_2$, $A \cap B = \emptyset$. Let $T_1 \cup T_2$ be the tree defined as follows:
\begin{itemize}
\item if $i, j$ are in a common rectangle in $T_1$ or $T_2$, then $i, j$ are in a common rectangle in $T_1 \cup T_2$
\item if $i$ is in a circle in both $T_1$ and $T_2$, then $i$ is in a circle in $T_1 \cup T_2$
\item if $i \in A, j\in B$ in $T_k$ and there exists an edge from $A$ to $B$ in $T_k$, then there exists an edge from the vertex containing $i$ to the vertex containing $j$ in $T_1 \cup T_2$
\end{itemize}
\end{defin}

\begin{theorem}
\label{mult}
Let $T_1, T_2 \in H^* \M (n_1, n_2)$ be two admissible forests. The product of $T_1$ and $T_2$, $T_1 \cdot T_2$, is given as follows:

\begin{enumerate}
\item If there exists rectangles $A$ in $T_1$ and $B$ in $T_2$ such that $A \cap B \neq \emptyset$, then $T_1 \cdot T_2 = 0$.
\item If there exists two indices that are in a common tree in both $T_1$ and $T_2$, then $T_1 \cdot T_2 = 0$
\item If $T_1 \cup T_2$ has a cycle, then $T_1 \cdot T_2 = 0$.
\item If $T_1 \cup T_2$ has a rectangle with no circles attached to it, then $T_1 \cdot T_2 = 0$
\item If $T_1 \cup T_2$ is an admissible forest, then $T_1 \cdot T_2 = T_1 \cup T_2$ with orientation set given by concatenation.
\item If $T_1 \cup T_2$ satisfies none of the above, then use the following relation to make $T_1 \cup T_2$ admissible
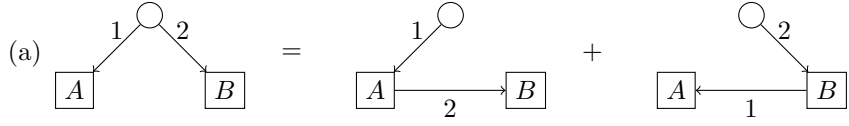
\begin{enumerate}
\vspace{0.4cm}
\item 
\begin{center}
\begin{tikzpicture}[baseline=0.4 cm]
\node[draw, rectangle] (1A) at (0,0) {$A$};
\node[draw, circle] (1B) at (1,1) {};
\node[draw, rectangle] (1C) at (2,0) {$B$};
\path[->] (1B) edge node[above]{$2$} (1C);
\path[->] (1B) edge node[above]{$1$} (1A);
\node[text width=0.5 cm] at (3, 0.5) {$=$};
\node[draw, rectangle] (2A) at (4, 0) {$A$};
\node[draw, circle] (2B) at (5, 1) {};
\node[draw, rectangle] (2C) at (6, 0) {$B$};
\path[->] (2B) edge node[above]{$1$} (2A);
\path[->] (2A) edge node[below]{$2$}(2C);
\node[text width=0.5 cm] at (7, 0.5) {$+$};
\node[draw, rectangle] (3A) at (8, 0) {$A$};
\node[draw, circle] (3B) at (9, 1) {};
\node[draw, rectangle] (3C) at (10, 0) {$B$};
\path[->] (3B) edge node[above]{$2$} (3C);
\path[->] (3C) edge node[below]{$1$} (3A);
\end{tikzpicture}
\end{center}
\end{enumerate}
\end{enumerate}
\end{theorem}

\begin{proof}
\begin{enumerate}
\item If $A \neq B$, then the corresponding chains to $T_1$ and $T_2$ do not intersect in $\M (\vec{n})$. If $A = B$, then one can perturb the chains slightly so that they do not intersect.
\item There exist orientations of edges such that the two corresponding chains do not intersect in $\M (\vec{n})$.
\item There exist orientations of edges such that the two corresponding chains do not intersect in $\M (\vec{n})$.
\item This is relation 4 from Lemma \ref{cohom_relations} with $r = 0$.
\item The chains corresponding to $T_1$ and $T_2$ are transversal and their intersection is the chain corresponding to $T_1 \cup T_2$.
\item Combine the proofs of (5) and relation 2 from Lemma \ref{cohom_relations}.
\end{enumerate}
\end{proof}

A similar construction works if we remove the condition $\sum_{i=1}^m n_i \leq 2$ implies $\vec{n} \in I$. When it comes to homology, the only additional complication arises in the proof of Theorem \ref{hom_basis}. In the second subcase of case 3, $[x_{i_{k, s_k}}^k, x_i^j]$ may be nonzero. If this is the case, it can be replaced by $\{x_{i_{k, s_k}}^k, x_i^j\}$. This can then be written in the desired form using subcase 1 of case 3.

The majority of complications arise in cohomology. This is because now rectangles may have a single coordinate in them. When it comes to the corresponding chain, there is no difference between rectangles containing one element and circles. Thus, we may allow rectangles to have no circles attached to them and add the relation that if a rectangle with one coordinate in it is attached to at most one rectangle and nothing else, it may be turned into a circle. Similarly, we may turn a circle into a rectangle provided such rectangles are allowed. With these changes, linear $I$-trees still form a basis . The only change in multiplication is condition 1 from Theorem \ref{mult} must additionally assume that $A$ and $B$ both have weight at least $2$.

\section{Bicolored Configuration Spaces}\label{bicolored}

As in the previous section, we will assume that for all $\vec{n} = (n_1, n_2)$ with $\sum_{i=1}^2 n_i \leq 2$, we have $\vec{n} \in I$. In this section, the complications that arise are worse than those in the previous section. At the end of this section, we will again comment on these.

Throughout this section, because we only have two colors, to reduce clutter we will write $x_i$ for $x^1_i$ and $y_i$ for $x^2_i$.

\subsection{Homology of $\M$} \label{hom_2}

Again, we will be concerned with homology with $\Z_2$ coefficients, ignoring the orientations of homology representatives. A generalization to $\Z$ coefficients is straightforward if one is careful with signs.

As is evident in the statement of Theorem \ref{main_2}, there are two non-trivial building blocks in $H_* \M$: 
\begin{itemize}
\item elements from $H_{(w_{(\ell_1, \ell_2)}-1)d-1} \M(\ell_1, \ell_2)$ for $(\ell_1, \ell_2) \in \C$
\item elements from $H_{(w_{(\ell_1, \ell_2)}+1)d-2} \M(\ell_1 + 1, \ell_2 + 1)$ for $(\ell_1, \ell_2) \in \D$
\end{itemize}

The first case was discussed in section \ref{hom_1}.

Determining the homotopy type of $H_{(w_{(\ell_1, \ell_2)}+1)d-2} \M(\ell_1 + 1, \ell_2 + 1)$ for $(\ell_1, \ell_2) \in \D$ requires more information about $I$ than just knowing $(\ell_1, \ell_2)$ is in $\D$. For instance, in the case that neither of $(\ell_1 + 1, 0)$ or $(0, \ell_2 + 1)$ are in $I$, then $\M(\ell_1 + 1, \ell_2 + 1)$ is a product of a no-$(\ell_1 + 1)$-equal space and a no-$(\ell_2 + 1)$-equal space. In this case, it is not difficult to check that $H_{(w_{(\ell_1, \ell_2)}+1)d-2} \M(\ell_1 + 1, \ell_2 + 1)= 0$. However, if at least one of $(\ell_1 + 1, 0)$ or $(0, \ell_2 + 1)$ is in $I$ (which is what we assume when we assume $I$ is not rectangular), then $H_{(w_{(\ell_1, \ell_2)}+1)d-2} \M(\ell_1 + 1, \ell_2 + 1) \neq 0$. In the case that $(\ell_1 + 1, 0) \in I$, a non-trivial class in $H_{(w_{(\ell_1, \ell_2)}+1)d-2} \M(\ell_1 + 1, \ell_2 + 1)$ is given by a product of spheres satisfying the equations:
\begin{align*}
\displaystyle c+ \sum_{i=1}^{\ell_1+1} x_i = 0 \hspace{.4 in}
\displaystyle |c|^2 + \sum_{i=1}^{\ell_1+1} |x_i|^2 = 1\\
\displaystyle \sum_{j=1}^{\ell_2+1} (y_j - c) = 0 \hspace{.4 in}
\displaystyle \sum_{j=1}^{\ell_2+1} |y_j - c|^2 = \epsilon
\end{align*}

\begin{defin}
Denote the above product of spheres by $\{x_1, \ldots, x_{\ell_1 + 1}, \{y_1, \ldots, y_{\ell_2 + 1}\} \}$.
\end{defin}

To see that $\{x_1, \ldots, x_{\ell_1+1}, \{y_1, \ldots, y_{\ell_2+1}\} \}$ is non-zero, we will again consider a chain in $\M(\ell_1+1, \ell_2+1)$. Consider the chain given by the following equations:

\begin{align*}
x_1 = \ldots = x_{\ell_1} &= y_1 = \ldots y_{\ell_2} \\
(x_1)_1 &< (x_{\ell_1+1})_1 \\
(x_1)_{\ell} &= (x_{\ell_1+1})_{\ell}\text{ for all }\ell > 1 \\
(x_1)_1 &< (y_{\ell_2+1})_1 \\
(x_1)_{\ell} &= (y_{\ell_2+1})_{\ell}\text{ for all }\ell > 1
\end{align*}

Again, the boundary of this chain is in the complement to $\M(\ell_1 + 1, \ell_2 + 1)$. It is not hard to check that the intersection pairing between the corresponding class in $H^*(\M(\ell_1+1, \ell_2+1), \Z_2)$ and $\{x_1, \ldots, x_{\ell_1+1}, \{y_1, \ldots, y_{\ell_2+1}\} \}$ is non-zero.

\begin{defin}
Define \emph{local classes} to be classes of one of the following forms:
\begin{itemize}
\item $\{x_1, \ldots, x_{\ell_1}, y_1, \ldots, y_{\ell_2}\} \in H_{(w_{(\ell_1, \ell_2)}-1)d-1} \M(\ell_1, \ell_2)$ for $(\ell_1, \ell_2) \in \C$
\item $\{x_1, \ldots, x_{\ell_1+1}, \{y_1, \ldots, y_{\ell_2+1}\} \} \in H_{(w_{(\ell_1, \ell_2)}+1)d-2} \M(\ell_1 + 1, \ell_2 + 1)$ for $(\ell_1, \ell_2) \in \D$
\item $x_1 \in H_0 \M(1, 0)$
\item $y_1 \in H_0 \M(0, 1)$
\end{itemize}
\end{defin}

As mentioned in the previous section, the action of $H_* \mathcal{M}_d$ on $H_* \M$ is very similar to the action of $H_* \mathcal{M}_d$ on itself. Recall, if $B_1$ and $B_2$ are two elements of $H_* \M$, a representative for $[B_1, B_2]$ is given by considering a representative for $[x_1, x_2]$ and replacing $x_i$ with sufficiently scaled representatives of $B_i$. We will show that all homology classes of $H_* \M$ can be built up using the left action of $\mathcal{M}_d$ on local classes.


Again, our proof will use a more general space. We will again consider the homotopy of an $s$-chain moving one of the points away from the others. The $(s+1)$-chain produced from this homotopy may intersect forbidden subspaces. In the decreasing setting, these subspaces do not intersect. In the bicolored setting, these subspaces may intersect. This is what produces the products of spheres mentioned above and what forces us to introduce more than one type of new point.

Let $\ell_I = |\D|$. We will consider a system of points colored with $\ell_I + 3$ colors. Throughout the proof, we will assume that for all $(\ell_1, \ell_2) \in \D$, we have $(\ell_1+1, 0) \in I$. At the end of this section, we will comment what changes occur if this is not the case. To emphasize the importance of the added colors, we will refer to color $3$ points by $z_i$ and color $i+3$ points by $^iw$.

Let $(\alpha_i, \beta_i)$ be an enumeration of the elements of $\D$. Consider $I' = \{ (a, b, 0, \ldots, 0) | (a,b) \in I \} \cup \{ (0, 0, 1, 0, \ldots, 0) \} \cup \{ (a,0,0, \bar{1}_i ) | a \leq \alpha_i \} \subset \N^{\ell_I + 3}$ where $\bar{1}_i$ has all zeros except for a $1$ in the $i^{th}$ coordinate. We will be concerned with the polychromatic configuration space ${\Mp}(n_1, \ldots, n_{\ell_I+3})$. In addition to classes previously mentioned, we will also consider classes of the form $\{ x_1, \ldots, x_{\alpha_i + 1},$ $^iw_1\}$. These are defined analogously to previous local classes.

\begin{defin}
Define \emph{augmented local classes} to be classes of one of the following forms:
\begin{itemize}
\item $\{x_1, \ldots, x_{\ell_1}, y_1, \ldots, y_{\ell_2}\}$ for $(\ell_1, \ell_2) \in \C$
\item $\{x_1, \ldots, x_{\ell_1+1}, \{y_1, \ldots, y_{\ell_2+1}\} \}$ for $(\ell_1, \ell_2) \in \D$
\item $\{ x_1, \ldots, x_{\alpha_i + 1},$ $^iw_1\}$ for $(\alpha_i , \beta_i) \in \D$
\item An element of $H_0 (\vec{e}_j)$ for some $j \leq \ell_I + 3$
\end{itemize}
\end{defin}

We will prove the following:

\begin{theorem}
\label{hom_proof_2}
The left module $H_* \Mp (\cdot, \ldots,  \cdot)$ is generated by augmented local classes.
\end{theorem}

As a corollary of this theorem, we get Theorem \ref{main_2}. For convenience, we define the following:

\begin{defin}
Call a class \emph{organized} if it can be written as a sum of products of augmented local classes.
\end{defin}

The main idea in the proof is that for each closed $s$-chain, $\gamma$, we write $\gamma$ as the sum of two closed $s$-chains: one that is organized and one that is less complex for some suitable measure of complexity.

\begin{defin}
For $n_2 > 0$, let $g_0(\gamma) = \sup \{ k :  \exists \text{ distinct } j_1, \ldots, j_k \text{ such that } \gamma \cap \{ y_{j_1} = \ldots = y_{j_k} \} \neq \emptyset \}$. In the case $n_2 = 0$, define $g_0(\gamma)$ to be $0$.

Let $g_1(\gamma) = \sup \{ k :  \exists i, j_1, \ldots, j_k \text{ such that } \gamma \cap \{ x_i = y_{j_1} = \ldots = y_{j_k} \} \neq \emptyset \}$.
\end{defin}

Ideally $g_0$ would be our measure of complexity; however, this is difficult to achieve, so we settle for $g_1$. Nonetheless, while we decrease $g_1$, we still want to control $g_0$. In our proof, we write each closed $s$-chain as a sum of two chains: one that is organized and one with lesser $g_1$ and suitably bounded $g_0$. Eventually, $g_1$ cannot get any smaller, and we show that this chain is organized.

Before stating our main lemma, there is one more function we define:

\begin{defin}
Let $f_I : \N \to \N\cup\{-\infty, \infty\}$ be defined by $f_I(n) = \sup \{m : (n,m) \in I\}$.
\end{defin}

\begin{lemma}
\label{main_lemma}
Suppose $\tilde{n}_1 > 0$ and organized classes span $H_* \Mp(\tilde{n}_1, \ldots, \tilde{n}_{\ell_I+3})$ for all $(\tilde{n}_1, \tilde{n}_2) < (n_1, n_2)$. Let $\gamma$ be any closed $s$-chain in $\Mp (n_1 ,\ldots, n_{\ell_I+3})$. For all $a \in \N$, $[\gamma] = [\gamma_1^a] + [\gamma_2^a]$ where $[\gamma_1^a]$ is organized and $\gamma_2^a$ satisfies one the following:
\begin{itemize}
\item $[\gamma_2^a]=0$
\item $g_0(\gamma_2^a) \leq f_I(1)+1$ and $ g_1(\gamma_2^a) \leq f_I(a)$
\end{itemize}
\end{lemma}

The proof of this lemma involves four lemmas and is left to the appendix. Given any $a \in \N$ and any closed $s$-chain $\gamma$ such that $g_0(\gamma) \leq f_I(1)+1$ and $g_1(\gamma) \leq f(a)$, these lemmas allow us to write $[\gamma] = [\gamma_1] + [\gamma_2]$ where $[\gamma_1]$ is organized and either $[\gamma_2] = 0$ or $g_0(\gamma_2) \leq f_I(1)+1$ and $g_1(\gamma_2) \leq f_I(a+1)$.

Before proving Theorem \ref{hom_proof_2}, we recall some notation that we will use.

\begin{defin}
For any $N \in H_* \M(\vec{n})$, let $N|_{a = A}$ be the class in $H_* \M(\vec{n}')$ given by substituting $A$ for $a$ where $a$ is some $z$ or $w$ coordinate and $A$ is some $s$-chain $\M$.
\end{defin}

Note: such a substitution does not always produce a class in $H_* \M (\vec{n}')$. However, we will be sure to only make substitutions that do.


We now prove Theorem \ref{hom_proof_2}. 
\begin{proof}
We will show that for all $(n_1, \ldots, n_{\ell_I+3})$, organized classes span $H_* \Mp(n_1, \ldots, n_{\ell_I+3})$. This will be done by induction on $(n_1, n_2)$. First suppose $n_1 = 0$. Because no $^iw$ may equal any $y$ or $z$ coordinate, we may treat them as $z$ coordinates and apply Theorem \ref{hom_proof}.

Now suppose $n_1>0$ and organized classes span $H_* \Mp(\tilde{n}_1, \ldots, \tilde{n}_{\ell_I+3})$ whenever $(\tilde{n}_1, \tilde{n}_2) < (n_1, n_2)$. Let $\gamma$ be a closed $s$-chain in $\Mp(n_1, \ldots, n_{\ell_I+3})$. Let $a = \min\{b | f_I(b) = f_I(n) \forall n>b\}$

By Lemma \ref{main_lemma}, $[\gamma] = [\gamma_1^a] + [\gamma_2^a]$ where $[\gamma_1^a]$ is organized and $\gamma_2^a$ satisfies one the following:
\begin{itemize}
\item $[\gamma_2^a]=0$
\item $g_0(\gamma_2^a) \leq f_I(1)+1$ and $ g_1(\gamma_2^a) \leq f_I(a)$
\end{itemize}

In the first instance, $[\gamma]$ is an organized class. Thus, assume the second case holds. We cannot have $g_1(\gamma_2^a) \leq -\infty$; thus, it must be the case that $f_I(a) \geq 0$. Consider the homotopy of $\gamma_2^a$ affecting only the first $x$ coordinate, $\gamma_t = \gamma + v \cdot t$ where $v$ is a vector that is non-zero only in the $x^1_{1}$ coordinate.. The only forbidden subspaces it may intersect are:
\begin{center}
$x_1 = y_{j_1} = \ldots = y_{j_{f_I(1)+1}}$\\
$x_1 = z_j$ \\
$x_1 = x_{i_2} = \ldots = x_{i_{f_I(\alpha_m)+1}} =$ $^mw_j$\\
\end{center}

As before, remove small tubular neighborhoods of each of these subspaces. The first case produces $N|_{z_{n_3+1} = \{ x_1, y_{j_1}, \ldots, y_{j_{f_I(1)+1}} \}}$ where $N \in H_*\Mp(n_1-1, n_2-(f_I(1)+1), n_3+1, \ldots, n_{\ell_I+3})$. The second case produces $N|_{z_{j} = [ x_1, z_j]}$ where $N \in H_*\Mp(n_1-1, n_2, n_3, \ldots, n_{\ell_I+3})$. The third case produces $N|_{z_{n_3+1} = \{ x_1, x_{i_2}, \ldots, ^mw_j \}}$ where $N \in H_*\Mp(n_1-(f_I(\alpha_m)+1), n_2, n_3+1, \ldots, n_{m+3}-1, \ldots, n_{\ell_I+3})$. For $t=M$, we get $N \cdot x_1$ where $N \in H_*\Mp(n_1-1, n_2, n_3, \ldots, n_{\ell_I+3})$. Inductively, all of these are organized classes. Thus, $\Gamma$ with its intersection with these tubular neighborhoods removed allows us to write $[\gamma]$ as a sum of organized classes. Thus, for all $(n_1, \ldots, n_{\ell_I+3})$, organized classes span $H_* \Mp(n_1, \ldots, n_{\ell_I+3})$.

Thus, Theorem \ref{hom_proof_2} holds.

\end{proof}

Theorem \ref{hom_proof_2} produces a generating set for $H_*\M(n_1, n_2)$; we would like a basis. For this, relations between various elements in the generating set are needed. The following relations are for $d>1$. Some of these relations involve elements that are not organized, but their meanings should be apparent. Additionally, some of the terms shown may be zero depending on $I$. Let $X = \{ x_i | i \leq n_1\}, Y = \{ y_i | i \leq n_2\}$

\begin{lemma}
\label{hom_relations_2}
Whenever $d>1$, the elements of $H_* ( \M(n_1, n_2), \Z_2)$ satisfy the following relations:
\begin{enumerate}

\item If $(n-1, m), (n, m-1) \in \C$, then
\begin{align*}
\displaystyle \sum_{i=1}^{n+m} [\{z_1, \ldots, \hat{z}_i, \ldots, z_{n+m}\}, z_i] = 0
\end{align*}
where $\{ z_1, \ldots, z_n \} \subset X, \{z_{n+1}, \ldots, z_{n+m}\} \subset Y$.
	
\item
	\begin{enumerate}
	\item If $(n-1, m) \in \C, (n, m-1) \in I$, then
	\begin{align*}
	\displaystyle \sum_{i=1}^{n} [\{x_1, \ldots, \hat{x}_i, \ldots, x_{n}, y_1, \ldots, y_m\}, x_i] = 0
	\end{align*}
	
	\item Similar relation for $(n, m-1) \in \C, (n-1, m) \in I$
	\end{enumerate}	

\item If $(n, m) \in \C$, $(n+1, m-2) \notin I$, then

\begin{enumerate}

\item
$\displaystyle \sum_{i=1}^{n+1}[\{x_1, \ldots, \hat{x}_i, \ldots, x_{n+1}, y_1, \ldots, y_m\}, x_i] = \{y_1, \ldots, y_m, \{ x_1, \ldots, x_{n+1}\}\}$

\item
$\displaystyle \sum_{i=1}^{m+1} [\{y_1, \ldots, \hat{y}_i, \ldots, y_{m+1}, \{ x_1, \ldots, x_{n+1}\}\}, y_{i}] = [\{y_1, \ldots, y_{m+1}\}, \{ x_1, \ldots, x_{n+1}\}]$\\
(Similar relations for $(n, m) \in \C, (n-2, m+1) \notin I$)

\end{enumerate}

\item If $(n-1, m-1) \in \D, (n, 0) \in I$, then
\begin{align*}
\displaystyle \sum_{i=1}^{n+1} [\{x_1, \ldots, \hat{x}_i, \ldots, x_{n+1}, \{y_1, \ldots, y_m\}\}, x_{i}] = [\{x_1, \ldots, x_{n+1}\}, \{y_1, \ldots, y_m\}]
\end{align*}
(Similar relation for $(n-1, m-1) \in \D, (0, m) \in I$)

\item  If $(n-1, m-1) \in \D, (n, 0) \in I$, then
\begin{align*}
\displaystyle [\{x_1, \ldots, x_{n}, \{y_1, \ldots, y_m \} \}, y_{m+1}] + [\{x_1, \ldots, x_n, y_{m+1} \}, \{y_1, \ldots, y_m \}] = \{x_1, \ldots, x_n, [\{y_1, \ldots, y_m \}, y_{m+1}] \}
\end{align*}
(Similar relation for $(n-1, m-1) \in \D, (0, m) \in I$)

\item If $(n-1, m-1) \in \D, (n,0) \in I$, then
\begin{align*}
\displaystyle \sum_{i=1}^{m+1} \{x_1, \ldots, x_n, [\{y_1, \ldots, \hat{y}_i, \ldots, y_{m+1}\}, y_i]\} = 0
\end{align*}

\end{enumerate}
\end{lemma}

\begin{proof}
\begin{enumerate}

\item Consider the sphere, $S$, given by the following equations:
\begin{align*}
\displaystyle\sum_{i=1}^{n+m} z_i = 0 \\
\displaystyle\sum_{i=1}^{n+m} |z_i|^2 = 1
\end{align*}
Remove from $S$ tubular neighborhoods of points on $S$ that are not in $\M$. This gives the above relation.

\item Same proof as above; it is just a different set of points removed.

\item

\begin{enumerate}
\item Consider the sphere, $S$, given by the following equations:
\begin{align*}
\displaystyle\sum_{i=1}^{n+1} x_i + \sum_{j=1}^{m} y_j = 0 \\
\displaystyle\sum_{i=1}^{n+1} |x_i|^2 + \sum_{j=1}^{m} |y_j|^2 = 1
\end{align*}
Remove from $S$ tubular neighborhoods of its intersection with the following:
\begin{center}
$x_1 = \ldots = \hat{x}_i = \ldots = x_{n+1} = y_1 = \ldots = y_m$\\
$x_1 = \ldots = x_{n+1}$
\end{center}
What is left is a chain that gives the above relation.

\item Consider the following augmented arrangement: there exists a $w$ coordinate that is not allowed to be equal to any $x$ coordinates and at most $m$ $y$ coordinates. Consider the sphere, $S$, given by the following equations:
\begin{align*}
\displaystyle w+\sum_{j=1}^{m+1} y_j = 0 \\
\displaystyle |w|^2 +\sum_{j=1}^{m+1} |y_j|^2 = 1
\end{align*}
Remove from $S$ tubular neighborhoods of its intersection with the following:
\begin{center}
$w = y_1 = \ldots = \hat{y}_j = \ldots = y_{m+1}$\\
$y_1 = \ldots = y_{m+1}$
\end{center}
This gives the relation:
\begin{align*}
\displaystyle \sum_{i=1}^{m+1} [\{y_1, \ldots, \hat{y}_i, \ldots, y_m, w \}, y_{i}] = [\{y_1, \ldots, y_{m+1}\}, w]. 
\end{align*}
Because of the limited interactions allowed for $w$, we can consider the same chain but with $\{x_1, \ldots, x_{n+1}\}$ substituted for $w$. This gives the relation.
\end{enumerate}

\item Similar proof as 3(b).

\item Consider an augmented space in which a $w$ coordinate is added. This coordinate is allowed to be equal to no $y$ coordinates and up to $n-1$ $x$ coordinates. Consider the sphere, $S$ given by the following:
\begin{align*}
\displaystyle w+\sum_{i=1}^{n} x_i + y_{m+1}= 0 \\
\displaystyle |w|^2 +\sum_{i=1}^{n} |x_i|^2 + |y_{m+1}|^2= 1
\end{align*}
Remove from $S$ tubular neighborhoods of its intersection with the following:
\begin{center}
$x_1 = \ldots = x_n = w$\\
$x_1 = \ldots = x_n = y_{m+1} $\\
$w = y_{m+1} $
\end{center}
This gives the relation:
\begin{align*}
[\{x_1, \ldots, x_{n}, w\}, y_{m+1}] + [\{x_1, \ldots, x_n, y_{m+1} \}, w] = \{x_1, \ldots, x_n, [w, y_{m+1}] \}
\end{align*}
In the above chain, we can substitute $\{y_1, \ldots, y_{m}\}$ for $w$ to get the above relation.

\item First, note that on the chain used to prove
\begin{align*}
\displaystyle\sum_{i=1}^{m+1} [\{y_1, \ldots, \hat{y}_i, \ldots, y_{m+1}\}, y_i]= 0
\end{align*}
there were never $m$ $y$ coordinates all equal. Thus, we have 
\begin{align*}
\displaystyle \{x_1, \ldots, x_n,\sum_{i=1}^{m+1} [\{y_1, \ldots, \hat{y}_i, \ldots, y_{m+1}\}, y_i]\} = 0
\end{align*}
Finally, take the sum outside of the brackets to get the above relation.
\end{enumerate}
\end{proof}

Using these relations, along with the Jacobi and anti-symmetry relations from $H_* \mathcal{M}_d$, we can find a smaller generating set for $H_*\M(n_1, n_2)$.

\begin{theorem}
\label{hom_basis_2}
For all $d>1, (n_1, n_2) \in \N^2$, let $S$ be the set of elements of $H_*\M(n_1, n_2)$ that can be written as a product where each factor is an $x_i$, a $y_j$ or of the form:
\begin{equation}
[ \ldots [ [B_1, B_2], B_3] \ldots B_{\ell}],  \ell \geq 1
\end{equation}
where each $B_s$ is of one of the following forms:
\begin{itemize}
\item
\begin{equation*}
[ \ldots [ [ \ldots [\{x_{i_1}, \ldots, x_{i_n}, y_{j_1}, \ldots, y_{j_m} \}, x_{i'_1}] \ldots x_{i'_{a}} ], y_{j'_1}] \ldots y_{j'_b}]
\end{equation*}
where $(n, m) \in \C, i_1 < \ldots < i_n, j_1 < \ldots < j_m, a,b\geq 0, i'_1 < \ldots < i'_a$, and $j'_1 < \ldots < j'_b$. Furthermore:
\begin{itemize}
\item if $(n+1, m-2) \in I$, then $i_n > i'_a$.
\item if $(n-1, m+1) \in I$, then $j_m > j'_b$.
\item if $(n+1, m-2) \notin I$ and $(n-1, m+1) \notin I$, then $i_n > i'_a$ or $j_m > j'_b$.
\end{itemize}
\item
\begin{equation*}
[ \ldots [ [ \ldots [\{x_{i_1}, \ldots, x_{i_n}, \{y_{j_1}, \ldots, y_{j_m} \} \}, x_{i'_1}] \ldots x_{i'_{a}} ], y_{j'_1}] \ldots y_{j'_b}]
\end{equation*}
where $(n-1, m-1) \in \D, i_1 < \ldots < i_n, j_1 < \ldots < j_m, a,b\geq 0, i'_1 < \ldots < i'_a$, and $j'_1 < \ldots < j'_b$. Furthermore, we require $i_n > i'_a$ and $j_m > j'_b$.
\end{itemize}
Additionally, we require the smallest $x$ index in $B_1, \ldots, B_{\ell}$ to be in $B_1$. Then $S$ is a generating set for $H_*\M(n_1, n_2)$.
\end{theorem}

The proof follows in the same manner as the proof to Theorem $\ref{hom_basis}$ in the previous section.

Again, in the case $d=1$, there are similar relations to those from Lemma \ref{hom_relations_2}. The only difference is any mention of $[B_1, B_2]$ is replaced with $B_1 \cdot B_2 + B_2 \cdot B_1$. Using these relations, we get the $d=1$ analogue to Theorem \ref{hom_basis}.

\begin{theorem}
\label{hom_basis_3}
For $d=1$ and any $(n_1, n_2) \in \N^2$, let $S$ be the set of elements of $H_* \M(n_1,n_2)$ that can be written in the form:
\begin{equation*}
A_{I_0} \cdot B_{J_1} \cdot A_{I_1} \cdot \ldots \cdot B_{J_\ell} \cdot A_{I_{\ell}}
\end{equation*}
Here, $I_0, J_1, \ldots, J_{\ell}, I_{\ell}$ is a partition of $X \cup Y$. $A_{I_s} = x_{i'_{1, s}} \cdot \ldots \cdot x_{i'_{a_s, s}} \cdot y_{j'_{1, s}} \cdot \ldots \cdot y_{j'_{b_s, s}}$ where $I_s =\{x_{i'_{1, s}}, \ldots, x_{i'_{a_s, s}}, y_{j'_{1, s}}, \ldots, y_{j'_{b_s, s}}\}$ with $i'_{1,s} < \ldots < i'_{a_s,s}$ and $j'_{1,s} < \ldots < j'_{b_s, s}$. $B_{J_s}$ is of one of the two forms:
\begin{itemize}
\item $\{x_{i_{1, s}}, \ldots, x_{i_{n_s, s}}, y_{j_{1, s}}, \ldots, y_{j_{m_s, s}} \}$
where $J_s$ is the set of elements $\{x_{i_{1, s}}, \ldots, x_{i_{n_s, s}}, y_{j_{1, s}}, \ldots, y_{j_{m_s, s}} \}$ for some $(n_s, m_s) \in \C$.

\item $\{x_{i_{1, s}}, \ldots, x_{i_{n_s, s}}, \{y_{j_{1, s}}, \ldots, y_{j_{m_s, s}} \} \}$
where $J_s$ is the set of elements $\{x_{i_{1, s}}, \ldots, x_{i_{n_s, s}}, y_{j_{1, s}}, \ldots, y_{j_{m_s, s}} \}$ for some $(n_s - 1, m_s - 1) \in \D$
\end{itemize}

In either case we have $i_{1,s} < \ldots < i_{a_s,s}$ and $j_{1,s} < \ldots < j_{b_s, s}$. Furthermore, if $B_{J_s}$ is of the first type, we have:
\begin{itemize}
\item if $(n+1, m-2) \in I$, then $i_{n_s, s} > i'_{a_s, s}$.
\item if $(n-1, m+1) \in I$, then $j_{m_s, s} > j'_{b_s, s}$.
\item if $(n+1, m-2) \notin I$ and $(n-1, m+1) \notin I$, then $i_{n_s, s} > i'_{a_s, s}$ or $j_{m_s, s} > j'_{b_s, s}$.
\end{itemize}

If $B_{J_s}$ is of the second type, we have $i_{n_s, s} > i'_{a_s, s}$ and $j_{m_s, s} > j'_{b_s, s}$. Then $S$ is a generating set for $H_* \M(n_1,n_2)$.
\end{theorem}

The proof follows similar to past proofs.

In the next section, we show that the generating sets given in Theorems \ref{hom_basis_2} and \ref{hom_basis_3} are actually bases. As mentioned earlier, there is a slight modification if there exists $(n-1, m-1) \in \D$ such that $(n, 0) \notin I$. In this case, $\{x_1, \ldots, x_n, \{y_1, \ldots, y_m\} \}$ does not live in $\M(n_1, n_2)$. If in addition $(0, m) \notin I$, then, as mentioned earlier, $\M(n_1, n_2)$ is a product of two no-$k$-equal spaces. The homology of $\M(n_1, n_2)$ can be determined using this structure. In this case, there is a generating set similar to that in Theorem \ref{hom_basis_2} where the second type of $B_s$ is never present (similarly for \ref{hom_basis_3} and the second type of $B_{J_s}$). In the event $(0, m)\in I$, then $\{y_1, \ldots, y_m, \{x_1, \ldots, x_n\} \}$ is a chain in $\M(n_1, n_2)$. We simply replace any mention of $\{x_1, \ldots, x_n, \{y_1, \ldots, y_m\} \}$ with $\{y_1, \ldots, y_m, \{x_1, \ldots, x_n\} \}$. The proof of this follows in a slightly modified but straightforward manner.


\subsection{Cohomology} \label{cohom_2}
The cohomology ring of $\M(n_1, n_2)$ can be described by a space of forests. We will be calculating cohomology with integer coefficients. We will still be working under the assumption that $I$ is not rectangular.

Again, there are distinguished elements of $\N^2$

\begin{defin}
Let $\Cp = \{ (n,m) \in I: (n+1, m), (n, m+1) \notin I\}$.

Let $\Dp = \{ (n, m) \in \C : \{(n-1, m+1), (n+1, m-1)\} \cap \N^2 \subset I \}$.
\end{defin}

\begin{defin}
An \emph{admissible forest} is a forest satisfying the following:
it has three types of vertices: rectangles, circles, and diamonds. Each circle and diamond contains exactly one element of $X \cup Y$. Each circle is connected to at most one rectangle and nothing else. Each diamond is connected to exactly one rectangle and nothing else. Each rectangle is connected to at least one circle. Each rectangle satisfies one of the following:
\begin{enumerate}
\item it contains $n$ elements from $X$ and $m$ elements from $Y$ for some $(n, m) \in \Cp$
\item it contains $n-1$ elements from $X$ and $m-1$ elements from $Y$ for some $(n, m) \in \Dp$
\item it contains $n-1$ elements from $X$ for $(n, 0) \in \Dp$
\item it contains $m-1$ elements from $Y$ for $(0, m) \in \Dp$
\end{enumerate}
In case 1, there are no diamonds attached to the rectangle. In case 2, either all circles attached to the rectangle contain elements of $Y$ and all diamonds attached to it contain elements from $X$ or all circles attached to the rectangle contain elements from $X$ and all diamonds attached to it contain elements from $Y$. In case 3, there are no diamonds attached to it, and all circles attached to it contain elements from $X$. In case 4, there are no diamonds attached to it, and all circles attached to it contain elements from $Y$. Finally, each element of $X \cup Y$ must be in exactly one vertex.

An \emph{orientation} of an admissible forest is:
\begin{itemize}
\item an orientation of each edge
\item an ordering of elements within each rectangle
\item an ordering of the set of rectangles and edges
\end{itemize}

For a rectangle, if it has no diamonds attached to it, we say its \emph{weight} is the number of elements it contains. Otherwise, its weight is one more than the number of elements it contains.
\end{defin}

Because of the assumption that $(2, 0), (1, 1), (0, 2) \in I$, all rectangles have weight at least $2$.

\begin{figure}[t]
\begin{minipage}{0.3 \linewidth}
\def\scale{0.3}
\begin{center}
\begin{tikzpicture}
\foreach \x/\y in {0/0, 0/1, 0/2, 0/3, 0/4, 0/5, 1/0, 1/1, 1/2, 1/3, 2/0, 2/1, 2/2, 2/3, 3/0, 3/1, 4/0, 5/0}
	{\node[draw,rectangle, minimum height = \scale cm, minimum width = \scale cm] at (\x*\scale, \y*\scale) {};}
\end{tikzpicture}
\end{center}
\end{minipage}
\begin{minipage}{0.6 \linewidth}
\begin{center}
\begin{tikzpicture}[every node/.style={scale=0.8}]
\node[draw, circle] (r0) at (0, 0.7) {$x_3$};
\node[draw, circle] (r1) at (1, 1.4) {$y_4$};
\node[draw, diamond] (d0) at (2,1.4) {$x_1$};
\node[draw, diamond] (d1) at (3, 1.4) {$x_9$};
\node[draw, rectangle] (s1) at (2,0) {$y_1$, $y_3$, $y_7$};
\node[draw, circle] (r2) at (4.5, 1.4) {$x_6$};
\node[draw, rectangle] (s2) at (4.5,0) {$x_2$, $x_7$, $y_2$, $y_8$, $y_9$};
\node[draw,rectangle] (s3) at (7, 0) {$x_4$, $x_5$, $x_8$, $y_5$};
\node[draw, circle](r3) at (7, 1.4) {$y_6$};
\foreach \dest/\source in {r1/s1, d0/s1, d1/s1, r2/s2, s2/s1, r3/s3, s3/s2}
	\path[-] (\dest) edge (\source);
\end{tikzpicture}
\end{center}
\end{minipage}
\caption{An example of an ideal, $I$, and an unoriented admissible forest for $I$.}
\end{figure}
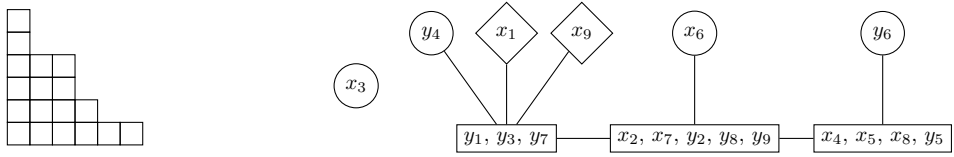

Again, to each admissible forest, we associate a chain in $\R^{d(n+m)}$ whose boundary lies in the complement of $\M$. Thus, the forest will represent a cocycle in $H^* \M(n, m)$. We associate the chain as follows:
\begin{itemize}
\item for each rectangle $A$ and each $i, j \in A$, $i = j$
\item if there is an edge from $A$ to $B$ and $i \in A, j \in B$, then $(i)_1 \leq (j)_1$ and $(i)_{\ell} = (j)_{\ell}$ for all $\ell > 1$ where $(i)_{\ell}$ is the $\ell^{th}$ coordinate of $i$
\item for each rectangle $A$ with diamonds $B$ attached to it, $\exists i \in B$ such that ${i}$ exhibits the behavior of an element in $A$.
\end{itemize}

The rest of the orientation data is used to coorient the chain. We coorient the chain by giving an explicit basis for the normal bundle. Suppose there exists an edge from vertex $A$ to vertex $B$. Suppose $i$ is the first element in vertex $A$ and $j$ is the first element in vertex $B$. Then this edge contributes:
\begin{align*}
\partial (j)_2 - \partial (i)_2, \ldots, \partial (j)_d - \partial (i)_d
\end{align*}
Suppose there exists a rectangle vertex with ordered elements $({i_1}, \ldots, {i_\ell})$. This rectangle vertex contributes:
\begin{align*}
\partial ({i_2})_1 - \partial ({i_1})_1, \ldots, ({i_2})_d - \partial ({i_1})_d,\ldots, \partial ({i_\ell})_d - \partial ({i_1})_d
\end{align*}

If this rectangle vertex has diamonds attached to it, then for the subspace where $j$ behaves like an element in the rectangle, add $\partial ({j})_d - \partial ({i_1})_d$ to the end of the rectangle's contribution.

We now give relations between forests:

\begin{lemma}
\label{cohom_relations_2}
For $d>1$, the cohomology classes given by admissible forests have the following relations:
\begin{enumerate}
\item Orientation Relations:
\begin{enumerate}
\item Changing the order of the orientation set produces the Koszul sign of the permutation
\item A permutation $\sigma \in S_n$ of elements inside a rectangle vertex produces a sign $(-1)^{|\sigma| d}$.
\item Changing the orientation of an edge produces the sign $(-1)^d$.
\end{enumerate}
\item 3-term relation:
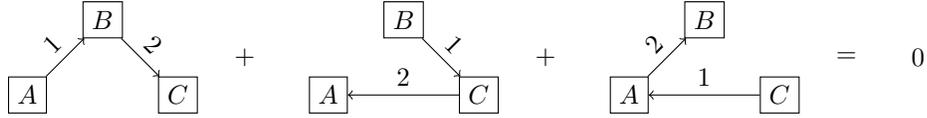
\begin{center}
\begin{tikzpicture}
\node[draw, rectangle] (1A) at (0,0) {$A$};
\node[draw, rectangle] (1B) at (1,1) {$B$};
\node[draw, rectangle] (1C) at (2,0) {$C$};
\path[->, style={sloped}] (1A) edge node[above]{$1$} (1B);
\path[->, style={sloped}] (1B) edge node[above]{$2$} (1C);
\node[text width=0.5 cm] at (3, 0.5) {$+$};
\node[draw, rectangle] (2A) at (4, 0) {$A$};
\node[draw, rectangle] (2B) at (5, 1) {$B$};
\node[draw, rectangle] (2C) at (6, 0) {$C$};
\path[->, style={sloped}] (2B) edge node[above]{$1$} (2C);
\path[->] (2C) edge node[above]{$2$} (2A);
\node[text width=0.5 cm] at (7, 0.5) {$+$};
\node[draw, rectangle] (3A) at (8, 0) {$A$};
\node[draw, rectangle] (3B) at (9, 1) {$B$};
\node[draw, rectangle] (3C) at (10, 0) {$C$};
\path[->] (3C) edge node[above]{$1$} (3A);
\path[->, style={sloped}] (3A) edge node[above]{$2$} (3B);
\node[text width=0.5 cm] at (11, 0.5) {$=$};
\node[text width=0.5 cm] at (12, 0.5) {$0$};
\end{tikzpicture}
\end{center}

\item Switching which color are diamonds: If $(n+1, m+1) \in \Dp$, then
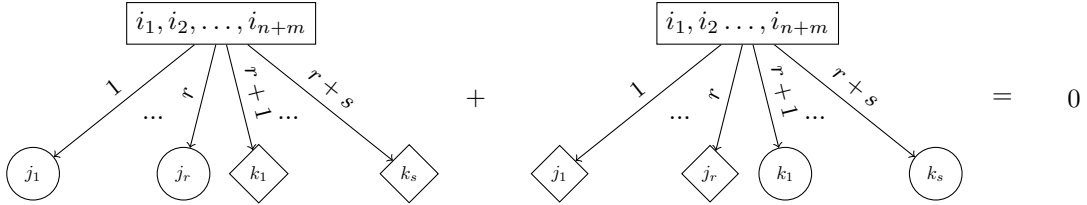
\begin{center}
\begin{tikzpicture}
\node[draw, rectangle] (r) at (1.5, 1) {$i_1, i_2, \ldots, i_{n+m}$};
\node[draw, circle, minimum size=1cm, scale = .7] (j1) at (-1, -1) {$j_1$};
\node[text width = .3 cm] at (0.6, -.25) {$...$};
\node[draw, circle, minimum size=1cm, scale = .7] (j5) at (1, -1) {$j_{r}$};
\node[draw, diamond, minimum size=1cm, scale = .7] (k1) at (2, -1) {$k_1$};
\node[text width = .3 cm] at (2.4, -.25) {$...$};
\node[draw, diamond, minimum size=1cm, scale = .7] (ks) at (4, -1) {$k_s$};
\foreach \label/\dest in {1/j1, {r}/j5, r+1/k1, r+s/ks}
	\path[->, style={sloped}] (r) edge node[above]{\small{$\label$}} (\dest);
\node[text width = .5cm] at (5, 0) {$+$};
\node[draw, rectangle] (r) at (8.5, 1) {$i_1, i_2 \ldots, i_{n+m}$};
\node[draw, diamond, minimum size=1cm, scale = .7] (j1) at (6, -1) {$j_1$};
\node[text width = .3 cm] at (7.6, -.25) {$...$};
\node[draw, diamond, minimum size=1cm, scale = .7] (j5) at (8, -1) {$j_{r}$};
\node[draw, circle, minimum size=1cm, scale = .7] (k1) at (9, -1) {$k_1$};
\node[text width = .3 cm] at (9.4, -.25) {$...$};
\node[draw, circle, minimum size=1cm, scale = .7] (ks) at (11, -1) {$k_s$};
\foreach \label/\dest in {1/j1, {r}/j5, r+1/k1, r+s/ks}
	\path[->, style={sloped}] (r) edge node[above]{\small{$\label$}} (\dest);
\node[text width = .5cm] at (12, 0) {$=$};
\node[text width = .5cm] at (13, 0) {$0$};
\end{tikzpicture}
\end{center}

where $\{i_1, \ldots, i_{n}, j_1, \ldots, j_r\} \subset X, \{ i_{n+1}, \ldots, i_{n+m}, k_1, \ldots, k_s\}\subset Y$.

\item Relations exchanging values in rectangles:
\begin{enumerate}
\item If $(n+1, m), (n, m+1) \in \Cp, then$
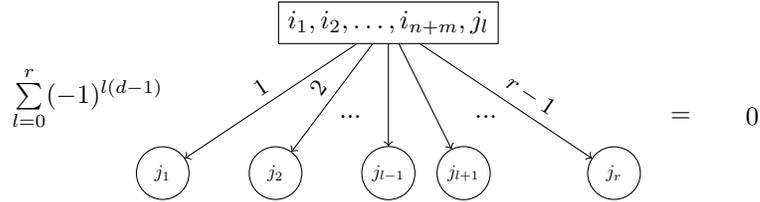
\begin{center}
\begin{tikzpicture}
\node[text width = 8cm] at (0,0) {$\sum\limits_{l=0}^r (-1)^{l(d-1)}$};
\node[draw, rectangle] (r) at (1, 1) {$i_1, i_2, \ldots, i_{n+m}, j_l$};
\node[draw, circle, minimum size=1cm, scale = .7] (j1) at (-2, -1) {$j_1$};
\node[draw, circle, minimum size=1cm, scale = .7] (j2) at (-0.5, -1) {$j_2$};
\node[text width=.3 cm] at (0.5, -.25) {$...$};
\node[draw, circle, minimum size=1cm, scale = .7] (j3) at (1, -1) {$j_{l-1}$};
\node[draw, circle, minimum size=1cm, scale = .7] (j4) at (2, -1) {$j_{l+1}$};
\node[text width = .3 cm] at (2.3, -.25) {$...$};
\node[draw, circle, minimum size=1cm, scale = .7] (j5) at (4, -1) {$j_{r}$};
\foreach \label/\dest in {1/j1, 2/j2, {}/j3, {}/j4, {\hspace{.45 cm} r-1}/j5}
	\path[->, style ={sloped}] (r) edge node[above]{\small{$\label$}} (\dest);
\node[text width = .5cm] at (5, -.25) {$=$};
\node[text width = .5cm] at (6, -.25) {$0$};
\end{tikzpicture}
\end{center}
where $\{i_1, \ldots, i_n\} \subset X, \{ i_{n+1}, \ldots, i_{n+m}\}\subset Y$, and $\{j_1, \ldots, j_r\} \subset X\cup Y$.

\item If $(n+1, m) \in \Cp, (n, m+1) \notin I$, then
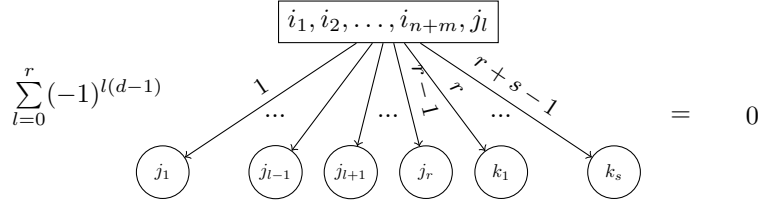
\begin{center}
\begin{tikzpicture}
\node[text width = 8cm] at (0,0) {$\sum\limits_{l=0}^r (-1)^{l(d-1)}$};
\node[draw, rectangle] (r) at (1, 1) {$i_1, i_2, \ldots, i_{n+m}, j_l$};
\node[draw, circle, minimum size=1cm, scale = .7] (j1) at (-2, -1) {$j_1$};
\node[text width=.3 cm] at (-0.5, -.25) {$...$};
\node[draw, circle, minimum size=1cm, scale = .7] (j3) at (-0.5, -1) {$j_{l-1}$};
\node[draw, circle, minimum size=1cm, scale = .7] (j4) at (0.5, -1) {$j_{l+1}$};
\node[text width = .3 cm] at (1, -.25) {$...$};
\node[draw, circle, minimum size=1cm, scale = .7] (j5) at (1.5, -1) {$j_{r}$};
\node[draw, circle, minimum size=1cm, scale = .7] (k1) at (2.5, -1) {$k_1$};
\node[text width = .3 cm] at (2.5, -.25) {$...$};
\node[draw, circle, minimum size=1cm, scale = .7] (ks) at (4, -1) {$k_s$};
\foreach \label/\dest in {1/j1, {}/j3, {}/j4, {r-1}/j5, r/k1, r+s-1/ks}
	\path[->, style={sloped}] (r) edge node[above]{\small{$\label$}} (\dest);
\node[text width = .5cm] at (5, -.25) {$=$};
\node[text width = .5cm] at (6, -.25) {$0$};
\end{tikzpicture}
\end{center}
where $\{i_1, \ldots, i_n, j_1, \ldots, j_r\} \subset X, \{ i_{n+1}, \ldots, i_{n+m}, k_1, \ldots, k_s\}\subset Y$.

\item Similar relation if $(n, m+1) \in \Cp, (n+1, m) \notin I$

\item For any $(n+1,m+1) \in \Cp,$

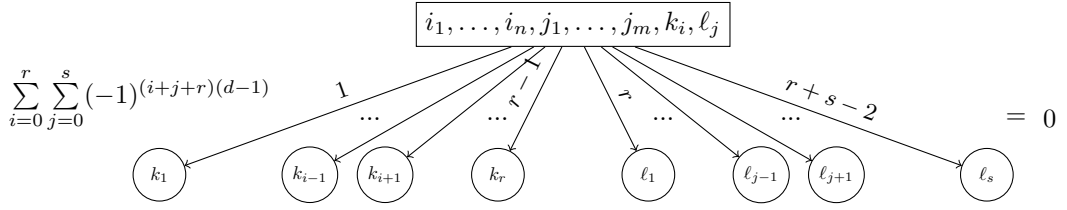
\begin{center}
\begin{tikzpicture}
\node[text width = 8cm] at (0,0) {$\sum\limits_{i=0}^r\sum\limits_{j=0}^s(-1)^{(i+j+r)(d-1)}$};
\node[draw, rectangle](r) at (3.5,1){$i_1, \ldots, i_n, j_1, \ldots, j_m, k_i, \ell_j$};
\node[draw, circle, minimum size=1cm, scale = .7] (i1) at (-2, -1) {$k_1$};
\node[text width=.3 cm] at (0.8, -.25) {$...$};
\node[draw, circle, minimum size=1cm, scale = .7] (i2) at (0, -1) {$k_{i-1}$};
\node[draw, circle, minimum size=1cm, scale = .7] (i3) at (1, -1) {$k_{i+1}$};
\node[text width = .3 cm] at (2.5, -.25) {$...$};
\node[draw, circle, minimum size=1cm, scale = .7] (i4) at (2.5, -1) {$k_{r}$};
\foreach \label/\dest in {1/i1, {}/i2, {}/i3, {\hspace{.25 cm} r-1}/i4}
	\path[->,style={sloped}] (r) edge node[above]{\small{$\label$}} (\dest);
\node[draw, circle, minimum size=1cm, scale = .7] (j1) at (4.5, -1) {$\ell_1$};
\node[text width=.3 cm] at (4.7, -.25) {$...$};
\node[draw, circle, minimum size=1cm, scale = .7] (j2) at (6, -1) {$\ell_{j-1}$};
\node[draw, circle, minimum size=1cm, scale = .7] (j3) at (7, -1) {$\ell_{j+1}$};
\node[text width = .3 cm] at (6.4, -.25) {$...$};
\node[draw, circle, minimum size=1cm, scale = .7] (j4) at (9, -1) {$\ell_{s}$};
\foreach \label/\dest in {\hspace{.1 cm} r/j1, {}/j2, {}/j3, {\hspace{.75 cm} r+s-2}/j4}
	\path[->, style={sloped}] (r) edge node[above]{\small{$\label$}} (\dest);
\node[text width = .5cm] at (9.5, -.25) {$=$};
\node[text width = .5cm] at (10, -.25) {$0$};
\end{tikzpicture}
\end{center}
where $\{i_1, \ldots, i_n, k_1, \ldots, k_r\} \subset X, \{j_1, \ldots, j_m, \ell_1, \ldots, \ell_s \} \subset Y$

\item For any $(n+1, m+1) \in \Dp$,
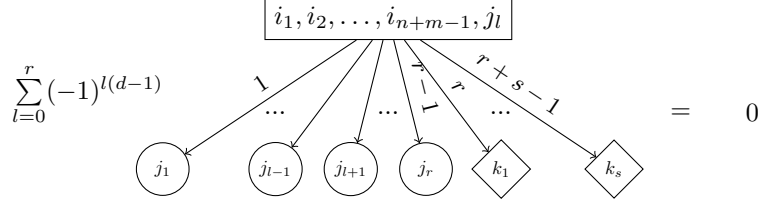
\begin{center}
\begin{tikzpicture}
\node[text width = 8cm] at (0,0) {$\sum\limits_{l=0}^r (-1)^{l(d-1)}$};
\node[draw, rectangle] (r) at (1, 1) {$i_1, i_2, \ldots, i_{n+m-1}, j_l$};
\node[draw, circle, minimum size=1cm, scale = .7] (j1) at (-2, -1) {$j_1$};
\node[text width=.3 cm] at (-0.5, -.25) {$...$};
\node[draw, circle, minimum size=1cm, scale = .7] (j3) at (-0.5, -1) {$j_{l-1}$};
\node[draw, circle, minimum size=1cm, scale = .7] (j4) at (0.5, -1) {$j_{l+1}$};
\node[text width = .3 cm] at (1, -.25) {$...$};
\node[draw, circle, minimum size=1cm, scale = .7] (j5) at (1.5, -1) {$j_{r}$};
\node[draw, diamond, minimum size=1cm, scale = .7] (k1) at (2.5, -1) {$k_1$};
\node[text width = .3 cm] at (2.5, -.25) {$...$};
\node[draw, diamond, minimum size=1cm, scale = .7] (ks) at (4, -1) {$k_s$};
\foreach \label/\dest in {1/j1, {}/j3, {}/j4, {r-1}/j5, r/k1, r+s-1/ks}
	\path[->, style={sloped}] (r) edge node[above]{\small{$\label$}} (\dest);
\node[text width = .5cm] at (5, -.25) {$=$};
\node[text width = .5cm] at (6, -.25) {$0$};
\end{tikzpicture}
\end{center}
where $\{i_1, \ldots, i_{n-1}, j_1, \ldots, j_r\} \subset X, \{ i_{n}, \ldots, i_{n+m-1}, k_1, \ldots, k_s\}\subset Y$.
\end{enumerate}
\end{enumerate}

In relations 2-4, the rectangles may be attached to other rectangles.
\end{lemma}

\begin{proof}
Relations 1(a) and 1(b) come from changing the coorientation. For 1(c), the inequality changes from $x_i^1 \leq x_j^1$ to $x_i^1 \geq x_j^1$. To see these are homologous (up to a sign), consider the cell given by the the inequality $x_i^2 < x_j^2$. Its boundary is a sum of the two cells in question.

Relation 2 is equivalent to
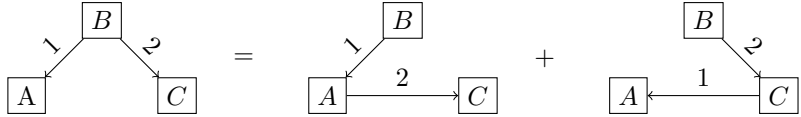
\begin{center}
\begin{tikzpicture}
\node[draw, rectangle] (1A) at (0,0) {A};
\node[draw, rectangle] (1B) at (1,1) {$B$};
\node[draw, rectangle] (1C) at (2,0) {$C$};
\path[->, style={sloped}] (1B) edge node[above]{$1$} (1A);
\path[->, style={sloped}] (1B) edge node[above]{$2$} (1C);
\node[text width=0.5 cm] at (3, 0.5) {$=$};
\node[draw, rectangle] (2A) at (4, 0) {$A$};
\node[draw, rectangle] (2B) at (5, 1) {$B$};
\node[draw, rectangle] (2C) at (6, 0) {$C$};
\path[->, style={sloped}] (2B) edge node[above]{$1$} (2A);
\path[->] (2A) edge node[above]{$2$} (2C);
\node[text width=0.5 cm] at (7, 0.5) {$+$};
\node[draw, rectangle] (3A) at (8, 0) {$A$};
\node[draw, rectangle] (3B) at (9, 1) {$B$};
\node[draw, rectangle] (3C) at (10, 0) {$C$};
\path[->] (3C) edge node[above]{$1$} (3A);
\path[->, style={sloped}] (3B) edge node[above]{$2$} (3C);
\end{tikzpicture}
\end{center}
The cell corresponding to the left hand side is the union of the cells corresponding to the trees on the right hand size.

Relation 3 comes from looking at the boundary of the cell corresponding to:
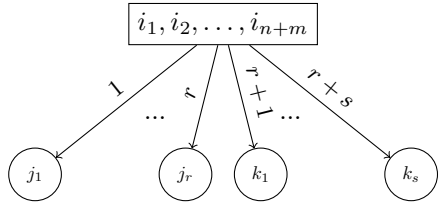
\begin{center}
\begin{tikzpicture}
\node[draw, rectangle] (r) at (1.5, 1) {$i_1, i_2, \ldots, i_{n+m}$};
\node[draw, circle, minimum size=1cm, scale = .7] (j1) at (-1, -1) {$j_1$};
\node[text width = .3 cm] at (0.6, -.25) {$...$};
\node[draw, circle, minimum size=1cm, scale = .7] (j5) at (1, -1) {$j_{r}$};
\node[draw, circle, minimum size=1cm, scale = .7] (k1) at (2, -1) {$k_1$};
\node[text width = .3 cm] at (2.4, -.25) {$...$};
\node[draw, circle, minimum size=1cm, scale = .7] (ks) at (4, -1) {$k_s$};
\foreach \label/\dest in {1/j1, {r}/j5, r+1/k1, r+s/ks}
	\path[->, style={sloped}] (r) edge node[above]{\small{$\label$}} (\dest);
\end{tikzpicture}
\end{center}

Relation 4(a) comes from looking at the boundary of the cell corresponding to:
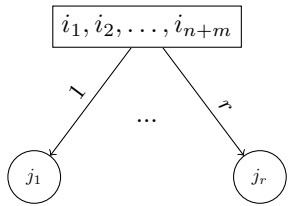
\begin{center}
\begin{tikzpicture}
\node[draw, rectangle] (r) at (1, 1) {$i_1, i_2, \ldots, i_{n+m}$};
\node[draw, circle, minimum size=1cm, scale = .7] (j1) at (-0.5, -1) {$j_1$};
\node[text width=.3 cm] at (1, -.25) {$...$};
\node[draw, circle, minimum size=1cm, scale = .7] (j5) at (2.5, -1) {$j_{r}$};
\foreach \label/\dest in {1/j1, {\hspace{.45 cm} r}/j5}
	\path[->, style ={sloped}] (r) edge node[above]{\small{$\label$}} (\dest);
\end{tikzpicture}
\end{center}

Relations 4(b) and 4(c) comes from looking at similar cells.

Relation 4(d) comes from looking at the boundary of the cell corresponding to:

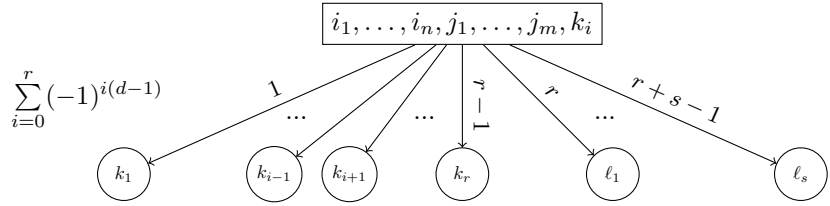
\begin{center}
\begin{tikzpicture}
\node[text width = 8cm] at (1.5,0) {$\sum\limits_{i=0}^r (-1)^{i(d-1)}$};
\node[draw, rectangle](r) at (3.5,1){$i_1, \ldots, i_n, j_1, \ldots, j_m, k_i$};
\node[draw, circle, minimum size=1cm, scale = .7] (i1) at (-1, -1) {$k_1$};
\node[text width=.3 cm] at (1.3, -.25) {$...$};
\node[draw, circle, minimum size=1cm, scale = .7] (i2) at (1, -1) {$k_{i-1}$};
\node[draw, circle, minimum size=1cm, scale = .7] (i3) at (2, -1) {$k_{i+1}$};
\node[text width = .3 cm] at (3, -.25) {$...$};
\node[draw, circle, minimum size=1cm, scale = .7] (i4) at (3.5, -1) {$k_{r}$};
\foreach \label/\dest in {1/i1, {}/i2, {}/i3, {\hspace{.25 cm} r-1}/i4}
	\path[->,style={sloped}] (r) edge node[above]{\small{$\label$}} (\dest);
\node[draw, circle, minimum size=1cm, scale = .7] (j1) at (5.5, -1) {$\ell_1$};
\node[text width = .3 cm] at (5.4, -.25) {$...$};
\node[draw, circle, minimum size=1cm, scale = .7] (j4) at (8, -1) {$\ell_{s}$};
\foreach \label/\dest in {\hspace{.1 cm} r/j1, {\hspace{.75 cm} r+s-1}/j4}
	\path[->, style={sloped}] (r) edge node[above]{\small{$\label$}} (\dest);
\end{tikzpicture}
\end{center}

Relation 4(e) comes from looking at the boundary of the cell corresponding to:

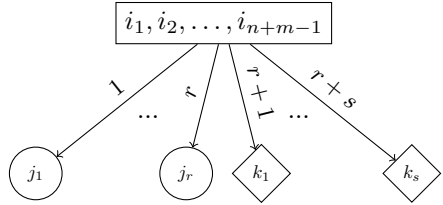
\begin{center}
\begin{tikzpicture}
\node[draw, rectangle] (r) at (1, 1) {$i_1, i_2, \ldots, i_{n+m-1}$};
\node[draw, circle, minimum size=1cm, scale = .7] (j1) at (-1.5, -1) {$j_1$};
\node[text width = .3 cm] at (0, -.25) {$...$};
\node[draw, circle, minimum size=1cm, scale = .7] (j5) at (0.5, -1) {$j_{r}$};
\node[draw, diamond, minimum size=1cm, scale = .7] (k1) at (1.5, -1) {$k_1$};
\node[text width = .3 cm] at (2, -.25) {$...$};
\node[draw, diamond, minimum size=1cm, scale = .7] (ks) at (3.5, -1) {$k_s$};
\foreach \label/\dest in {1/j1, {r}/j5, r+1/k1, r+s/ks}
	\path[->, style={sloped}] (r) edge node[above]{\small{$\label$}} (\dest);
\end{tikzpicture}
\end{center}

\end{proof}

The only relation that does not work when $d=1$ is relation 1(c). There is a substitute for 1(c) in the case that $d=1$:

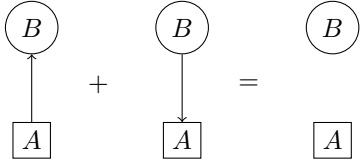
\begin{center}
\begin{tikzpicture}
\node[draw, rectangle] (1A) at (0,0) {$A$};
\node[draw, circle] (1B) at (0,1.5) {$B$};
\path[->] (1A) edge (1B);
\node[text width=0.5 cm] at (1, .75) {$+$};
\node[draw, rectangle] (2A) at (2, 0) {$A$};
\node[draw, circle] (2B) at (2, 1.5) {$B$};
\path[->] (2B) edge (2A);
\node[text width=0.5 cm] at (3, 0.75) {$=$};
\node[draw, rectangle] (3A) at (4, 0) {$A$};
\node[draw, circle] (3B) at (4, 1.5) {$B$};
\end{tikzpicture}
\end{center}

Here, the circle may be replaced with a rectangle. One may also replace the circles on the left hand side of the equation with diamonds. In the event that this causes a rectangle that should have diamonds attached to it to no longer have any diamonds attached to it, the right hand side is zero.

Using these relations, we produce bases for cohomology similar to that from section \ref{cohom}

\begin{defin}
Define a \emph{linear $I$-tree} to be an admissible tree of the following form:

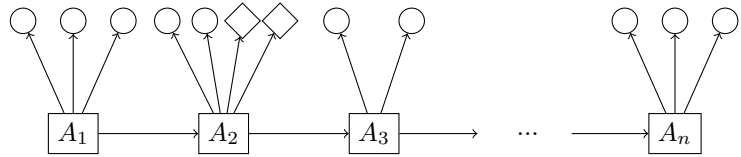
\begin{center}
\begin{tikzpicture}
\foreach \label/\x/\numc/\numd in {1/2/3/0,2/4/2/2,3/6/2/0, n/10/3/0}
	{
	\node[draw, rectangle] (r\label) at (\x,0){$A_{\label}$};
	\foreach \m in {1,...,\numc}
		{
		\node[draw, circle] (s\label\m) at ({\x-1+2*(\m-0.5)/(\numc+\numd)}, 1.5) {} ;
		\path[->] (r\label) edge (s\label\m);
		}
	\ifnum \numd > 0
	{
	\foreach \m in {1,...,\numd}
		{
		\node[draw, diamond] (s\label\m) at ({\x-1+2*(\m+\numc-0.5)/(\numc+\numd)}, 1.5) {} ;
		\path[->] (r\label) edge (s\label\m);
		}		
	}
	\fi
	}
\node[text width = 1cm](r4) at (8,0){$\hspace{.4cm}...$};
\path[->] (r1) edge (r2);
\path[->] (r2) edge (r3);
\path[->] (r3) edge (r4);
\path[->] (r4) edge (rn);
\end{tikzpicture}
\end{center}

such that
\begin{itemize}
\item The elements in $A_i$ appear elements from $X$ first, then $Y$, all in their linear order
\item The circles and diamonds attached to each rectangle are ordered similarly
\item All elements in diamonds are from $Y$
\item The minimal $X$ element in the tree is in $B_1$
\item For each $i$, let $\ell^1_i$ be the maximum element from $X$ in $B_i$; let $\ell^2_i$ be the maximum element from $Y$ in $B_i$
\begin{itemize}
\item If $A_i$ contains $n$ elements from $X$ and $m$ elements from $Y$ for some $(n, m) \in \Cp$:
\begin{itemize}
\item If $(n-1, m+2)\in I$, then $\ell^1_i \notin A_i$ 
\item If $(n+1, m-1) \in I$, then $\ell^2_i \notin A_i$ 
\item If $(n-1, m+2), (n+1, m-1) \notin I$, then $\ell^1_i \notin A_i$ or $\ell^2_i \notin A_i$ 
\end{itemize}
\item If $A_i$ contains $n$ elements from $X$ and $m$ elements from $Y$ for some $(n+1, m+1) \in \Dp$:
\begin{itemize}
\item If $n=0$, then $\ell^1_i \notin A_i$ 
\item If $m=0$, then $\ell^2_i \notin A_i$ 
\item If $n, m \neq 0$, then $\ell^1_i \notin A_i$ and $\ell^2_i \notin A_i$ 
\end{itemize}
\end{itemize}
\end{itemize}
where $B_i$ is the set of elements in $A_i$, circles attached to $A_i$, and diamonds attached to $A_i$.
\end{defin}

Using the relations from Lemma \ref{cohom_relations_2}, any admissible forest can be written as a forest whose components are linear $I$-trees and singleton circles. We will show that this is a basis for $H^* \M(n_1, n_2)$. For $d>1$, this basis will be dual to the generating set for homology from Theorem \ref{hom_basis_2}.

\begin{defin}
Let $\mathcal{H}$ be the set of generators given in Theorem \ref{hom_basis_2}. Let $\mathcal{C}$ be the set of cohomology classes represented by products of linear $I$-trees and singleton circles. Define $f: \mathcal{H} \to \mathcal{C}$ as follows:

Let $A \in \mathcal{H}$. Let $f(A)$ be the forest satisfying the following:
\begin{itemize}
\item For each $x_i$ factor in $A$, $f(A)$ has a singleton circle containing $x_i$
\item For each $y_i$ factor in $A$, $f(A)$ has a singleton circle containing $y_i$
\item Each other factor in $A$ has a corresponding linear $I$-tree as follows:
Suppose the factor is given by $[ \ldots [ [B_1, B_2], B_3] \ldots B_{\ell}]$, then each $B_i$ has a corresponding rectangle vertex, $A_i$. These rectangle vertices form a path from $A_1$ to $A_\ell$. For each $B_i$, the corresponding rectangle vertex has the following form:
\begin{itemize}
\item if $B_i = [ \ldots [ [ \ldots [\{x_{i_1}, \ldots, x_{i_n}, y_{j_1}, \ldots, y_{j_m} \}, x_{i'_1}] \ldots x_{i'_{a}} ], y_{j'_1}] \ldots y_{j'_b}]$ where $(n, m) \in \C$:

\begin{itemize}
\item if $(n-1, m)\in \Cp$ and $i_n > i'_a$, then $A_i$ contains $x_{i_1}, \ldots, x_{i_{n-1}}, y_{j_1}, \ldots, y_{j_m}$, all other elements in $B_i$ correspond to circles attached to $A_i$
\item if the previous condition does not hold and $(n, m-1) \in \Cp$, then $A_i$ contains $x_{i_1}, \ldots, x_{i_n}, y_{j_1}, \ldots, y_{j_{m-1}}$, all other elements in $B_i$ correspond to circles attached to $A_i$
\item if $(n, m) \in \Dp$, then $A_i$ contains $x_{i_1}, \ldots, x_{i_{n-1}}, y_{j_1}, \ldots, y_{j_{m-1}}$, all other $x_j$ in $B_i$ correspond to circles attached to $A_i$, all other $y_j$ in $B_i$ correspond to diamonds attached to $A_i$
\end{itemize}

\item if $B_i = [ \ldots [ [ \ldots [\{x_{i_1}, \ldots, x_{i_n}, \{y_{j_1}, \ldots, y_{j_m} \} \}, x_{i'_1}] \ldots x_{i'_{a}} ], y_{j'_1}] \ldots y_{j'_b}]$ for $(n-1, m-1) \in \D$, then $A_i$ contains $x_{i_1}, \ldots, x_{i_{n-1}}, y_{j_1}, \ldots, y_{j_{m-1}}$, all other elements in $B_i$ correspond to circles attached to $A_i$.
\end{itemize}
\end{itemize}
\end{defin}

It is again a (somewhat tedious) exercise to check $f$ is a bijection.

Order $\mathcal{H}$ such that if $A$ has more rectangles than $B$, then $A$ comes after $B$. Order $\mathcal{C}$ according to the ordering of corresponding elements in $\mathcal{H}$.

\begin{theorem}
\label{diag_proof_2}
With this ordering, the intersection pairing matrix is diagonal with $\pm 1$ on the diagonal.
\end{theorem}

\begin{cor}
$\mathcal{H}$ and $\mathcal{C}$ are bases for $H_* \M(n_1, n_2)$ and $H^* \M(n_1, n_2)$, respectively.
\end{cor}

The proof of Theorem \ref{diag_proof_2} is essentially the same as that of Theorem \ref{diag_proof}.


\subsubsection{Multiplicative Structure}

The addition of diamonds causes us to have additional rules for multiplication not in the previous section.

\begin{defin}
Let $T_1, T_2 \in H^* \M(n_1, n_2)$ be two admissible forests. Suppose for all rectangles $A$ in $T_1$ and $B$ in $T_2$, $A \cap B = \emptyset$. Let $T_1 \cup T_2$ be the tree defined as follows:
\begin{itemize}
\item if $i, j$ are in a common rectangle in $T_1$ or $T_2$, then $i, j$ are in a common rectangle in $T_1 \cup T_2$
\item if $i$ is in a circle in both $T_1$ and $T_2$, then $i$ is in a circle in $T_1 \cup T_2$
\item if $i$ is in a diamond in both $T_1$ and $T_2$, then $i$ is in a diamond in $T_1 \cup T_2$
\item if $i$ is in a diamond in $T_1$ and is in a circle attached to nothing in $T_2$, then $i$ is in a diamond in $T_1 \cup T_2$ (likewise switching $T_1$ and $T_2$)
\item if $i$ is in a diamond in $T_1$ and is in a circle attached to a rectangle in $T_2$, then $i$ is in a star in $T_1 \cup T_2$
\item if $i \in A, j\in B$ in $T_k$ and there exists an edge from $A$ to $B$ in $T_k$, then there exists an edge from the vertex containing $i$ to the vertex containing $j$ in $T_1 \cup T_2$
\end{itemize}
\end{defin}

\begin{theorem}
\label{mult_2}
Let $T_1, T_2 \in H^* \M (n_1, n_2)$ be two admissible forests. The product of $T_1$ and $T_2$, $T_1 \cdot T_2$, is given as follows:

\begin{enumerate}
\item If there exists rectangles $A$ in $T_1$ and $B$ in $T_2$ such that $A \cap B \neq \emptyset$, then $T_1 \cdot T_2 = 0$.
\item If there exists two indices that are in a common tree in both $T_1$ and $T_2$, then $T_1 \cdot T_2 = 0$
\item If $T_1 \cup T_2$ has a cycle, then $T_1 \cdot T_2 = 0$.
\item If $T_1 \cup T_2$ has a rectangle with no circles attached to it, then $T_1 \cdot T_2 = 0$
\item If $T_1 \cup T_2$ has a rectangle that should have diamonds attached to it but doesn't, then $T_1 \cdot T_2 = 0$.
\item If $T_1 \cup T_2$ is an admissible forest, then $T_1 \cdot T_2 = T_1 \cup T_2$ with orientation set given by concatenation.
\item If $T_1 \cup T_2$ satisfies none of the above, then use the following relations to make $T_1 \cup T_2$ admissible
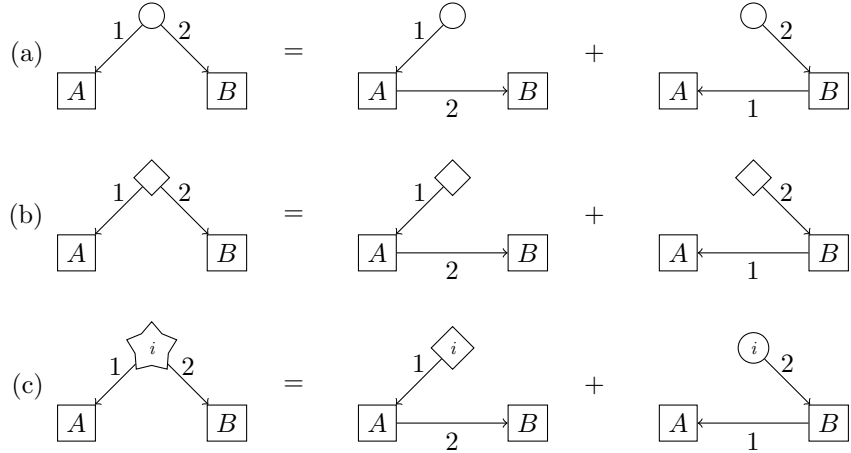
\begin{enumerate}
\vspace{0.4cm}
\item 
\begin{center}
\begin{tikzpicture}[baseline=0.4 cm]
\node[draw, rectangle] (1A) at (0,0) {$A$};
\node[draw, circle] (1B) at (1,1) {};
\node[draw, rectangle] (1C) at (2,0) {$B$};
\path[->] (1B) edge node[above]{$2$} (1C);
\path[->] (1B) edge node[above]{$1$} (1A);
\node[text width=0.5 cm] at (3, 0.5) {$=$};
\node[draw, rectangle] (2A) at (4, 0) {$A$};
\node[draw, circle] (2B) at (5, 1) {};
\node[draw, rectangle] (2C) at (6, 0) {$B$};
\path[->] (2B) edge node[above]{$1$} (2A);
\path[->] (2A) edge node[below]{$2$}(2C);
\node[text width=0.5 cm] at (7, 0.5) {$+$};
\node[draw, rectangle] (3A) at (8, 0) {$A$};
\node[draw, circle] (3B) at (9, 1) {};
\node[draw, rectangle] (3C) at (10, 0) {$B$};
\path[->] (3B) edge node[above]{$2$} (3C);
\path[->] (3C) edge node[below]{$1$} (3A);
\end{tikzpicture}
\end{center}
\vspace{0.4cm}
\item 
\begin{center}
\begin{tikzpicture}[baseline=0.4 cm]
\node[draw, rectangle] (1A) at (0,0) {$A$};
\node[draw, diamond] (1B) at (1,1) {};
\node[draw, rectangle] (1C) at (2,0) {$B$};
\path[->] (1B) edge node[above]{$2$} (1C);
\path[->] (1B) edge node[above]{$1$} (1A);
\node[text width=0.5 cm] at (3, 0.5) {$=$};
\node[draw, rectangle] (2A) at (4, 0) {$A$};
\node[draw, diamond] (2B) at (5, 1) {};
\node[draw, rectangle] (2C) at (6, 0) {$B$};
\path[->] (2B) edge node[above]{$1$} (2A);
\path[->] (2A) edge node[below]{$2$}(2C);
\node[text width=0.5 cm] at (7, 0.5) {$+$};
\node[draw, rectangle] (3A) at (8, 0) {$A$};
\node[draw, diamond] (3B) at (9, 1) {};
\node[draw, rectangle] (3C) at (10, 0) {$B$};
\path[->] (3B) edge node[above]{$2$} (3C);
\path[->] (3C) edge node[below]{$1$} (3A);
\end{tikzpicture}
\end{center}
\vspace{0.4cm}
\item 
\begin{center}
\begin{tikzpicture}[baseline=0.4 cm]
\node[draw, rectangle] (1A) at (0,0) {$A$};
\node[draw, star,star points=5, scale=.7] (1B) at (1,1) {i};
\node[draw, rectangle] (1C) at (2,0) {$B$};
\path[->] (1B) edge node[above]{$2$} (1C);
\path[->] (1B) edge node[above]{$1$} (1A);
\node[text width=0.5 cm] at (3, 0.5) {$=$};
\node[draw, rectangle] (2A) at (4, 0) {$A$};
\node[draw, diamond, scale = 0.7] (2B) at (5, 1) {i};
\node[draw, rectangle] (2C) at (6, 0) {$B$};
\path[->] (2B) edge node[above]{$1$} (2A);
\path[->] (2A) edge node[below]{$2$}(2C);
\node[text width=0.5 cm] at (7, 0.5) {$+$};
\node[draw, rectangle] (3A) at (8, 0) {$A$};
\node[draw, circle, scale= 0.7] (3B) at (9, 1) {i};
\node[draw, rectangle] (3C) at (10, 0) {$B$};
\path[->] (3B) edge node[above]{$2$} (3C);
\path[->] (3C) edge node[below]{$1$} (3A);
\end{tikzpicture}
\end{center}
where $A$ is the vertex that $i$ is attached to as a diamond in $T_1$ or $T_2$ and $B$ is the vertex that $i$ is attached to as a circle in $T_1$ or $T_2$.
\end{enumerate}
\end{enumerate}
\end{theorem}

\begin{proof}
\begin{enumerate}
\item If $A \neq B$, then the corresponding chains to $T_1$ and $T_2$ do not intersect in $\M (n_1, n_2)$. If $A = B$, then one can perturb the chains slightly so that they do not intersect.
\item There exist orientations of edges such that the two corresponding chains do not intersect in $\M (n_1, n_2)$.
\item There exist orientations of edges such that the two corresponding chains do not intersect in $\M (n_1, n_2)$.
\item This is relation 4 from Lemma \ref{cohom_relations_2} with $r = 0$ or $r = s = 0$.
\item In this case, the two corresponding chains do not intersect in $\M (n_1, n_2)$.
\item The chains corresponding to $T_1$ and $T_2$ are transversal and their intersection is the chain corresponding to $T_1 \cup T_2$.
\item Combine the proofs of (6) and relation 2 from Lemma \ref{cohom_relations_2}.
\end{enumerate}
\end{proof}

\subsubsection{$I$ not a rectangle and ${\{(2, 0), (1, 1), (0, 2)\} \not\subset I}$}

The difficulty in applying the same definition of admissible forests to this case is the fact that there exist weight one rectangles. This case splits into two sub-cases: whether $(1, 1)$ is in $\Dp$ or not. In the case that $(1, 1) \notin \Dp$, the above construction can be altered so that it works with weight one rectangles. We keep the same space of forests as in the general case except that we now allow rectangles to have no circles attached to them. In addition, because when it comes to the corresponding chain, there is no difference between rectangles containing one element and circles, we add the relation that if a weight one rectangle is attached to at most one rectangle and nothing else, it may be turned into a circle. Similarly, we may turn a circle containing an element of $M$ to a rectangle, provided such rectangles are allowed (similarly for $N$). With these changes, our basis consisting of linear $I$-trees and singletons is also a basis in this scenario. The only change in multiplication is condition 1 from Theorem \ref{mult_2} must additionally assume that $A$ and $B$ both have weight at least $2$.

In the case $(1, 1) \notin \Dp$, this construction does not work. An enlightening example is the following tree:
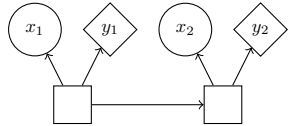
\begin{center}
\begin{tikzpicture}
\node[draw, rectangle, minimum size = .5cm] (r1) at (0, 0) {};
\node[draw, circle, minimum size=1cm, scale = .7] (c1) at (-.5, 1) {$x_1$};
\node[draw, diamond, minimum size=1cm, scale = .7] (d1) at (.5, 1) {$y_1$};
\node[draw, rectangle, minimum size = .5cm] (r2) at (2, 0) {};
\node[draw, circle, minimum size=1cm, scale = .7] (c2) at (1.5, 1) {$x_2$};
\node[draw, diamond, minimum size=1cm, scale = .7] (d2) at (2.5, 1) {$y_2$};
\path[->, style={sloped}] (r1) edge (r2);
\path[->, style={sloped}] (r1) edge (c1);
\path[->, style={sloped}] (r1) edge (d1);
\path[->, style={sloped}] (r2) edge(c2);
\path[->, style={sloped}] (r2) edge (d2);
\end{tikzpicture}
\end{center}

If the previous space of forests were applicable to the case $(1, 1) \in \Dp, d>1$, then this tree should be dual to the homology element $[\{x_1, y_1\}, \{x_2, y_2\}]$. The problem is the boundary of the cell corresponding to this tree does not live in the complement of $\M(2, 2)$. One way to remedy this is to allow for circle vertices to be connected to two rectangles. Instead of the above tree, we could take the following tree:

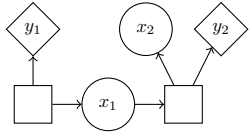
\begin{center}
\begin{tikzpicture}
\node[draw, rectangle, minimum size = .5cm] (r1) at (0, 0) {};
\node[draw, diamond, minimum size=1cm, scale = .7] (d1) at (0, 1) {$y_1$};
\node[draw, rectangle, minimum size = .5cm] (r2) at (2, 0) {};
\node[draw, circle, minimum size=1cm, scale = .7] (c1) at (1, 0) {$x_1$};
\node[draw, circle, minimum size=1cm, scale = .7] (c2) at (1.5, 1) {$x_2$};
\node[draw, diamond, minimum size=1cm, scale = .7] (d2) at (2.5, 1) {$y_2$};
\path[->, style={sloped}] (c1) edge (r2);
\path[->, style={sloped}] (r1) edge (c1);
\path[->, style={sloped}] (r1) edge (d1);
\path[->, style={sloped}] (r2) edge(c2);
\path[->, style={sloped}] (r2) edge (d2);
\end{tikzpicture}
\end{center}

In fact, consider linear $I$-trees with the following change: for each weight one rectangle that is not last in the chain of rectangles, the maximum circle attached to it is between it and the next rectangle in the chain. Then our basis of products of linear $I$-trees and singleton circles becomes a basis in this situation. This raises a few questions: what is the full space of forests analogous to the previous situations? How does multiplication behave?

The first question does not have a clear answer. One possibility is that each weight one rectangle should have a circle between it and any other rectangle. Alternatively, one could restrict so that this only need be true between two weight one rectangles. In either case, we want to write any tree as a sum of linear $I$-trees; that is, we want to be able to decrease the degree of rectangles. We consider an example:

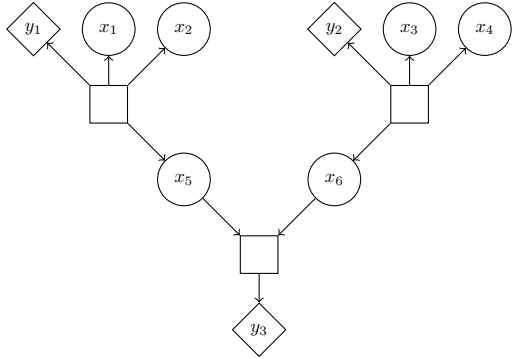
\begin{center}
\begin{tikzpicture}
\node[draw, rectangle, minimum size = .5cm] (r1) at (0, 0) {};
\node[draw, rectangle, minimum size = .5cm] (r2) at (4, 0) {};
\node[draw, rectangle, minimum size = .5cm] (r3) at (2, -2) {};
\node[draw, diamond, minimum size=1cm, scale = .7] (d1) at (-1, 1) {$y_1$};
\node[draw, circle, minimum size=1cm, scale = .7] (c1) at (0, 1) {$x_1$};
\node[draw, circle, minimum size=1cm, scale = .7] (c2) at (1, 1) {$x_2$};
\node[draw, diamond, minimum size=1cm, scale = .7] (d2) at (3, 1) {$y_2$};
\node[draw, circle, minimum size=1cm, scale = .7] (c3) at (4, 1) {$x_3$};
\node[draw, circle, minimum size=1cm, scale = .7] (c4) at (5, 1) {$x_4$};
\node[draw, circle, minimum size=1cm, scale = .7] (ca) at (1, -1) {$x_5$};
\node[draw, circle, minimum size=1cm, scale = .7] (cb) at (3, -1) {$x_6$};
\node[draw, diamond, minimum size=1cm, scale = .7] (d3) at (2, -3) {$y_3$};
\path[->, style={sloped}] (r1) edge (d1);
\path[->, style={sloped}] (r1) edge (c1);
\path[->, style={sloped}] (r1) edge (c2);
\path[->, style={sloped}] (r1) edge (ca);
\path[->, style={sloped}] (r2) edge (d2);
\path[->, style={sloped}] (r2) edge (c3);
\path[->, style={sloped}] (r2) edge (c4);
\path[->, style={sloped}] (r2) edge (cb);
\path[->, style={sloped}] (ca) edge (r3);
\path[->, style={sloped}] (cb) edge(r3);
\path[->, style={sloped}] (r3) edge (d3);
\end{tikzpicture}
\end{center}

We may have a tree that has this as a subtree in it and has the bottom empty rectangle attached to other rectangles. Following the same idea in the proof of the three term relation in Lemma \ref{cohom_relations}, we get this tree is equal to the following:

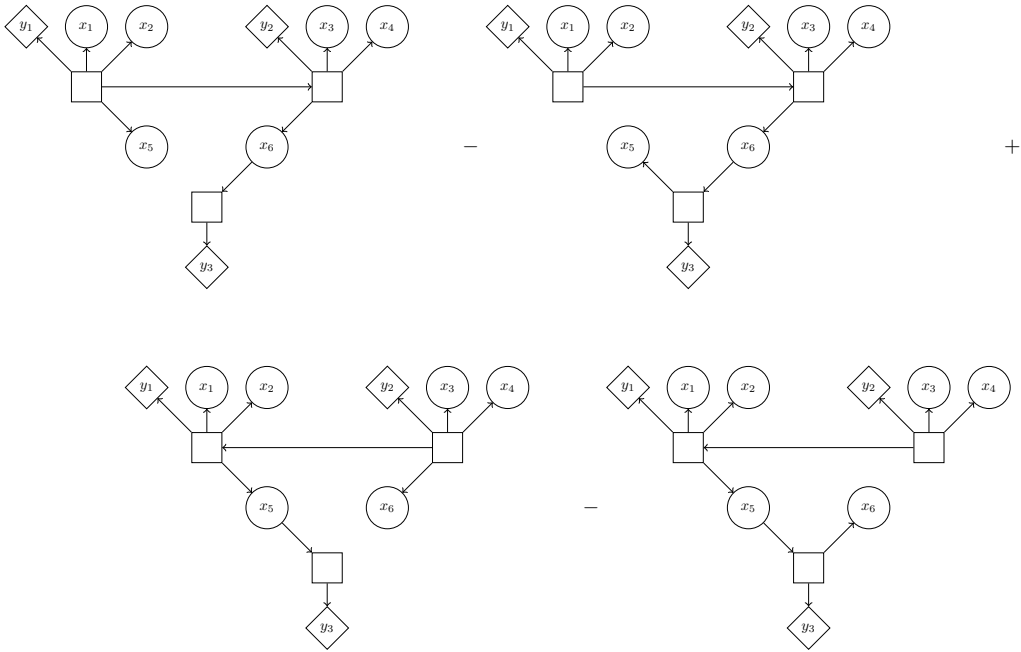
\begin{center}
\scalebox{.8}{
\begin{tikzpicture}
\node[draw, rectangle, minimum size = .5cm] (r1) at (0, 0) {};
\node[draw, rectangle, minimum size = .5cm] (r2) at (4, 0) {};
\node[draw, rectangle, minimum size = .5cm] (r3) at (2, -2) {};
\node[draw, diamond, minimum size=1cm, scale = .7] (d1) at (-1, 1) {$y_1$};
\node[draw, circle, minimum size=1cm, scale = .7] (c1) at (0, 1) {$x_1$};
\node[draw, circle, minimum size=1cm, scale = .7] (c2) at (1, 1) {$x_2$};
\node[draw, diamond, minimum size=1cm, scale = .7] (d2) at (3, 1) {$y_2$};
\node[draw, circle, minimum size=1cm, scale = .7] (c3) at (4, 1) {$x_3$};
\node[draw, circle, minimum size=1cm, scale = .7] (c4) at (5, 1) {$x_4$};
\node[draw, circle, minimum size=1cm, scale = .7] (ca) at (1, -1) {$x_5$};
\node[draw, circle, minimum size=1cm, scale = .7] (cb) at (3, -1) {$x_6$};
\node[draw, diamond, minimum size=1cm, scale = .7] (d3) at (2, -3) {$y_3$};
\path[->, style={sloped}] (r1) edge (d1);
\path[->, style={sloped}] (r1) edge (c1);
\path[->, style={sloped}] (r1) edge (c2);
\path[->, style={sloped}] (r1) edge (ca);
\path[->, style={sloped}] (r2) edge (d2);
\path[->, style={sloped}] (r2) edge (c3);
\path[->, style={sloped}] (r2) edge (c4);
\path[->, style={sloped}] (r2) edge (cb);
\path[->, style={sloped}] (r1) edge (r2);
\path[->, style={sloped}] (cb) edge(r3);
\path[->, style={sloped}] (r3) edge (d3);

\node[text width = .5 cm] at (6.5, -1) {$-$};

\node[draw, rectangle, minimum size = .5cm] (r1) at (8, 0) {};
\node[draw, rectangle, minimum size = .5cm] (r2) at (12, 0) {};
\node[draw, rectangle, minimum size = .5cm] (r3) at (10, -2) {};
\node[draw, diamond, minimum size=1cm, scale = .7] (d1) at (7, 1) {$y_1$};
\node[draw, circle, minimum size=1cm, scale = .7] (c1) at (8, 1) {$x_1$};
\node[draw, circle, minimum size=1cm, scale = .7] (c2) at (9, 1) {$x_2$};
\node[draw, diamond, minimum size=1cm, scale = .7] (d2) at (11, 1) {$y_2$};
\node[draw, circle, minimum size=1cm, scale = .7] (c3) at (12, 1) {$x_3$};
\node[draw, circle, minimum size=1cm, scale = .7] (c4) at (13, 1) {$x_4$};
\node[draw, circle, minimum size=1cm, scale = .7] (ca) at (9, -1) {$x_5$};
\node[draw, circle, minimum size=1cm, scale = .7] (cb) at (11, -1) {$x_6$};
\node[draw, diamond, minimum size=1cm, scale = .7] (d3) at (10, -3) {$y_3$};
\path[->, style={sloped}] (r1) edge (d1);
\path[->, style={sloped}] (r1) edge (c1);
\path[->, style={sloped}] (r1) edge (c2);
\path[->, style={sloped}] (r3) edge (ca);
\path[->, style={sloped}] (r2) edge (d2);
\path[->, style={sloped}] (r2) edge (c3);
\path[->, style={sloped}] (r2) edge (c4);
\path[->, style={sloped}] (r2) edge (cb);
\path[->, style={sloped}] (r1) edge (r2);
\path[->, style={sloped}] (cb) edge(r3);
\path[->, style={sloped}] (r3) edge (d3);

\node[text width = .5 cm] at (15.5, -1) {$+$};

\node[draw, rectangle, minimum size = .5cm] (r1) at (2, -6) {};
\node[draw, rectangle, minimum size = .5cm] (r2) at (6, -6) {};
\node[draw, rectangle, minimum size = .5cm] (r3) at (4, -8) {};
\node[draw, diamond, minimum size=1cm, scale = .7] (d1) at (1, -5) {$y_1$};
\node[draw, circle, minimum size=1cm, scale = .7] (c1) at (2, -5) {$x_1$};
\node[draw, circle, minimum size=1cm, scale = .7] (c2) at (3, -5) {$x_2$};
\node[draw, diamond, minimum size=1cm, scale = .7] (d2) at (5, -5) {$y_2$};
\node[draw, circle, minimum size=1cm, scale = .7] (c3) at (6, -5) {$x_3$};
\node[draw, circle, minimum size=1cm, scale = .7] (c4) at (7, -5) {$x_4$};
\node[draw, circle, minimum size=1cm, scale = .7] (ca) at (3, -7) {$x_5$};
\node[draw, circle, minimum size=1cm, scale = .7] (cb) at (5, -7) {$x_6$};
\node[draw, diamond, minimum size=1cm, scale = .7] (d3) at (4, -9) {$y_3$};
\path[->, style={sloped}] (r1) edge (d1);
\path[->, style={sloped}] (r1) edge (c1);
\path[->, style={sloped}] (r1) edge (c2);
\path[->, style={sloped}] (r1) edge (ca);
\path[->, style={sloped}] (r2) edge (d2);
\path[->, style={sloped}] (r2) edge (c3);
\path[->, style={sloped}] (r2) edge (c4);
\path[->, style={sloped}] (r2) edge (cb);
\path[->, style={sloped}] (ca) edge (r3);
\path[->, style={sloped}] (r2) edge(r1);
\path[->, style={sloped}] (r3) edge (d3);

\node[text width = .5 cm] at (8.5, -7) {$-$};

\node[draw, rectangle, minimum size = .5cm] (r1) at (10, -6) {};
\node[draw, rectangle, minimum size = .5cm] (r2) at (14, -6) {};
\node[draw, rectangle, minimum size = .5cm] (r3) at (12, -8) {};
\node[draw, diamond, minimum size=1cm, scale = .7] (d1) at (9, -5) {$y_1$};
\node[draw, circle, minimum size=1cm, scale = .7] (c1) at (10, -5) {$x_1$};
\node[draw, circle, minimum size=1cm, scale = .7] (c2) at (11, -5) {$x_2$};
\node[draw, diamond, minimum size=1cm, scale = .7] (d2) at (13, -5) {$y_2$};
\node[draw, circle, minimum size=1cm, scale = .7] (c3) at (14, -5) {$x_3$};
\node[draw, circle, minimum size=1cm, scale = .7] (c4) at (15, -5) {$x_4$};
\node[draw, circle, minimum size=1cm, scale = .7] (ca) at (11, -7) {$x_5$};
\node[draw, circle, minimum size=1cm, scale = .7] (cb) at (13, -7) {$x_6$};
\node[draw, diamond, minimum size=1cm, scale = .7] (d3) at (12, -9) {$y_3$};
\path[->, style={sloped}] (r1) edge (d1);
\path[->, style={sloped}] (r1) edge (c1);
\path[->, style={sloped}] (r1) edge (c2);
\path[->, style={sloped}] (r1) edge (ca);
\path[->, style={sloped}] (r2) edge (d2);
\path[->, style={sloped}] (r2) edge (c3);
\path[->, style={sloped}] (r2) edge (c4);
\path[->, style={sloped}] (r2) edge (r1);
\path[->, style={sloped}] (ca) edge (r3);
\path[->, style={sloped}] (r3) edge(cb);
\path[->, style={sloped}] (r3) edge (d3);

\end{tikzpicture}
}
\end{center}

These are not in the space of allowable trees; in each of the four summands, there are weight one rectangles adjacent to each other. Thus, we need to ``simplify" these more. This is possible to do, but it depends on what other rectangles the upper two rectangles are connected to. In each step in the simplification, there are trees not in the space of admissible trees. It may be possible to redefine admissible trees so that these intermediate trees are admissible. The trouble with that is the chains corresponding to each individual tree does not have boundary in the complement to $\M$. It is only when considered together that their collective boundary is in the complement.

\section{General Polychromatic Configuration Spaces}\label{tricolored}

Recall the main difference between the homology of decreasing polychromatic configuration spaces and the homology of bicolored configuration spaces: the homology of decreasing polychromatic configuration spaces is generated as an $\mathcal{M}_d$ module by the homology of $\M (\vec{n})$ for $\vec{n} \in \C$ while the homology of bicolored configuration spaces is generated as an $\mathcal{M}_d$ module by the homology of $\M (n_1, n_2)$ for $(n_1, n_2) \in \C$ and $\M (n_1+1, n_2+1)$ for $(n_1, n_2) \in \D$. That is, there exists a new type of class in the general setting that is not present in the decreasing setting. The obvious question to ask next is what happens for higher $m$. For this, we will focus on $m = 3$. However, before doing so, we will discuss a particular example in $m = 2$.

Consider the following ideal in $\N^2$: $I = \{ (\ell_1, \ell_2) | 0 \leq \ell_i \leq 2\text{ for all }i \}$. This is not a particularly interesting example. Thus, consider $I' = I \cup \bigcup_{i=1}^2 \{3 \vec{e_i}\}$. That is $I'$ is just $I$ with one $2$-tuple added along each axis. Now, $H_3 \Mp (3, 3) = \Z$ while $H_3 \M(3, 3) = 0$. A generator for this additional class that appears is $\{x_1, x_2, x_3, \{y_1, y_2, y_3\} \}$.

Now consider the similar construction in $m = 3$. That is, $I = \{ (\ell_1, \ell_2, \ell_3) | 0 \leq \ell_i \leq 2\text{ for all }i \}$ and $I' = I \cup \bigcup_{i=1}^3 \{3 \vec{e_i}\}$. Now, $\Mp (3, 3, 3)$ is the complement to a subspace arrangement. Thus, using the formula of Goresky-MacPherson \cite{GM}, one can compute the cohomology groups, and thus, homology groups of $\Mp (3, 3, 3)$. Using this formula, one can see $H_5 \Mp (3, 3, 3) = \Z$ while $H_5(\M(3,3,3)) = 0$. The question is: what is this class?

An understandable impulsive reaction is $\{x^1_1, x^1_2, x^1_3, \{x^2_1, x^2_2, x^2_3, \{x^3_1, x^3_2, x^3_3 \}\} \}$. However, said class would not live in $\Mp$ for it contains points where $x^1_1 = x^1_2 = x^1_3 = x^2_1=x^2_2$. The same argument disqualifies $\{x^1_1, x^1_2, x^1_3, \{x^2_1, x^2_2, x^2_3\}, \{x^3_1, x^3_2, x^3_3 \} \}$.

Another asymmetry between the two cases presented is the following. When $m = 2$, we can instead consider $I' = I \cup \bigcup_{i=1}^1 \{3 \vec{e_i}\}$. Again, we would have $H_3 \Mp (3, 3) = \Z$. In the $m = 3$ case, we could either consider $I' = I \cup \bigcup_{i=1}^1 \{3 \vec{e_i}\}$ or $I' = I \cup \bigcup_{i=1}^2 \{3 \vec{e_i}\}$. In either case $H_5 \Mp (3, 3, 3) = 0$. Thus, there is an asymmetry in when the new classes appear. Furthermore, one can show it cannot be expressed as iterated curly brackets in the sense of the new class for $m=2$ was. Thus, a new type of bracket must be introduced, say $\{x^1_1, x^1_2, x^1_3, x^2_1, x^2_2, x^2_3, x^3_1, x^3_2, x^3_3\}_2$. One can show that for $I'$, this is the only new bracket that is needed. This leads to further questions which will be mentioned in the following section.

We may not be able to fully compute the homology of general polychromatic configuration spaces, but this does not mean we can't say anything about them.

\begin{theorem} \label{Rep_Stab}
For any ideal $I \subset \N^m$ and $d > 1$, $H_i(\M)$ exhibit representation stability.
\end{theorem}

The theory of representation stability was first introduced by Church and Farb \cite{CF}. Church showed that for any connected, orientable manifold, $M$, $H_i (C_n(M), \Q)$ are representation stable where $C_n(M)$ denotes the configuration space of $n$ points on $M$ \cite{Church}. The proof of Theorem \ref{Rep_Stab} follows from recent work by Gadish \cite{G1, G2}. While Gadish's work focuses on complex subspace arrangements, straightforward alterations allow us to apply it to $\R^d$ when $d > 1$.


\section{Further Questions} \label{further}

If the reader just finished reading the previous section, there should be one obvious question in mind: for a given $m$, how many types of generators must be added to generate $\M$ as a $\mathcal{M}_d$ module? In $m = 1$, there is a single type of generator, in $m = 2$ there are two types of generators. It was shown in the previous section that in $m = 3$, there must be at least $3$ types of generators. However, it is not hard to argue there must be at least $4$, since there exist ideals where $\{x^1_1, x^1_2, x^1_3, \{x^2_1, x^2_2, x^2_3, \{x^3_1, x^3_2, x^3_3 \}\} \}$ is a valid generator.

More than just how many types of generators are needed for general $m$, what are they? For a general polychromatic configuration space, what is its homology generated by as a $\mathcal{M}_d$ module?

Furthermore, we only discussed polychromatic configuration spaces of $\R^d$. The definition of a polychromatic configuration space makes sense for any topological space. What can be said about the homology and cohomology of polychromatic configuration spaces of a general manifold? CW-complex? These questions still need work replacing polychromatic configuration spaces with no-$k$-equal spaces.

Revisiting the bicolored situation, we did not have a description of the cohomology ring when $(1, 1) \in \Dp$. What is a description of the cohomology ring in the case that $(1, 1) \in \Dp$?

One thing that is not so obvious from the way I presented this material is a sort of duality between the conditions for the bases in homology and cohomology. Fix $n_1, n_2 \in \N$. Let $I \subset \N^2$ be an ideal. Let $\bar{I}$ be the ideal such that $(n, m) \in \bar{I}$ if and only if $(n_1-n, n_2-m) \in I$. Call $\bar{I}$ the $(n_1, n_2)-complement$ to $I$. Suppose $(n, m) \in \C$. Then $[ \ldots [ [ \ldots [\{x_{i_1}, \ldots, x_{i_n}, y_{j_1}, \ldots, y_{j_m} \}, x_{i'_1}] \ldots x_{i'_{a}} ], y_{j'_1}] \ldots y_{j'_b}]$ represents an element in $H_* \M(n_1, n_2)$. There are certain ordering conditions that are always present, but we also have a decision tree to determine which additional restrictions are imposed:

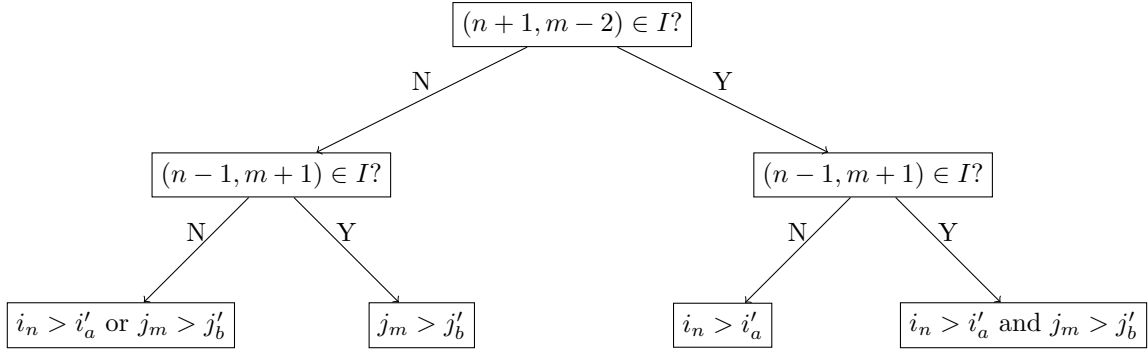
\begin{center}
\begin{tikzpicture}[baseline=0.4 cm]
\node[draw, rectangle] (1) at (0,0) {$(n+1, m-2)\in I ?$};
\node[draw, rectangle] (2A) at (4,-2) {$(n-1, m+1) \in I?$};
\node[draw, rectangle] (2B) at (-4,-2) {$(n-1, m+1) \in I?$};
\path[->] (1) edge node[above]{Y} (2A);
\path[->] (1) edge node[above]{N} (2B);
\node[draw, rectangle] (2AA) at (6, -4) {$i_n > i'_a$ and $j_m > j'_b$};
\node[draw, rectangle] (2AB) at (2, -4) {$i_n > i'_a$};
\node[draw, rectangle] (2BA) at (-2, -4) {$j_m > j'_b$};
\node[draw, rectangle] (2BB) at (-6, -4) {$i_n > i'_a$ or $j_m > j'_b$};
\path[->] (2A) edge node[above]{Y} (2AA);
\path[->] (2A) edge node[above]{N} (2AB);
\path[->] (2B) edge node[above]{Y} (2BA);
\path[->] (2B) edge node[above]{N} (2BB);
\end{tikzpicture}
\end{center}
\vspace{.2 in}

Now suppose $\bar{I}$ is not a rectangle. Then $(\bar{n}, \bar{m}) = (n_1-n, n_2-m) \in \mathcal{C}_{\bar{I}}'$. Consider the conditions for a rectangle vertex corresponding to $(\bar{n}, \bar{m})$:

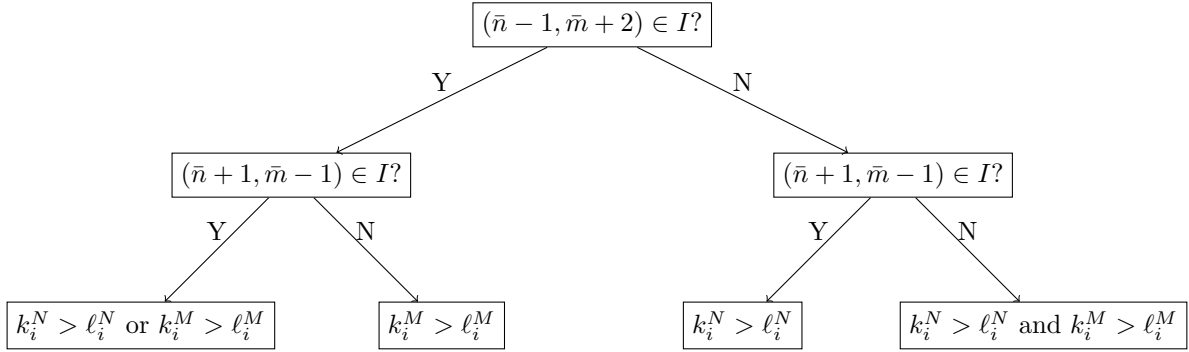
\begin{center}
\begin{tikzpicture}[baseline=0.4 cm]
\node[draw, rectangle] (1) at (0,0) {$(\bar{n}-1, \bar{m}+2)\in I ?$};
\node[draw, rectangle] (2A) at (4,-2) {$(\bar{n}+1, \bar{m}-1) \in I?$};
\node[draw, rectangle] (2B) at (-4,-2) {$(\bar{n}+1, \bar{m}-1) \in I?$};
\path[->] (1) edge node[above]{N} (2A);
\path[->] (1) edge node[above]{Y} (2B);
\node[draw, rectangle] (2AA) at (6, -4) {$k^N_i > \ell^N_i$ and $k^M_i > \ell^M_i$};
\node[draw, rectangle] (2AB) at (2, -4) {$k^N_i > \ell^N_i$};
\node[draw, rectangle] (2BA) at (-2, -4) {$k^M_i > \ell^M_i$};
\node[draw, rectangle] (2BB) at (-6, -4) {$k^N_i > \ell^N_i$ or $k^M_i > \ell^M_i$};
\path[->] (2A) edge node[above]{N} (2AA);
\path[->] (2A) edge node[above]{Y} (2AB);
\path[->] (2B) edge node[above]{N} (2BA);
\path[->] (2B) edge node[above]{Y} (2BB);
\end{tikzpicture}
\end{center}
\vspace{.2 in}

The similarity in these trees is striking and can be attributed to the way the bases of homology and cohomology were chosen. The homology relations proceed by including an extra element in the curly bracket and then taking turns removing each element. The cohomology relations proceed by including one fewer element in the rectangle vertex and then taking turns adding in each element attached to it. These two actions work in harmony to produce the similarity in the restrictions above. One question is: is there anything else? Is this duality in the restrictions simply a coincidence in the basis I chose to work with or is there something deeper relating the spaces $\M (n_1, n_2)$ and $\mathcal{M}_{\bar{I}, d} (n_1, n_2)$? One can furthermore ask the same question for higher $m$.


\begin{appendices}
\section{Appendix}
In this section we prove Lemma \ref{main_lemma}. In order to prove Lemma \ref{main_lemma}, we will first prove a few technical lemmas. These lemmas will decrease $g_1(\gamma)$. For $(n-1, m-1) \in \D$, two of the lemmas will involve $\min\{k : f_I(k) = f_I(m-1)\}$.

\begin{defin}
For all $a \in \N$, let $m_{a}= \min\{k : f_I(k) = f_I(a)\}$.
\end{defin}
 
The proof involves removing from chains their intersections with tubular neighborhoods of subspaces. Many times, these subspaces lie in the complement of $\Mp (n_1, \ldots, n_{\ell_I+3})$. However, there are cases where they do not. For these, we only want to remove tubular neighborhoods of subsets of these subspaces. We restrict by using distances between points. In the following definition and lemmas, we will use the set $\{1, 2, \ldots, n\}$ multiple times. Thus, for notational convenience, we will denote it by $\underline{n}$.

\begin{defin}
For any set $A$ of distinct elements from $\underline{n_2}$, let $\tilde{A} = (\R^d)^{(n_1+\ldots + n_{\ell_I+3})} \cap \{ y_{j_1} = \ldots = y_{j_{k}} \}$ where $A= \{j_1, \ldots, j_k\}$. For any point $\bar{x} \in \gamma\cap\tilde{A}$, let $d_{\gamma, A}^{(b, c)}(\bar{x}) = \min_{K} \max_{k \in K} \{ d(y_{j_1}, k) \}$ where the minimum ranges over all $K$ containing $b$ distinct $x$ coordinates and $c$ distinct $y$ coordinates not in $\{y_j : j \in A\}$.
\end{defin}

That is, $d_{\gamma, A}^{(b, c)}(\bar{x})$ is the minimum radius, $r$, such that the ball of radius $r$ centered at $y_{j_1}$ contains $b$ points of color one and $c$ points of color two not labeled by an element of $A$.

With this notation, we now state and prove the technical lemmas.

\begin{lemma}
\label{minor_lemma_1}
Let $(n_1, n_2) \in \N^2$ with $n_1 > 0$. Suppose organized classes span $\Mp(\tilde{n}_1, \ldots, \tilde{n}_{\ell_I+3})$ for all $(\tilde{n}_1, \tilde{n}_2) < (n_1, n_2)$. Let $a \in N$. Let $\gamma$ be a closed $s$-chain in $\Mp (n_1, \ldots, n_{\ell_I+3})$ such that $g_0(\gamma) \leq f_I(1)+1$ and $g_1(\gamma) \leq f_I(a)$. Furthermore, suppose $f_I(a+1) < f_I(a)-1$ and $f_I(a-1) = f_I(a)$. Then $[\gamma] = [\gamma_1] + [\gamma_2]$ where $[\gamma_1]$ is organized and $\gamma_2$ satisfies one the following:
\begin{itemize}
\item $[\gamma_2]=0$
\item there exist $0<\epsilon_1<\epsilon_2$ such that $\gamma_2$ satisfies the following:
\begin{itemize}
\item $g_0(\gamma_2) \leq f_I(1)+1$
\item $g_1(\gamma_2) \leq f_I(a)$
\item for any $A$ with $|A| = f_I(a)+1$ and any $\bar{x} \in \gamma_2 \cap \tilde{A}, d_{\gamma_1, A}^{(m_{a}, 0)}(\bar{x}) > \epsilon_2$
\item for any $A$ with $|A| = f_I(a)$ and any $\bar{x} \in \gamma_2 \cap \tilde{A}, d_{\gamma_1, A}^{(m_{a}, 1)}(\bar{x}) > \epsilon_2$
\item for any $A$ with $|A| = f_I(a)$ and any $\bar{x} \in \gamma_2 \cap \tilde{A}, d_{\gamma_1, A}^{(m_{a}, 0)}(\bar{x}) \in [0, \epsilon_1) \cup (\epsilon_2, \infty)$
\end{itemize}
\end{itemize}
\end{lemma}

\begin{proof}
In this case, $(a, f_I(a)) \in \D$; suppose $a = \alpha_k$. We will use induction to prove a more general statement. We will show by induction that for all $q \leq n_1$, $[\gamma] = [\gamma_1^q] + [\gamma_2^q]$ where $[\gamma_1^q]$ is organized and $\gamma_2^q$ satisfies one of the following:
\begin{itemize}
\item $[\gamma_2^q]=0$
\item there exist $0<\epsilon^q_1<\epsilon^q_2$ such that $\gamma_2^q$ satisfies the following:
\begin{itemize}
\item $g(\gamma_2^q) \leq f_I(1)+1$
\item $g_1(\gamma_2^q) \leq f_I(a)$
\item for any $A$ with $|A| = f_I(a)+1$ and any $\bar{x} \in \gamma_2^q \cap \tilde{A}, d_{\gamma_2^q, A}^{(m_a, 0)}(\bar{x}) > \epsilon_2^q$
\item for any $A$ with $|A| = f_I(a)$ and any $\bar{x} \in \gamma_2^q \cap \tilde{A}, d_{\gamma_2^q, A}^{(m_a, 1)}(\bar{x}) > \epsilon_2^q$
\item for any $A$ with $|A| = f_I(a)$ and any $\bar{x} \in \gamma_2^q \cap \tilde{A}, d(y_j, x_i) \in [0, \epsilon_1^q) \cup (\epsilon_2^q, \infty)$ for all $i \leq q$ and $j \in A$.
\end{itemize}
\end{itemize}

Suppose $\gamma$ is a closed $s$-chain in $\Mp(n_1, \ldots, n_{\ell_I+3})$ such that $g_1(\gamma) \leq f_I(a)$ and $g_0(\gamma) \leq f_I(1)+1$. Let $\gamma_2^0 = \gamma$. The first and second conditions are true by assumption. The fifth condition is vacuously true. To get the third condition, notice that for all $A$ with $|A| = f_I(a) +1$ and any $\bar{x} \in \gamma_2^q \cap \tilde{A}$, we have  $d_{\gamma_2^q, A}^{(m_a, 0)}(\bar{x}) > 0$. Because $\gamma$ is compact, there exists $\delta_1>0$ such that $d_{\gamma_2^q, A}^{(m_a, 0)}(\bar{x}) > \delta_1$. Similarly, there exists $\delta_2>0$ such that for any $A$ with $|A| = f_I(a)$ and any $\bar{x} \in \gamma_2^q \cap \tilde{A}, d_{\gamma_2^q, A}^{(m_a, 1)}(\bar{x}) > \delta_2$. Letting $\epsilon_2^0 = \min \{\delta_1, \delta_2\}$ and $\epsilon_1^0 = \epsilon_2^0/2$ gives the claim for $q = 0$.

Now, suppose $0 < q \leq n_1$ and the claim holds for all $\tilde{q} < q$. Consider the homotopy of $\gamma_2^{q-1}$ affecting only the $x_{q}$ coordinate, $\gamma_t = \gamma_2^{q-1} + v_q \cdot t$ where $v_q$ is a vector that is non-zero only in the $x_q$ coordinate. For large enough $t$, say $t=M$, the $x_q$ coordinate is always far away from all other points. Call the $(s+1)$-chain given by this homotopy $\Gamma$. $\Gamma$ may not be a chain in $\Mp(n_1, \ldots, n_{\ell_I+3})$. It may intersect forbidden subspaces of the forms:
\begin{center}
$x_q = y_{j_1} = \ldots = y_{j_{f_I(1)+1}}$\\
$x_q = z_j$ \\
$x_q = x_{i_2} = \ldots = x_{i_{u+1}}$\\
$x_q = x_{i_2} = \ldots = x_{i_{\alpha_m+1}} =$ $^mw_j$\\
$x_q = x_{i_2} = \ldots = x_{i_{b}} = y_{j_1} = \ldots = y_{j_c}$ where $b>a, 1\leq c \leq f_I(a)$\\
\end{center}

In the first case, remove a sufficiently small tubular neighborhood. The intersection of $\Gamma$ with its boundary is $N|_{z_{n_3+1} = \{ x_q, y_{j_1}, \ldots, y_{j_{f_I(1)+1}} \}}$ where $N \in H_*\Mp(n_1-1, n_2-(f_I(1)+1), n_3+1, \ldots, n_{\ell_I+3})$.

In the second case, again remove a sufficiently small tubular neighborhood. The intersection of $\Gamma$ with the boundary of this neighborhood is $N|_{z_{j} = [ x_q, z_j]}$ where $N \in H_*\Mp(n_1-1, n_2, n_3, \ldots, n_{\ell_I+3})$.

The third case only occurs if $(u, 0) \notin I$. If so, proceed as in the two preceding cases. Remove a small tubular neighborhood. The intersection of $\Gamma$ with the boundary of this neighborhood is $N|_{z_{n_3+1} = \{ x_q, \ldots, x_{i_{u+1}}\}}$ where $N \in H_*\Mp(n_1-(u+1), n_2, n_3+1, n_4, \ldots, n_{\ell_I+3})$.

For the fourth case, proceed as before. The intersection of $\Gamma$ with the boundary of this tubular neighborhood produces $N|_{z_{n_3+1} = \{ x_q, \ldots, ^mw_{j}\}}$ where $N \in H_*\Mp(n_1-(\alpha_m+1), n_2, n_3+1, n_4, \ldots, n_{m+3}-1 \ldots, n_{\ell_I+3})$.

The fifth case must be treated differently because various subspaces of this form are connected, making it impossible to find disjoint tubular neighborhoods. However, as before, we want to remove tubular neighborhoods of all of these, say of radius $r$. Let $\gamma_2^q$ be the intersection of $\Gamma$ with the boundary of the unions of these tubular neighborhoods. The radii of these tubular neighborhoods can be chosen sufficiently small so that we still have $g_0(\gamma_2^q) \leq f_I(1)+1$ and $g_1(\gamma_2^q) \leq f_I(a)$. Thus, conditions 1 and 2 still hold.
 
Let $A\subset \underline{n_2}$ with $|A| = f_I(a)+1$ and $\gamma_2^q \cap \tilde{A} \neq \emptyset$. Let $\bar{x} \in \gamma_2^q \cap \tilde{A}$. Then $\bar{x}$ comes from a point on $\Gamma$ with the $x$-coordinates and $y$-coordinates corresponding to a forbidden subspace perturbed slightly. We can choose $r$ sufficiently small so that $A \cap \{j_1, \ldots, j_c\} = \emptyset$. Let $j \in A$ and $\tilde{x}$ be the point on $\gamma_2^{q-1}$ corresponding to $\bar{x}$. If $\{i | d(y_j, x_i) \leq d^{(m_a, 0)}_{\gamma_2^{q-1}, A}(\tilde{x})\} \cap \{q, i_2, \ldots, i_b \} = \emptyset$, then $d^{(m_a, 0)}_{\gamma_2^{q}, A}(\bar{x}) = d^{(m_a, 0)}_{\gamma_2^{q-1}, A}(\tilde{x}) > \epsilon_2^{q-1}$. Thus, lets assume this is not the case. First note that each forbidden subspace that we removed tubular neighborhoods of in the fifth case involved at least $a$ $x$-coordinates and at least 1 $y$-coordinate. Thus, it comes from a point on $\gamma_2^{q-1}$ where at least $a-1$ $x$-coordinates and $1$ $y$-coordinate were equal. Also note that $a-1 \geq m_a$. Thus, at $\tilde{x}$, $d(y_j, x_{i_2}) > \epsilon_2^{q-1}$. The only coordinates that were changed were changed by at most $r$. Thus, $d_{\gamma_2^q, A}^{(m_a, 0)}(\bar{x}) > \epsilon_2^{q-1}-r$.

Now let $A\subset \underline{n_2}$ with $|A| = f_I(a)$ and $\gamma_2^q \cap \tilde{A} \neq \emptyset$. Let $\bar{x} \in \gamma_2^q \cap \tilde{A}$. Then $\bar{x}$ comes from a tubular neighborhood of a subspace where $x_q = x_{i_2} = \ldots = x_{i_{b}} = y_{j_1} = \ldots = y_{j_{c}}$. By choosing $r$ small enough, we can restrict to two cases: either $A$ is disjoint from $\{j_1, \ldots, j_c\}$ or $A = \{j_1, \ldots, j_c\}$. The first case follows in a very similar manner to the previous argument. In the second case, prior to the $x_q$-coordinate being equal to $y_{j_i}$, there were already $m$ $x$ coordinates equal to these $f_I(a)$ $y$ coordinates. Thus, the next closest $y$ coordinate had to be at least $\epsilon_2^{q-1}$ far away. Thus, in either case, $d_{\gamma_2^q, A}^{(m_a, 1)}(\bar{x}) > \epsilon_2^{q-1}-r$.

Again let $A\subset \underline{n_2}$ with $|A| = f_I(a)$ and $\gamma_2^q \cap \tilde{A} \neq \emptyset$. Let $\bar{x} \in \gamma_2^q \cap \tilde{A}$. As before, there are two cases: the $y$-coordinates in $A$ come from one of the forbidden subspaces or not. Suppose $j \in A$. In the first case, we have $d(y_j, x_q) \in [0, 2r)$. In the second case, we have $d(y_j, x_q) > \epsilon_2^{q-1}-r$. Because we're only changing coordinates other than $x_q$ by at most a distance of $r$ from $\gamma_2^{q-1}$, for all $i < q$, we have $d(y_j, x_i) \in [0, \epsilon_1^{q-1}+r) \cup (\epsilon_2^{q-1}-r, \infty)$. Thus, assuming $r$ has been chosen sufficiently small, for all $i \leq q$, we have $d(y_j, x_i) \in [0, \epsilon_1^{q-1}+r) \cup (\epsilon_2^{q-1}-r, \infty)$.

For $t=M$, we have a class $N \cdot x_q$ where $N \in H_*\Mp(n_1-1, n_2, n_3, \ldots, n_{\ell_I+3})$.

Thus, $\Gamma$ with its intersection with the above tubular neighborhoods removed allows us to write $[\gamma] = [\gamma_1^{q}] + [\gamma_2^q]$ where $[\gamma_1^q]$ is organized. Assuming $r$ has been chosen sufficiently small, $[\gamma_2^q]$ satisfies the above conditions. This proves the claim.
\end{proof}


\begin{lemma}
\label{minor_lemma_2}
Let $(n_1, n_2) \in \N^2$ with $n_1 > 0$. Suppose organized classes span $\Mp(\tilde{n}_1, \ldots, \tilde{n}_{\ell_I+3})$ whenever $(\tilde{n}_1, \tilde{n}_2) < (n_1, n_2)$. Let $a\in N$ be such that $f_I(a+1) < f_I(a)-1$ and $f_I(a-1) = f_I(a)$. Let $\gamma$ be a closed $s$-chain in $\Mp (n_1, \ldots, n_{\ell_I+3})$ such that $g_0(\gamma) \leq f_I(1)+1$, $g_1(\gamma) \leq f_I(a)$, and there exist $0 < \epsilon_1 < \epsilon_2$ as in Lemma \ref{minor_lemma_1}. Then $[\gamma] = [\gamma_1] + [\gamma_2]$ where $[\gamma_1]$ is organized and $\gamma_2$ satisfies one of the following:
\begin{itemize}
\item $[\gamma_2]=0$
\item $g_0(\gamma_2) \leq f_I(1)+1$ and $g_1(\gamma_2) \leq f_I(a+1)+1$
\end{itemize}
\end{lemma}

\begin{proof}
If $n_2 = 0$, then $g_1(\gamma) \leq f_I(a+1)$ and the claim holds. Thus, suppose $n_2 > 0$. We will show by induction that for all $q \leq n_2$, $[\gamma] = [\gamma_1^q] + [\gamma_2^q]$ where $[\gamma_1^q]$ is organized and $\gamma_2^q$ satisfies one of the following:

\begin{itemize}
\item $[\gamma_2^q]=0$
\item there exists $0 < \epsilon_1^q < \epsilon_2^q$ such that $\gamma_2^q$ satisfies the following
\begin{itemize}
\item $g_0(\gamma_2^q) \leq f_I(1)+1$
\item $g_1(\gamma_2^q) \leq f_I(a)$
\item for any $A$ with $|A| = f_I(a)+1$ and any $\bar{x} \in \gamma_2^q \cap \tilde{A}, d_{\gamma_2^q, A}^{(m_a, 0)}(\bar{x}) > \epsilon_2^q$
\item for any $A$ with $|A| = f_I(a)$ and any $\bar{x} \in \gamma_2^q \cap \tilde{A}, d(y_j, x_i) \in [0, \epsilon_1^q) \cup (\epsilon_2^q, \infty)$ for all $i \leq q$ and $j \in A$.
\item for all $\tilde{q} \leq q$, there does not exist $i$ and distinct $j_2, \ldots, j_{f_I(a+1)+2}$ such that $\gamma_2^q \cap \{x_i = y_{\tilde{q}} = y_{j_2} = \ldots = y_{j_{f_I(a+1)+2}} \} \neq \emptyset$
\end{itemize}
\end{itemize}

For $q=0$, the claim is assumed. Thus, suppose $0 < q \leq n_2$ and the claim holds for all $\tilde{q} < q$. Consider the homotopy of $\gamma_2^{q-1}$ affecting only the $y_{q}$ coordinate, $\gamma_t = \gamma_2^{q-1} + v_q \cdot t$ where $v_q$ is a vector that is non-zero only in the $y_q$ coordinate. For large enough $t$, say $t=M$, the $y_q$ coordinate is always far away from all other points. Call the $(s+1)$-chain given by this homotopy $\Gamma$. $\Gamma$ may not be a chain in $\Mp(n_1, \ldots, n_{\ell_I+3})$. It may intersect forbidden subspaces of the forms:

\begin{center}
$y_q = y_{j_2} = \ldots = y_{j_{f_I(0)+1}}$ (if $f_I(0) = f_I(1)+1$)\\
$y_q = z_j$ \\
$y_q = $$^mw_j$\\
$y_q = x_{i_1} = \ldots = x_{i_{b}} = y_{j_2} = \ldots = y_{j_{f_I(a)+1}}$ where $m_a \leq b \leq a$\\
$y_q = x_{i_1} = \ldots = x_{i_{b}} = y_{j_2} = \ldots = y_{j_c}$ where $b\geq a+1, 1\leq c \leq f_I(a+1)+1$\\
\end{center}

In the first case, remove a sufficiently small tubular neighborhood. The intersection of $\Gamma$ with the boundary of this neighborhood is $N|_{z_{n_3+1} = \{ y_k, y_{j_2}, \ldots, y_{j_{f_I(0)+1}} \}}$ where $N \in H_*\Mp(n_1, n_2-(f_I(0)+1), n_3+1, \ldots, n_{\ell_I+3})$.

In the second case, again remove a sufficiently small tubular neighborhood. The intersection of $\Gamma$ with the boundary of this neighborhood is $N|_{z_{j} = [ y_k, z_j]}$ where $N \in H_*\Mp(n_1, n_2-1, n_3, \ldots, n_{\ell_I+3})$.

For the third case, proceed as before. This intersection of $\Gamma$ with the boundary of this tubular neighborhood produces a class $N|_{z_{n_3+1} = [ y_k, ^mw_{j}]}$ where $N \in H_*\Mp(n_1, n_2-1, n_3+1, n_4, \ldots, n_{m+3}-1, \ldots, n_{\ell_I+3})$.

For the fourth case, remove a small tubular neighborhood of the set $T=\{ y_q = y_{i_2} = \ldots = y_{i_{f_I(a)+1}}: m_a^{th}$ closest $x$ is $< \epsilon^{q-1}_2$ away $\}$. Clearly the intersection in question lives within this set. Also, by the third condition for $\gamma_2^{q-1}, T \cap \gamma_2^{q-1} = \emptyset$. Additionally, by the fourth condition for $\gamma_2^{q-1}, \Gamma \cap \partial T = \emptyset$. Thus, the intersection of $\gamma_2^{q-1}$ and the boundary of the tubular neighborhood of $T$ is a class $N|_{^iw_{n_{i+3}+1} = \{y_q, y_{j_2}, \ldots, y_{j_{f_I(a)+1}} \}}$ where $N \in H_*\Mp(n_1, n_2-(f_I(a+1)+1), n_3,\ldots, n_{i+3}+1, \ldots, n_{\ell_I+3})$. In the case that $a \neq 0$ and this $^iw_{n_{i+3}+1}$ is not in some $\{ x_{i_1}, \ldots, x_{i_{a+1}},$$^iw_{n_{i+3}+1}\}$, then this class is null homologous. Otherwise, this class is organized.

For the fifth case, remove tubular neighborhoods of all of these, say of radius $r$. Let $\gamma_2^q$ be the intersection of $\Gamma$ with the boundary of the unions of these tubular neighborhoods. The radii of these tubular neighborhoods can be made sufficiently small so $g_0(\gamma_2^q) \leq f_I(1)+1$ and $g_1(\gamma_2^q) \leq f_I(a)$. Since for this case we have $c \leq f_I(a+1) + 1 < f_I(a)$, we can choose $r$ small enough so that conditions three and four hold for $\epsilon_1^q = \epsilon_1^{q-1} + r$ and $\epsilon_2^q = \epsilon_2^{q-1} - r$. Furthermore, it can be made small enough so that there does not exist $i$ and distinct $j_2, \ldots, j_{f_I(a+1)+1}$ such that $\gamma_2^q \cap \{x_{i} = y_{q} = y_{j_2} = \ldots = y_{j_{f_I(a+1)+1}} \} \neq \emptyset$ and that this property still holds for all $\tilde{q} < q$.

For $t=M$, we have a class $N \cdot y_q$ where $N \in H_*\Mp(n_1, n_2-1, n_3, \ldots, n_{\ell_I+3})$.

Thus, $\Gamma$ with its intersection with the above tubular neighborhoods removed allows us to write $[\gamma] = [\gamma_1^q] + [\gamma_2^q]$ where $[\gamma_1]$ is organized and $[\gamma_2^q]$ satisfies the required conditions. Thus, the claim holds for all $q \leq n_2$. The lemma is the case where $q = n_2$.

\end{proof}


\begin{lemma}
\label{minor_lemma_3}
Let $(n_1, n_2) \in \N^2$ with $n_1 > 0$. Suppose organized classes span $\Mp(\tilde{n}_1, \ldots, \tilde{n}_{\ell_I+3})$ for all $(\tilde{n}_1, \tilde{n}_2) < (n_1, n_2)$. Let $a\in N$ be such that $f_I(a+1) < f_I(a)$ and $f_I(a-1) \neq f_I(a)$. Let $\gamma$ be a closed $s$-chain in $\Mp (n_1, \ldots, n_{\ell_I+3})$ such that $g_0(\gamma) \leq f_I(1)+1$ and $g_1(\gamma) \leq f_I(a)$. Then $[\gamma] = [\gamma_1] + [\gamma_2]$ where $[\gamma_1]$ is organized $\gamma_2$ satisfies one the following:
\begin{itemize}
\item $[\gamma_2]=0$
\item $g_0(\gamma_2) \leq f_I(1)+1$ and $g_1(\gamma_2) \leq f_I(a+1)+1$
\end{itemize}
\end{lemma}

\begin{proof}
As in the proof of the previous lemma, if $n_2=0$, the claim holds. Thus, suppose $n_2 > 0$. We will show by induction that for all $q \leq n_2$, $[\gamma] = [\gamma_1^q] + [\gamma_2^q]$ where $[\gamma_1^q]$ is organized and $\gamma_2^q$ satisfies one of the following:
\begin{itemize}
\item $[\gamma_2^q]=0$
\item $g_0(\gamma_2^q) \leq f_I(1)+1$, $g_1(\gamma_2^q) \leq f_I(a)$, and for all $\tilde{q} \leq q$, there does not exist $i$ and distinct $j_2, \ldots, j_{f_I(a+1)+2}$ such that $\gamma_2^q \cap \{x_i = y_{\tilde{q}} = y_{j_2} = \ldots = y_{j_{f_I(a+1)+2}} \} \neq \emptyset$
\end{itemize}

For $q=0$, the claim is assumed. Thus, suppose $0 < q \leq n_2$ and the claim holds for all $\tilde{q} < q$. Consider the homotopy of $\gamma_2^{q-1}$ affecting only the $y_{q}$ coordinate, $\gamma_t = \gamma_2^{q-1} + v_q \cdot t$ where $v_q$ is a vector that is non-zero only in the $y_q$ coordinate. For large enough $t$, say $t=M$, the $y_q$ coordinate is always far away from all other points. Call the $(s+1)$-chain given by this homotopy $\Gamma$. $\Gamma$ may not be a chain in $\Mp(n_1, \ldots, n_{\ell_I+3})$. It may intersect forbidden subspaces of the forms:

\begin{center}
$y_q = y_{j_2} = \ldots = y_{f_I(0)+1}$ (if $f_I(0) = f_I(1)+1$)\\
$y_q = z_j$ \\
$y_q = $$^mw_j$\\
$y_q = x_{i_1} = \ldots = x_{i_{a}} = y_{j_2} = \ldots = y_{j_{f_I(a)+1}}$\\
$y_q = x_{i_1} = \ldots = x_{i_{b}} = y_{j_2} = \ldots = y_{j_c}$ where $b\geq a+1, 1\leq c \leq f_I(a+1)+1$\\
\end{center}

In the first case, remove a sufficiently small tubular neighborhood. The intersection of $\Gamma$ with the boundary of this neighborhood is $N|_{z_{n_3+1} = \{ y_q, y_{j_2}, \ldots, y_{f_I(0)+1} \}}$ where $N \in H_*\Mp(n_1, n_2-(f_I(0)+1), n_3+1, \ldots, n_{\ell_I+3})$.

In the second case, again remove a sufficiently small tubular neighborhood. The intersection of $\Gamma$ with the boundary of this neighborhood is $N|_{z_{j} = [ y_q, z_j]}$ where $N \in H_*\Mp(n_1, n_2-1, n_3, \ldots, n_{\ell_I+3})$.

For the third case, proceed as before. The intersection of $\Gamma$ with the boundary of this tubular neighborhood produces $N|_{z_{n_3+1} = [ y_q, ^mw_{j}]}$ where $N \in H_*\Mp(n_1, n_2-1, n_3+1, n_4, \ldots, n_{m+3}-1, \ldots, n_{\ell_I+3})$.

For the fourth case, proceed as before. We get a class $N|_{z_{n_3+1} = \{ y_q, x_1, \ldots, y_{j_{f_I(a+1)+1}}\}}$ where $N \in H_*\Mp(n_1-a, n_2-f_I(a+1)-1, n_3+1, n_4, \ldots, n_{\ell_I+3})$.

For the fifth case, remove tubular neighborhoods of all of these, say of radius $r$. Let $\gamma_2^q$ be the intersection of $\Gamma$ with the boundary of the unions of these tubular neighborhoods. The radii of these tubular neighborhoods can be made sufficiently small so that $g_0(\gamma_2^q) \leq f_I(1)+1$ and $g_1(\gamma_2^q) \leq f_I(a)$. Furthermore, they can be chosen small enough so that there does not exist $i$ and distinct $j_2, \ldots, j_{f_I(a+1)+2}$ such that $\gamma_2^q \cap \{x_{i} = y_{q} = y_{j_2} = \ldots = y_{j_{f_I(a+1)+2}} \} \neq \emptyset$. They can also be chosen small enough to ensure this property still holds for all $\tilde{q} < q$.

For $t=M$, we have a class $N \cdot y_q$ where $N \in H_*\Mp(n_1, n_2-1, n_3, \ldots, n_{\ell_I+3})$.

Thus, $\Gamma$ with its intersection with the above tubular neighborhoods removed allows us to write $[\gamma] = [\gamma_1^q] + [\gamma_2^q]$ where $[\gamma_1^q]$ is organized and $[\gamma_2^q]$ satisfies the above conditions. Thus, the claim holds for all $q \leq n_2$. The lemma is the case where $q = n_2$.

\end{proof}


\begin{lemma}
\label{minor_lemma_4}
Let $(n_1, n_2) \in \N^2$ with $n_1 > 0$. Suppose organized classes span $\Mp(\tilde{n}_1, \ldots, \tilde{n}_{\ell_I+3})$ for all $(\tilde{n}_1, \tilde{n}_2) < (n_1, n_2)$. Let $a\in N$. Let $\gamma$ be a closed $s$-chain in $\Mp (n_1, \ldots, n_{\ell_I+3})$ such that $g_0(\gamma) \leq f_I(1)+1$ and $g_1(\gamma) \leq f_I(a+1)+1$. Then $[\gamma] = [\gamma_1] + [\gamma_2]$ where $ [\gamma_1]$ is organized and $\gamma_2$ satisfies one the following:
\begin{itemize}
\item $[\gamma_2]=0$
\item $g_0(\gamma_2) \leq f_I(1)+1$ and $g_1(\gamma_2) \leq f_I(a+1)$
\end{itemize}
\end{lemma}

\begin{proof}

If $n_2 = 0$, then $g_1(\gamma) \leq f_I(a+1)$ and the claim holds. Thus, suppose $n_2 > 0$. We will show by induction that for all $q \leq n_1$, $[\gamma] =[\gamma_1^q] + [\gamma_2^q]$ where $[\gamma_1^q]$ is organized and $\gamma_2^q$ satisfies one the following:
\begin{itemize}
\item $[\gamma_2^q]=0$
\item $g_0(\gamma_2^q) \leq f_I(1)+1$, $g_1(\gamma_2^q) \leq f_(a+1)+1$, and for all $\tilde{q} \leq q$, there does not exist distinct $j_1, \ldots, j_{f_I(a+1)+1}$ such that $\gamma_2^q \cap \{x_{\tilde{q}} = y_{j_1} = \ldots = y_{j_{f_I(a+1)+1}} \} \neq \emptyset$
\end{itemize}

For $q=0$, the claim is assumed. Thus, suppose $0 < q \leq n_1$ and the claim holds for all $\tilde{q} < q$. Consider the homotopy of $\gamma_2^{q-1}$ affecting only the $x_{q}$ coordinate, $\gamma_t = \gamma_2^{q-1} + v_q \cdot t$ where $v_q$ is a vector that is non-zero only in the $x_q$ coordinate. For large enough $t$, say $t=M$, the $x_q$ coordinate is always far away from all other points. Call the $(s+1)$-chain given by this homotopy $\Gamma$. $\Gamma$ may not be a chain in $\Mp(n_1, \ldots, n_{\ell_I+3})$. It may intersect forbidden subspaces of the forms:

\begin{center}
$x_q = y_{j_1} = \ldots = y_{j_{f_I(1)+1}}$\\
$x_q = z_j$ \\
$x_q = x_{i_2} = \ldots = x_{i_{u+1}}$\\
$x_q = x_{i_2} = \ldots = x_{i_{\alpha_m+1}} =$ $^mw_j$\\
$x_q = x_{i_2} = \ldots = x_{i_{a+1}} = y_{j_1} = \ldots = y_{j_{f_I(a+1)+1}}$\\
$x_q = x_{i_2} = \ldots = x_{i_{b}} = y_{j_1} = \ldots = y_{j_c}$ where $b>a+1, 1\leq c \leq f_I(a+1)$\\
\end{center}

In the first case, remove a sufficiently small tubular neighborhood. The intersection of $\Gamma$ with the boundary of this neighborhood is $N|_{z_{n_3+1} = \{ x_q, y_{j_1}, \ldots, y_{j_{f_I(1)+1}} \}}$ where $N \in H_*\Mp(n_1-1, n_2-(f_I(1)+1), n_3+1, \ldots, n_{\ell_I+3})$.

In the second case, again remove a sufficiently small tubular neighborhood. The intersection of $\Gamma$ with the boundary of this neighborhood is $N|_{z_{j} = [ x_q, z_j]}$ where $N \in H_*\Mp(n_1-1, n_2, n_3, \ldots, n_{\ell_I+3})$.

The third case only occurs if $(u, 0) \notin I$. If so, proceed as in the two preceding cases. Remove a small tubular neighborhood. The intersection of $\Gamma$ with the boundary of this neighborhood is $N_{z_{n_3+1} = \{ x_q, \ldots, x_{i_{u+1}}\}}$ where $N \in H_*\Mp(n_1-(u+1), n_2, n_3+1, n_4, \ldots, n_{\ell_I+3})$.

For the fourth case, proceed as before. The intersection of $\Gamma$ with the tubular neighborhood produces a class $N|_{z_{n_3+1} = \{ x_q, \ldots, ^mw_{j}\}}$ where $N \in H_*\Mp(n_1-(\alpha_m+1), n_2, n_3+1, n_4, \ldots, n_{m+3}-1, \ldots, n_{\ell_I+3})$.

For the fifth case, proceed as before. The intersection of $\Gamma$ with the tubular neighborhood produces a class $N|_{z_{n_3+1} = \{ x_q, \ldots, y_{j_{f_I(a+1)+1}}\}}$ where $N \in H_*\Mp(n_1-(a+1), n_2-(f_I(a+1)+1), n_3+1, n_4, \ldots, n_{\ell_I+3})$.

For the sixth case, remove tubular neighborhoods of all of these, say of radius $r$. Let $\gamma_2^q$ be the intersection of $\Gamma$ with the boundary of the unions of these tubular neighborhoods. The radii of these tubular neighborhoods can be chosen sufficiently small so that $g_0(\gamma_2^q) \leq f_I(1)+1$ and $g_1(\gamma_2^q) \leq f_I(a+1)+1$. Furthermore, they can be chosen small enough so that there does not exist distinct $j_1, \ldots, j_{f_I(a+1)+1}$ such that $\gamma_2^q \cap \{x_{q} = y_{j_1} = \ldots = y_{j_{f_I(a+1)+1}} \} \neq \emptyset$. They can also be small enough to ensure this property still holds for all $\tilde{q} < q$.

For $t=M$, we have a class $N \cdot x_q$ where $N \in H_*\Mp(n_1-1, n_2, n_3, \ldots, n_{\ell_I+3})$.

Thus, $\Gamma$ with its intersection with the above tubular neighborhoods removed allows us to write $[\gamma] =[\gamma_1^q] + [\gamma_2^q]$ where $[\gamma_1^q]$ is organized and $\gamma_2^q$ satisfies the above conditions. Thus, the claim holds for all $q \leq n_1$. The lemma is the case where $q = n_1$.
\end{proof}

We will now prove Lemma \ref{main_lemma}.

\begin{proof}
Let $(n_1, n_2) \in \N^2$ with $n_1 > 0$. Suppose organized classes span $\Mp(\tilde{n}_1, \ldots, \tilde{n}_{\ell_I+3})$ whenever $(\tilde{n}_1, \tilde{n}_2) < (n_1, n_2)$. Let $\gamma$ be a closed $s$-chain in $\Mp(n_1, \ldots, n_{\ell_I+3})$. Let $a=0$. If $n_2=0$, then $g_0(\gamma) = g_1(\gamma) = 0$ and the claim holds. Thus, suppose $n_2 > 0$. It is clear that $g_1(\gamma) \leq f_I(0)$. There are two cases. First, suppose $f_I(0) = f_I(1)$. Then $g_0(\gamma) < f_I(0)+1 = f_I(1)+1$, and the claim holds. Second, suppose $f_I(0) > f_I(1)$. We will prove by induction that for all $q \leq n_2$, we can write $[\gamma] = [\gamma_1^q] + [\gamma_2^q]$ where $ [\gamma_1^q]$ is organized and $\gamma_2^q$ satisfies one the following:
\begin{itemize}
\item $[\gamma_2^q]=0$
\item for all $\tilde{q} \leq q$, there does not exist distinct $j_2, \ldots, j_{f_I(1)+2}$ such that $\gamma_2^q \cap \{y_{\tilde{q}} = y_{j_2} = \ldots = y_{j_{f_I(1)+2}} \} \neq \emptyset$
\end{itemize}

For $q=0$, the claim is trivial. Thus, suppose $ 0 < q \leq n_2$ and the claim holds for all $\tilde{q} < q$. Consider the homotopy of $\gamma_2^{q-1}$ affecting only the $y_{q}$ coordinate, $\gamma_t = \gamma_2^{q-1} + v_q \cdot t$ where $v_q$ is a vector that is non-zero only in the $y_q$ coordinate. For large enough $t$, say $t=M$, the $y_q$ coordinate is always far away from all other points. Call the $(s+1)$-chain given by this homotopy $\Gamma$. $\Gamma$ may not be a chain in $\Mp(n_1, \ldots, n_{\ell_I+3})$. It may intersect forbidden subspaces of the forms:

\begin{center}
$y_q = y_{j_2} = \ldots = y_{j_{f_I(0)+1}}$\\
$y_q = z_j$ \\
$y_q = $$^mw_j$\\
$y_q = x_{j_1} = \ldots = x_{j_{b}} = y_{i_2} = \ldots = y_{i_c}$ where $b\geq 1, 1\leq c \leq f_I(1)+1$\\
\end{center}

In the first case, remove a sufficiently small tubular neighborhood. The intersection of $\Gamma$ with the boundary of this neighborhood is $N|_{z_{n_3+1} = \{ y_q, y_{j_2}, \ldots, y_{j_{f_I(0)+1}} \}}$ where $N \in H_*\Mp(n_1, n_2-(f_I(0)+1), n_3+1, \ldots, n_{\ell_I+3})$.

In the second case, again remove a sufficiently small tubular neighborhood. The intersection of $\Gamma$ with the boundary of this neighborhood is $N|_{z_{j}} = [ y_q, z_j]$ where $N \in H_*\Mp(n_1, n_2-1, n_3, \ldots, n_{\ell_I+3})$.

For the third case, proceed as before. The intersection of $\Gamma$ with the tubular neighborhoods produces a class $N|_{z_{n_3+1} = [ y_q, ^mw_{j}]}$ where $N \in H_*\Mp(n_1, n_2-1, n_3+1, n_4, \ldots, n_{m+3} - 1, \ldots, n_{\ell_I+3})$.

For the fourth case, remove tubular neighborhoods of all of these, say of radius $r$. We can choose $r$ arbitrarily small. Let $\gamma_2^q$ be the intersection of $\Gamma$ with the boundary of the unions of these tubular neighborhoods. The radii of these tubular neighborhoods can be chosen sufficiently small so that there does not exist distinct $j_2, \ldots, j_{f_I(1)+2}$ such that $\gamma_2^q \cap \{y_{q} = y_{j_2} = \ldots = y_{j_{f_I(1)+2}} \} \neq \emptyset$. They can also be chosen small enough to ensure this property still holds for all $\tilde{q} < q$.

For $t=M$, we have a class $N \cdot y_q$ where $N \in H_*\Mp(n_1, n_2-1, n_3, \ldots, n_{\ell_I+3})$.

Thus, $\Gamma$ with its intersection with the above tubular neighborhoods removed allows us to write $[\gamma] = [\gamma_1^q] + [\gamma_2^q]$ where $[\gamma_1^q]$ is organized and $[\gamma_2^q]$ satisfies the above conditions. Thus, the claim holds for all $q \leq n_2$. The $a=0$ case occurs when $q = n_2$.

Now we wish to show that if the claim holds for $a$, then the claim holds for $a+1$. There will be four cases.

\textbf{Case I}: $f_I(a+1) = f_I(a)$: This case is trivial.

\textbf{Case II}: ${f_I(a+1) = f_I(a)-1}$: Use Lemma \ref{minor_lemma_4}

\textbf{Case III}: ${f_I(a+1) < f_I(a)-1}$ and $f_I(a-1) \neq f_I(a)$: Use Lemma \ref{minor_lemma_3} followed by Lemma \ref{minor_lemma_4}.

\textbf{Case IV}: ${f_I(a+1) < f_I(a)-1}$ and $f_I(a-1) = f_I(a)$: Use Lemma \ref{minor_lemma_1} followed by Lemma \ref{minor_lemma_2} followed by Lemma \ref{minor_lemma_4}.

Thus, Lemma \ref{main_lemma} holds.
\end{proof}

\end{appendices}

\bibliographystyle{amsalpha}
\bibliography{biblio.bib}

\end{document}